\documentclass[reqno]{amsart}
\usepackage{amsmath}
\usepackage{amsthm}
\usepackage{amsfonts}
\usepackage{amssymb}
\usepackage{xcolor}
\usepackage{array}
\usepackage{comment}
\usepackage{upgreek}
\usepackage{mathrsfs}
\usepackage{graphicx}
\usepackage{hyperref}
\usepackage{enumerate}
\usepackage{enumitem}
\usepackage{tikz-cd}
\usepackage{array}
\usepackage{subcaption}
\usepackage[fontsize=11pt]{scrextend}
\usepackage{setspace}
\usepackage[utf8]{inputenc}
\usepackage{geometry}
\theoremstyle{definition}
\newtheorem{thm}{Theorem}[section]
\newtheorem{defn}[thm]{Definition}
\newtheorem{defn-prop}[thm]{Definition-Proposition}
\newtheorem{rem}[thm]{Remark}
\newtheorem{lem}[thm]{Lemma}

\newtheorem{exmp}[thm]{Example}
\newtheorem{cor}[thm]{Corollary}
\newtheorem{prop}[thm]{Proposition}
\newtheorem{notat}[thm]{Notation}
\newtheorem{notats}[thm]{Notations}
\newtheorem{conj}[thm]{Conjecture}
\theoremstyle{plain}
\newtheorem{theorem}[thm]{Theorem}
\DeclareMathOperator{\hmr}{\textit{HMR}}
\DeclareMathOperator{\thmr}{\text{\smaller{$\widetilde{\textit{HMR}}$}}}
\DeclareMathOperator{\hhfr}{\text{\smaller{$\widehat{\textit{HFR}}$}}}

\DeclareMathOperator{\tcmr}{\text{\smaller{$\widetilde{\textit{CMR}}$}}}
\DeclareMathOperator{\trsw}{\text{\smaller{$\widetilde{\textit{SWR}}$}}}

\DeclareMathOperator{\sfT}{\text{\smaller{\textsf{T}}}}
\DeclareMathOperator{\fix}{\text{Fix}}
\DeclareMathOperator{\re}{\text{Re}}
\DeclareMathOperator{\grad}{\text{grad}}

\DeclareMathOperator{\Hess}{\text{Hess}}
\DeclareMathOperator{\gr}{\textsf{gr}}
\DeclareMathOperator{\nbhd}{\textsf{nbhd}}
\DeclareMathOperator{\id}{\mathsf{id}}

\DeclareMathOperator{\dbcv}{\mathsf{\Sigma}}
\DeclareMathOperator{\ind}{\mathsf{ind}}
\DeclareMathOperator{\spinc}{{\text{spin\textsuperscript{c}}}}
\DeclareMathOperator{\rspinc}{{\textsf{RSpin\textsuperscript{c}}}}
\newcommand{\del}{\ensuremath{\partial}}

\newcommand{\rrsc}{\textsc{rrs}\textsuperscript{c}}
\newcommand{\rrscs}{{\textsc{rrs}\textsuperscript{c} structure}}

\DeclareMathOperator{\tfrr}{\widetilde{\text{\textsc{r}}}}
\DeclareMathOperator{\R}{\text{\textsc{r}}}
\DeclareMathOperator{\mR}{\text{\textsc{r}}}
\newcommand{\frr}{{\text{\textsc{r}}}}
\newcommand{\tcalB}{\tilde{\mathcal{B}}}

\newcommand{\frakb}{\mathfrak{b}}

\newcommand{\frakq}{\mathfrak{q}}

\newcommand{\sr}{\textsc{r}}
\newcommand{\irr}{\mathrm{irr}}
\newcommand{\red}{\mathrm{red}}

\newcommand{\con}{\textsf{conj}}
\makeatletter
\DeclareRobustCommand\widecheck[1]{{\mathpalette\@widecheck{#1}}}
\def\@widecheck#1#2{%
    \setbox\z@\hbox{\m@th$#1#2$}%
    \setbox\tw@\hbox{\m@th$#1%
       \widehat{%
          \vrule\@width\z@\@height\ht\z@
          \vrule\@height\z@\@width\wd\z@}$}%
    \dp\tw@-\ht\z@
    \@tempdima\ht\z@ \advance\@tempdima2\ht\tw@ \divide\@tempdima\thr@@
    \setbox\tw@\hbox{%
       \raise\@tempdima\hbox{\scalebox{1}[-1]{\lower\@tempdima\box
\tw@}}}%
    {\ooalign{\box\tw@ \cr \box\z@}}}
\makeatother
\newcommand{\wtilde}{\widetilde}

\newcommand\hcancel[2][black]{\setbox0=\hbox{$#2$}%
    \rlap{\raisebox{.45\ht0}{\textcolor{#1}{\rule{\wd0}{1pt}}}}#2}
\DeclareMathOperator{\pertL}{\hcancel{\mathcal{L}}}
\newcommand\blfootnote[1]{%
	\begingroup
	\renewcommand\thefootnote{}\footnote{#1}%
	\addtocounter{footnote}{-1}%
	\endgroup
}
\begin{document}
\setcounter{secnumdepth}{2}
\setcounter{tocdepth}{1}
\author{Jiakai Li}
\title{Multi-framed real monopole Floer theory}
\begin{abstract}
    This paper constructs a framed real monopole Floer homology for three-manifolds with involutions, marked with multiple basepoints.
    The relative gradings of these Floer homologies depend on the framing information and the paper gives a sufficient condition for the existence of relative mod two gradings.
    Assuming orientability and choices of orientations, this paper also proposes a definition of $\mathbf{Z}$-valued framed real Seiberg--Witten invariants for 4-manifolds with involutions, marked with circles.
\end{abstract}
\maketitle
\tableofcontents
\section{Introduction}
\blfootnote{This work was partially supported by a Simons Foundation Award \#994330 (Simons Collaboration on New Structures in Low-Dimensional Topology).}Recent years have seen the growing development of real Seiberg--Witten theory~\cite{TianWang2009, NNakamura2015,Kato2022,MR5002153,ljk2022,MR4821360,miyazawa2025satelliteformularealseibergwitten,baraglia2026exoticembeddedsurfacesinvolutions}.
This theory has proven to be surprisingly effective in the study of knots in 3-manifolds and surfaces in 4-manifolds.
Some notable applications include Miyazawa's proof of the existence of exotic $P^2$-knots in $S^4$~\cite{Miyazawa2023}, and Kang--Park--Taniguchi's proof of the non-sliceness of cables of the figure-eight~\cite{KPT2024,kang2025smoothconcordancecablesfigureeight}.

Unlike its ordinary counterpart, the framed\footnote{``Framed'' means counting without dividing by the constant gauge group.}
enumerative invariants from real Seiberg--Witten theory detect interesting information about manifolds with involution.
For example, Miyazawa \cite{Miyazawa2023} showed that the integral count of the framed real Seiberg--Witten invariant is nontrivial for certain branched double covers of $P^2$-knots.
By contrast, the ordinary framed Seiberg--Witten invariant, while well-defined, is always the integer $1$ (or $0$ if $b^+1(X) > 0$) regardless of the underlying 4-manifold.

Guth and Manolescu developed the real Heegaard Floer homology (\textit{HFR})~\cite{HFR2} for 3-manifold $(Y,\tau)$ with involution.
The \textit{HFR} package produces three $\mathbb{Z}/2$-equivariant Floer homology groups and, in addition, a hat-version $\hhfr$ which comes with a mod two grading.
The Euler characteristic of $\hhfr(Y,\tau)$ can be thought of as the Heegaard Floer analogue of Miyazawa's degree invariant when $Y$ is the double branched $\dbcv_2(S^3,L)$ of some link $L$,
equipped with the deck transformation $\tau_{\text{deck}}$.
This Euler characteristic can also be interpreted as the Euler characteristic of $\thmr(Y,\tau)$.
More recently, Srivastava~\cite{srivastava2026absolutemathbbz2gradingsreal} showed there exist absolute mod two gradings $\hat{\chi}$ in $\hhfr(\dbcv_2(S^3,L),\tau_{\text{deck}})$ and gave a combinatorial interpretation of $\hat{\chi}(\hhfr(\dbcv_2(S^3,L),\tau_{\text{deck}}))$ in terms of the Alexander polynomial evaluated at $\sqrt{-1}$.

The analogue of Srivastava's calculation in Seiberg--Witten theory would lead to interesting calculations of invariants of surfaces.
Nevertheless, the Seiberg--Witten story is less straightforward than the Heegaard Floer story for the following reasons.

First, despite the formal similarities between real Heegaard Floer and real monopole Floer theories, the current versions of these groups are not expected to be directly isomorphic.
The Guth--Manolescu construction uses multiple basepoints, and the absolute mod two grading in $\hhfr$ is well-defined only when every component of the link contains a basepoint.
On the other hand, all of the real monopole groups in \cite{ljk2022,ljk2024SSKh} are single-pointed.
For instance, the framed $\thmr$ is the mapping cone of the $\upsilon$-map based at a single point.

Second, the  monopole theory does not have a natural Lagrangian intersection interpretation in the same way as the Heegaard Floer theory.
This creates difficulties for sign assignments across different real $\spinc$ structures.
Indeed, the configuration spaces associated to different $\spinc$ structures are unrelated without further geometric choices.
In fact, what plays the role of a real $\spinc$ structure in the multi-framed setup is the notion of a relative real $\spinc$ structure, relative to codimension-3 data.

This paper resolves the first point by setting up the framework for multi-framed real monopole Floer homology denoted as $\thmr_*(Y,\tau,\mathbf{p})$ and the associated enumerative framed invariant for closed real 4-manifolds $\trsw(X,\tau_X,\mathbf{c})$.
As for the second point, this paper gives a sufficient condition for relative mod two gradings, similar to the Heegaard Floer case~\cite{HFR2}.

\medskip
Let $\mathbf{F}$ be the field of two elements.
The triple $(Y,\tau,\mathbf{p})$, where $\mathbf{p}$ is a finite set of basepoints, will be referred to as a real marked 3-manifold.
Cobordisms are 4-manifolds with involutions marked with 1-dimensional basepoints, and the closed cases are 4-manifolds $(X,\tau_X,\mathbf{c})$ with orientation-preserving involutions marked with a set of circles $\mathbf{c}$.
In summary:
\begin{theorem}
    Let $\textsc{rcob}_*$ denote the category of marked real cobordisms and $\textsc{vect}_{\mathbf{F}}$ denote the category of finite-dimensional $\mathbf{F}$-vector spaces.
    The framed real monopole Floer homology defines a functor:
    \[\thmr_* \colon \textsc{rcob}_* \to \textsc{vect}_{\mathbf{F}}.\]
\end{theorem}
With this understood, the sufficient condition for relative mod two grading can be stated as follows (see Proposition~\ref{lem:gamma_in_2H}).
\begin{theorem}
    \label{thm:mod-2_grading_intro}
    Suppose the fixed-point set $C = \textsf{Fix}(\tau) \subset Y$ is null-homologous in $Y/\tau$,
	and $Y \to Y'$ is specified by an orientable Seifert surface.
	Suppose also every component of $C$ is marked with at least one basepoint.
    Then $\thmr(Y,\tau,\mathbf{p})$ admits a relative mod two grading.
\end{theorem}
As observed in the single-framed case \cite{ljk2022,ljk2024SSKh}, for certain non-torsion $\spinc$ structures even a relative mod two grading is absent due to the $1/2$ factors in the dimension formulae involving real Dirac operators.

Theorem~\ref{thm:mod-2_grading_intro} alone does not pin down a total Euler characteristic $\tilde{\chi}$ like $\hat{\chi}$ in $\widehat{\textit{HFR}}$, as such a $\tilde{\chi}$ requires a choice of absolute mod two grading for each relative real $\spinc$ structure.
The problem of \emph{canonical} absolute mod two grading is closely related to orientations in the 4-dimensional enumarative invariants.
We expect the framing to play an essential role in questions of orientation, but defer this issue to future work.
\subsection{Relations with other invariants}
The construction of framed real monopole Floer homology can be seen as the direct analytic counterpart of the algebraic construction in \cite{ljk2024SSKh}.
Denote by $\thmr^{mc}(Y,\tau,p)$ the definition in~\cite{ljk2024SSKh} (``mc'' for mapping cone).
These two definitions are expected to coincide:
   $\thmr^{mc}(Y,\tau,p,\mathfrak{s},\R) \cong \thmr(Y,\tau,\{p\},\mathfrak{s},\tfrr)$.
\begin{rem}
    There should exist a $\mathbb{Z}/2$-equivariant construction of $\hmr^{\circ}(Y,\tau,\mathbf{p})$ in the spirit of \cite{ljk2022}, adapted to the multi-based setting.
    Taking iterated mapping cones should recover the framed definition $\thmr(Y,\mathbf{p})$.
\end{rem}
The following conjecture is the counterpart of Srivastava's theorem~\cite[Prop.~1.6]{srivastava2026absolutemathbbz2gradingsreal}.
\begin{conj}
	\label{conj:grading}
    Let $L$ be a link and $\mathbf{p}$ be a set of basepoints containing exactly one basepoint in each component of $L$.
    \emph{There exists a} mod two grading $\tilde{\chi}$ for which
    \begin{equation*}
        \wtilde{\chi}\left(\thmr(\mathsf{\Sigma}_2(S^3,L),\tau,\mathbf{p})\right) = 2^{|L|-1}\Delta_L(\sqrt{-1},\dots,\sqrt{-1}),
    \end{equation*}
    where $\Delta_L$ is the multivariate Alexander polynomial.
\end{conj}

A categorified version of the above conjecture is the following.
\begin{conj}
    $\thmr(L,\mathbf{p})$ satisfies an \emph{oriented} exact triangle: given an oriented skein triple $(L_+,L_-,L_0)$,
    \[\begin{tikzcd}
        {\thmr_*(L_+,\mathbf{p})} && {\thmr_*(L_-,\mathbf{p})} \\
        & {\thmr_*(L_0,\mathbf{p})}
        \arrow[from=1-1, to=1-3]
        \arrow[from=1-3, to=2-2]
        \arrow[from=2-2, to=1-1]
    \end{tikzcd},\]
    where the basepoints $\mathbf{p}$ are assumed to be away from the resolutions.
    (Compare with the \emph{unoriented} exact triangle~\cite{ljk2023triangle}.)
\end{conj}
Finally, it is reasonable to conjecture that the framed version of $\hmr$ matches the hat version of real Heegaard Floer homology, following the isomorphism in the ordinary cases (see e.g.~\cite{Taubes2010ECHSWF1,KutluhanLeeTaubes2010HFHM1,ColinGhigginiHonda2012HFECH1}).
\begin{conj}
    There is an isomorphism
    $\thmr_*(Y,\tau,\mathbf{p}) \cong \widehat{\textit{HFR}}_*(Y,\tau,\mathbf{p})$.
    This isomorphism is grading-preserving with the internal homological grading in $\thmr$ and the Maslov grading in $\widehat{\textit{HFR}}$.
    This isomorphism is grading-preserving with respect to the hat grading \textit{HFR} and the conjectural tilde-trading $\tilde{\chi}$ (Conj.~\ref{conj:grading}) in $\thmr$.
\end{conj}
\begin{rem}
    It is unclear how the $\{\wedge,\vee,-\}$ flavours of $\hmr$ and the $\{+,-,\infty\}$ flavours of \textit{HFR} are related.
    For instance, the multi-pointed equivariant $\{\pm 1\}$-equivariant $\hmr$ should not involve any curved chain complex, unlike the construction of Guth--Manolescu, where $\del^2 \ne 0$ and the chain complexes are curved.
\end{rem}
\subsection{Organization}
Section~\ref{sec:conceptual_overview} gives a conceptual picture of the framed homology and, in particular, describes the finite-dimensional Morse theory counterpart.
Section~\ref{sec:pointed_real_mflds} introduces the notion of multi-pointed real 3-manifolds.
Section~\ref{sec:framed_config_space} defines the framed versions of Seiberg--Witten configuration spaces, and Section~\ref{sec:analysis_framed} adapts the analysis of non-blown-up trajectories to the framed setting.
Section~\ref{sec:Floer_homology} and~\ref{sec:functor} define  framed real monopole Floer homology and establishes its functoriality.
Section~\ref{sec:grading} discusses some aspects of the grading in the framed theory.
\subsubsection{Acknowledgement}
The author would like to thank Deeparaj Bhat, Gary Guth, Ciprian Manolescu, Eha Srivastava, and Hongjian Yang for many helpful discussions.
\section{The finite-dimensional model}
\label{sec:conceptual_overview}\label{sec:fd_model}
In short, \emph{the framed real monopole Floer homology is the infinite-dimensional analogue of the Morse homology of a $(\mathbf{Z}/2)^k$-branched covering of the real monopole configuration space}.
For the reader's convenience, Figure~\ref{fig:tab_morse} compares the notations in the Morse theory setup with its Seiberg--Witten counterpart.

Let $H = (\mathbf{Z}/2)^k$ and $C_2 \le H$ be a distinguished subgroup of order-2.
Suppose $\tilde{B}$ is a smooth finite-dimensional manifold on which $H$ acts smoothly.
Assume the $H$-action has exactly one type of stabilizer $\{\pm 1\} \cong C_2 \le H$, that fixes a connected submanifold $\tilde{B}^{\red} \subset B$.
Let $H' = H/C_2 \cong (\mathbf{Z}/2)^{k-1}$ be the quotient group.

Denote $B = \tilde{B}/H$ and $B^{\red} = \tilde{B}^{\red}/H = \tilde{B}^{\red}/H'$.
In particular, $\tilde{B}^{\red} \to B^{\red}$ is a (unbranched) covering having deck transformation $H'$.
A complement $K$ of the $C_2$ subgroup of $H$ gives rise to an intermediate double branched cover $B^o = \tilde{B}/K$ of $B$, having $C_2$ as its deck transformation group.
THe cover $B^o$ fits in the sequence $\wtilde{B} \to B^{o} \to B$ of coverings, where the first arrow is an unbranched $(\mathbf{Z}/2)^{k-1}$ covering and the second arrow is a branched double covering.

Suppose $f \colon \tilde{B} \to \mathbf{R}$ is a Morse--Smale function and consider its Morse chain complex $(\tilde{C}_*,\tilde{\del})$.
To simplify the discussion, assume $f$ in this model is $H$-invariant\footnote{In the Floer settin below,
the perturbations of the Chern--Simons--Dirac functional on the framed configuration space need not be $H$-invariant.}
Let $\mathfrak{C}$  be the set of critical points of $f$,
and $\mathfrak{C}^{\red}, \mathfrak{C}^{\irr}$ be the subsets of $\mathfrak{C}$ consisting of critical points in $\tilde{B}^{\red}$ and $\tilde{B} \setminus \tilde{B}^{\red}$, respectively.
Then the critical points come in discrete $H$- or $H'$-orbits:
\[
    \mathfrak{C}^{\irr} = \bigcup_{\mathfrak{a}} H\mathfrak{a},\quad
    \mathfrak{C}^{\red} = \bigcup_{\mathfrak{a}'} H'\mathfrak{a}',
\]
where the $\mathfrak{a}$ and $\mathfrak{a}'$ are representatives of the orbits.
Furthermore, $H$ acts naturally on the Morse chain complex $(\tilde{C}_*,\tilde{\del})$.

\begin{exmp}[$S^1$]
	\label{exmp:morse_S1}
	This example mimics the calculation of the framed Floer homology of $S^1 \times S^2$ as the branched double cover of the 2-component unlink.
	Let $B=B^{\red}=S^1$ and $H=(\mathbf{Z}/2)^2$.
	Take the connected double cover $\pi\colon\tilde{B}\to B$.
	The distinguished subgroup $C_2\le H$ acts trivially on $\tilde{B}$,
	while $H'=H/C_2$ acts by deck transformations.
	Let $f'\colon B\to\mathbf{R}$ be a Morse function with one index-1 critical point $\mathfrak{a}_1$ and one index-0 critical point $\mathfrak{a}_0$.
	There are two trajectories from $\mathfrak{a}_1$ to $\mathfrak{a}_0$,
	so the differential downstairs vanishes over $\mathbf{F}$.

	With the pullback metric,
	$f=\pi^*f'$ is Morse--Smale and has four critical points $\mathfrak{a}_1^+,\mathfrak{a}_1^-,\mathfrak{a}_0^+,\mathfrak{a}_0^-$.
	The two descending trajectories from either maximum end at different lifts of the minimum.
	Thus $\tilde{\del}\mathfrak{a}_0^\pm=0$.
	\[
	\tilde{\del}\mathfrak{a}_1^+
	=\tilde{\del}\mathfrak{a}_1^-
	=\mathfrak{a}_0^+ + \mathfrak{a}_0^-,
	\]
	It follows that
	\[
	H_1^{\text{Morse}}(\tilde{B};\mathbf{F})
	=\mathbf{F}[\mathfrak{a}_1^+ + \mathfrak{a}_1^-],
	\qquad
	H_0^{\text{Morse}}(\tilde{B};\mathbf{F})
	=\mathbf{F}[\mathfrak{a}_0^+],
	\]
	where $[\mathfrak{a}_0^+]=[\mathfrak{a}_0^-]$.
	Thus the homology generators are independent of the choice
	of $\pm$ lifts.
\end{exmp}

\begin{exmp}[$\mathbf{T}^t$]
    Suppose $0 \le r \le t$.
    This example models the case of a pointed real $3$-manifold that admits an invariant metric of positive scalar curvature.
    The prototype is the double branched cover $\#^t(S^1 \times S^2)$ of the $(t+1)$-component unlink $U_{t+1}$, and there are $(r+1)$ basepoints on $(r+1)$ components of $U_{t+1}$, cf~\S~\ref{sec:pointed_real_mflds}.
    Let $B = B^{\red} = \mathbf{T}^t \cong (\mathbf{R}/\mathbf{Z})^t$;
    the covering space $\pi \colon \wtilde{B} \to B$ is given by doubling the first $r$ coordinates:
    \[\wtilde{B} = (\mathbf{R}/2\mathbf{Z})^r \times (\mathbf{R}/\mathbf{Z})^{t-r}.\]
    Let $f$ be a Morse--Smale function on $B$ and $\tilde{f} = \pi^*f$.
    Since $\tilde{B}$ and $B$ are tori of the same rank, abstractly
    \[
    H_*^{\text{Morse}}(\tilde{B};\mathbf{F}) \cong H_*^{\text{Morse}}(B;\mathbf{F}).\]
    This isomorphism is induced by $
    (x_1,\dots,x_t)
    \mapsto (2x_1,\dots,2x_r,x_{r+1},\dots,x_t)$.
\end{exmp}
\begin{exmp}[$S^1 \sqcup S^1 \sqcup S^1$]
	This is a hypothetical example of one reducible circle and two irreducible circles upstairs.
	Let $C_2=\langle\sigma\rangle$ and $H=(\mathbf{Z}/2)^2=\langle\sigma,\delta\rangle$,
	so that $H' \cong \mathbf{Z}/2$.
	Denote circles by
	\[
	B=S^1_a\sqcup S^1_b,
	\qquad
	\tilde{B}
	=\tilde{S}^1_a\sqcup\tilde{S}^1_{b+}\sqcup\tilde{S}^1_{b-}.
	\]
	Parametrize each circle downstairs by $\mathbf{R}/\mathbf{Z}$ and each circle upstairs by $\mathbf{R}/2\mathbf{Z}$.
	Let $\delta$ act by $x\mapsto x+1$ on every component.
	Let $\sigma$ fix $\tilde{S}^1_a$ pointwise and interchange
	$\tilde{S}^1_{b+}$ and $\tilde{S}^1_{b-}$.
	Then $\pi \colon \tilde{B} \to \tilde{B}/H=B$ is given by reduction modulo $\mathbf{Z}$.
	The stabilizer is $C_2$ on $\tilde{S}^1_a$ and is trivial on the two $b$-circles.
	Thus $\pi$ is two-to-one over $S^1_a$ and four-to-one over $S^1_b$.
	Each circle upstairs is a connected double cover of its image.

	Choose a Morse function $f'\colon B\to\mathbf{R}$ with one index-1 critical point and one index-0 critical point on each component,
	labelled $\mathfrak{a}_1,\mathfrak{a}_0$ and $\mathfrak{b}_1,\mathfrak{b}_0$, respectively.
	The pullback $f=\pi^*f'$ is an $H$-invariant Morse--Smale function with twelve critical points:
	\[
	\mathfrak{a}_i^\pm,
	\mathfrak{b}_{i,\varepsilon}^\pm,
	\qquad i=0,1,\ \varepsilon\in\{+,-\}.
	\]
	Here the subscript $\varepsilon$ specifies the component $\tilde{S}^1_{b\varepsilon}$.
	The nonzero differentials are
	\begin{equation*}
		\tilde{\del}\mathfrak{a}_1^+
		=\tilde{\del}\mathfrak{a}_1^-
		=\mathfrak{a}_0^+ + \mathfrak{a}_0^-,\qquad
		\tilde{\del}\mathfrak{b}_{1,+}^+
		=\tilde{\del}\mathfrak{b}_{1,+}^-
		=\mathfrak{b}_{0,+}^+
		+\mathfrak{b}_{0,+}^-,\qquad
		\tilde{\del}\mathfrak{b}_{1,-}^+
		=\tilde{\del}\mathfrak{b}_{1,-}^-
		=\mathfrak{b}_{0,-}^+
		+\mathfrak{b}_{0,-}^-,
	\end{equation*}
Therefore,
	\begin{align*}
		H_1^{\text{Morse}}(\tilde{B};\mathbf{F})
		&=\mathbf{F}[\mathfrak{a}_1^+ + \mathfrak{a}_1^-]
		\oplus
		\mathbf{F}[\mathfrak{b}_{1,+}^+
		+\mathfrak{b}_{1,+}^-] \oplus \mathbf{F}[\mathfrak{b}_{1,-}^+
		+\mathfrak{b}_{1,-}^-],\\
		H_0^{\text{Morse}}(\tilde{B};\mathbf{F})
		&=\mathbf{F}[\mathfrak{a}_0^+]
		\oplus \mathbf{F}[\mathfrak{b}_{0,+}^+] \oplus
		\mathbf{F}[\mathfrak{b}_{0,-}^+].
	\end{align*}
	The total homology has rank six, whereas $H_*^{\text{Morse}}(B;\mathbf{F})$ has rank four.
\end{exmp}
\begin{figure}[htb]
	\centering
	\renewcommand{\arraystretch}{2}
	\begin{tabular}{|c||c|c|c|c|c|c|c|}
		\hline
		Morse
		& $B$ & $B^o$ & $\wtilde{B}$ & $H$ & $H'=H/C_2$ & $C_2$ & $f$ \\
		\hline
		Floer
		& $\mathcal{B}$ & $\mathcal{B}^o$ & $\wtilde{\mathcal{B}}_{\mathbf{p}}$ & $\dfrac{\mathcal{G}}{\mathcal{G}_{\mathbf{p}}}$
		& $\dfrac{\mathcal{G}} {\{\pm1\}\mathcal{G}_{\mathbf{p}}}$ & $\{\pm1\}$ & $\pertL$ \\
		\hline
	\end{tabular}
	\caption{Correspondence between the Morse and monopole
		theories.
		Here $\{\pm1\}\mathcal{G}_{\mathbf{p}}$ is the subgroup
		generated by the constant gauge transformations $\{\pm1\}$ and the framed gauge group $\mathcal{G}_{\mathbf{p}}$.}
	\label{fig:tab_morse}
\end{figure}
\section{Marked real manifolds}
\label{sec:pointed_real_mflds}
\subsection{The category of marked real manifolds}
A \emph{real manifold} is an oriented 3- or 4-manifold with an orientation-preserving involution having a nonempty codimension-$2$ fixed point set.
The notion of ``marked'' manifolds keeps track of codimension-$3$ submanifolds of the fixed point sets as follows:
\begin{itemize}[leftmargin=*]
    \item A \emph{marked real 3-manifold} $\mathbb{Y}$ is a triple $(Y,\tau,\mathbf{p})$ where $(Y,\tau)$ is a real 3-manifold $(Y,\tau)$ and $\mathbf{p}$ is a collection of $\tau$-fixed \emph{marked} points, for which every component of $Y$ contains a marked point.

    \item A \emph{marked real 4-manifold} $\mathbb{X}$ is a triple $(X,\tau_X,\mathbf{c})$ where $(X,\tau_X)$ is a real 4-manifold $(X,\tau_X)$, possibly with boundary, and $\mathbf{c}$ is a nonempty 1-manifold that is fixed pointwise by $\tau_X$.\footnote{
         A 4-manifold ``marked with a 1-dimensional basepoint set'' is a slight misnomer.
         However, the adjective ``framed'' is typical reserved for a manifold with a trivialization of its tangent bundle.}
	For disconnected $X$, assume every connected component of $X$ contains marked points.

    \item A \emph{marked real cobordism} $\mathbb{W} \colon \mathbb{Y}_- \to \mathbb{Y}_+$ is a triple $(W,\tau_W,\mathbf{a})\colon (Y_-,\tau_-,\mathbf{p}_-) \to (Y_+,\tau_+,\mathbf{p}_+)$  where $(W,\tau_W)$ is a marked real 4-manifold whose boundary is $-(Y_-,\tau_-) \sqcup (Y_+,\tau_+)$, and $\mathbf{a}$ is a 1-manifold whose boundary is precisely $-\mathbf{p}_- \cup \mathbf{p}_+$.
    The boundary manifolds $Y_{\pm}$ need not be connected.

    \item A \emph{marked real manifold} $(M,\tau,\mathbf{c})$ is either a real 3-manifold or 4-manifold.
\end{itemize}

\begin{notats}
    As in the above definitions, a triple of the form $(M,\tau,\mathbf{c})$ denoting a marked real 3- or 4-manifold maybe be denoted as the blackboard-bold $\mathbb{M}$.
    This convention will be heavily used in Section~\ref{sec:functor}.
\end{notats}

\begin{defn}
    The \emph{category of marked real 3-manifolds} $\textsc{rcob}_{*}$ has objects marked real 3-manifolds and morphisms given by marked real cobordisms.
\end{defn}
Over a 0-dimensional basepoint set $\mathbf{c}$, let $\underline{\mathbf{C}}$ be the trivial line bundle and $\con\colon \underline{\mathbf{C}} \to \underline{\mathbf{C}}$ be the fibre-wise complex conjugation.
Over a 1-dimensional, denote the same bundle map as $\con_0\colon \underline{\mathbf{C}} \to \underline{\mathbf{C}}$;
the real subbundle is the trivial line bundle $\underline{\mathbf{R}}$.
On each circle, fix a choice $\con_1 \colon \underline{\mathbf{C}} \to \underline{\mathbf{C}}$ of a real structure having the M\"{o}bius line bundle $\gamma$ as its real subbundle.
\begin{defn}
    \label{defn:rrs_on_line_bundles}
    A \emph{relative real hermitian line bundle} over a marked real manifold $(M,\tau,\mathbf{c})$ is a triple $(L,\sfT,\frr)$ consisting of a complex hermitian line bundle $L$ over $M$, an trivialization
    $\sfT\colon L|_{\mathbf{c}} \to \underline{\mathbf{C}}$, and a conjugate-linear involution $\R \colon L \to L$ that lifts $\tau: M \to M$ and $\sfT \circ \R = \con \circ \sfT$ when $M$ is a 3-manifold, or  $\sfT \circ \R = \con_i \circ \sfT$, for $i = 0$ or $1$ when $M$ is a 4-manifold.
\end{defn}
\begin{rem}
    While every line bundle on a 1-manifold is trivializable, there are two inequivalent real structures, depending on whether the real subbundle is the trivial $\mathbf{R}$-bundle or the Mobius line bundle.
    Definition~\ref{defn:rrs_on_line_bundles} does not exclude the latter case.
    In particular, every line bundle that has real structure also admits a relative real structure.
\end{rem}

\begin{lem}
	\label{lem:rrsc_on_line_bundles}
	Fix a relative hermitian line bundle $(L,\sfT)$ and a prescribed model on each component of $\mathbf c$.
	Assume at least one real structure exists.
	The set of real structures, modulo unitary automorphism preserving $\sfT$, is a torsor over
\[
\frac{H^1(M,\mathbf c;\mathbf Z)^{\tau^*}}
     {\operatorname{im}(1+\tau^*)},
\]
where the denominator is the image of $(1+\tau^*) \colon H^1(M,\mathbf{c}) \to H^1(M,\mathbf{c})^{\tau^*}$ mapping $\omega$ to $\omega + \tau^*\omega$.
Moreover, the set of relative real structures equivalent as absolute real structure is a torsor over the kernel of restriction map:
\begin{equation*}
    \text{ker} \left(\frac{H^1(M,\mathbf{c};\mathbf{Z})^{\tau^*}}{\text{im}(1+\tau^*)} \to \frac{H^1(M;\mathbf{Z})^{\tau^*}}{\text{im}(1+\tau^*)} \right).
\end{equation*}
\end{lem}

\begin{proof}
Let a real structure $\R_0$ be given.
Then any other real structure has the form $\R=g\R_0$,
where $g \colon M\to S^1$ satisfies $g\circ\tau=g$ and $g|_{\mathbf c}=1$.
Any unitary automorphism $u \colon M\to S^1$ preserving $\sfT$ satisfies $u|_{\mathbf c}=1$,
and changes $g$ to $u\tau^*u g$.
By the classification argument of \cite[\S 3.1]{ljk2022} gives the stated quotient of relative cohomology.
\end{proof}
\subsection{Relative real spin\textsuperscript{c} structures}
\label{sec:rrsc}
Let $(M,\tau)$ be a real 3- or 4-manifold.
Recall~\cite[Defn.~3.7]{ljk2022}:
\begin{defn}
    An \emph{(absolute) real $\spinc$ structure} over $(M,\tau)$ is a pair $(\mathfrak{s},\R)$ of a $\spinc$ structure $\mathfrak{s}$ given by $(S,\rho)$ and a real structure $\R$ on the spinor bundle $S$ covering $\tau$, compatible with the Clifford multiplication, that is,
    \[
    \R(\rho(\xi) \Phi) = \rho(\tau_*\xi)\R(\Phi),\]
    for every $\xi \in \Gamma(TM)$ and $\Phi \in \Gamma(S)$.
    Two real structures $\R_1$ and $\R_2$ on a common underlying spin\textsuperscript{c} structure $(S,\rho)$ are \emph{equivalent} if there exists a $\spinc$ automorphism $u$ of $S$ for which $u \circ \R_1 =\R_2 \circ u$.
\end{defn}

A relative real $\spinc$ structure (to be defined below soon) records an additional choice of a trivialization of the spin\textsuperscript{c} vector bundle on the set of basepoints, together with a real structure that are of standard forms in the trivializations, described as follows.

Let $(\mathbf{C}^2,\rho_3)$ be the Clifford module of $(\mathbf{R}^3,\langle \cdot,\cdot\rangle)$ such that the standard basis $\{e_1,e_2,e_3\}$ acts by
Pauli matrices:
\[
\rho_3(e_1) = \begin{bmatrix}+i & 0\\0 & -i\end{bmatrix}, \quad
\rho_3(e_2)=\begin{bmatrix}0 & -1\\+1 & 0\end{bmatrix}, \quad
\rho_3(e_3) = \begin{bmatrix} 0 & +i \\ +i & 0\end{bmatrix}.
\]
Consider the involution $\tau_3(x_1,x_2,x_3)= (-x_1,+x_2,-x_3)$ on $\mathbf{R}^3$.
The coordinate-wise complex conjugation $\textsf{conj}$ on $\mathbf{C}^2$ is compatible with the Clifford module structure.
Indeed, for all $x \in \mathbf{R}^3$ and $\Phi \in \mathbf{C}^2$:
\[ \mathsf{conj}(\rho_3(x)\Phi) = \rho_3(\tau_3(x))\mathsf{conj}(\Phi).\]
As for the models over basepoints in 4-manifolds, let $(\mathbf{C}^4,\rho_{4})$ be the Clifford module of $(\mathbf{R}^4,\langle \cdot,\cdot \rangle)$ for which
\begin{equation*}
    \rho_{4}(e_0) = \begin{bmatrix}
        0 & -I_2\\
        I_2 & 0\\
    \end{bmatrix},
    \quad
    \rho_{4}(e_i) = \begin{bmatrix}
        0 & -\rho_3(e_i)^*\\
        \rho_3(e_i) & 0
    \end{bmatrix},
\end{equation*}
where $\{e_0,e_1,e_2,e_3\}$ is the standard basis,
$\mathbf{C}^4=\mathbf{C}^2_+ \oplus \mathbf{C}^2_-$,
and $\rho_3$ in the off-diagonal blocks is the 3-dimensional Clifford multiplication on $\mathbf{C}^2_{\pm}$ as above.
(The involution in $\mathbf{R}^4$ acts like $(x_0,x_1,x_2,x_3) \mapsto (+x_0,-x_1,+x_2,-x_3)$.)

Denote by $\underline{\mathbf{C}}^4$ be the trivial bundle with fibre $\mathbf{C}^4$ over a basepoint 1-manifold.
Let $\mathsf{conj}_0\colon \underline{\mathbf{C}}^4 \to \underline{\mathbf{C}}^4$ be the fibre-wise complex conjugation, for which the real subbundle is the trivial $\mathbf{R}^4$-bundle, decomposed as two trivial $\mathbf{R}^2$-bundles:
$\underline{\mathbf{R}}^2_+ \oplus \underline{\mathbf{R}}^2_-$.

Consider a closed component of the basepoint 1-manifold, conveniently viewed the second local model as over
\[\{[x:y:0] \ | \ (x,y) \in \mathbf{R}^2 \setminus \{0\}\} = \mathbf{RP}^1 \subset \mathbf{RP}^2 \subset \mathbf{CP}^2.\]
Consider the canonical spinor bundle
\[\Lambda^{0,*} = (\Lambda^{0,0} \oplus \Lambda^{0,2}) \oplus \Lambda^{0,1}\]
consisting of $(0,*)$-forms on $\mathbf{CP}^2$, restricted over $\mathbf{RP}^1$.
The tangent bundle $T\mathbf{CP}^2$ restricts to $\mathbf{RP}^1$ as a trivializable $\mathbf{R}^4$-bundle over $\mathbf{RP}^1$ and acts on $\Lambda^{0,*}$ by the standard Clifford multiplication $\rho_{can}$.
The complex conjugation on $\mathbf{CP}^2$ acts naturally on the forms $\Lambda^{0,*}$.
Denote the corresponding real structure on $\Lambda^{0,*}$ by $\textsf{conj}_1 \colon \Lambda^{0,*}  \to \Lambda^{0,*} $.
The positive spinor bundle $\Lambda^{0,0} \oplus \Lambda^{0,2}$ trivializable as a $\mathbf{C}$-vector bundle, but
\begin{equation*}
    (\Lambda^{0,2})^{\con_1}|_{\mathbf{RP}^1} = (\mathcal{O}(3)^{\con_1}|_{\mathbf{CP}^1})|_{\mathbf{RP}^1} \cong \gamma^3,
\end{equation*}
where $\gamma$ is the M\"{o}bius line bundle.
Therefore, $\langle w_1(\Lambda^{0,0} \oplus \Lambda^{0,2})^{\con_1},\mathbf{RP}^1\rangle \ne 0$.
The real subbundle of $\Lambda^{0,*}$ is isomorphic to $\underline{\mathbf{R}}_+ \oplus \gamma_+ \oplus \underline{\mathbf{R}}_- \oplus \gamma_-$,
where the subscripts indicate positive/negative spinors.
\begin{defn}
\label{defn:3d-rrsc}
Let $(Y,\tau,\mathbf{p})$ be a marked real 3-manifold.
A \emph{relative real $\spinc$ (\rrsc) structure} on $(Y,\tau,\mathbf{p})$ is a triple $(\mathfrak{s},\sfT,\mR)$ consisting of a $\spinc$ structure $\mathfrak{s}=(S,\rho)$ over $Y$, a trivialization $\sfT\colon S|_{\mathbf{p}} \to \underline{\mathbf{C}}^2$ over $\textbf{p}$ as a bundle of Clifford modules of $TY|_{\mathbf{p}}$, and a $\spinc$ compatible real structure $\frr$ over $(Y,\tau)$ for which
\[\sfT \circ \mR = \textsf{conj} \circ \sfT.\]

\end{defn}
\begin{defn}
\label{defn:4d-rrsc}
    Let $(X,\tau_X,\mathbf{c})$ be a marked real 4-manifold.
  A \emph{relative real $\spinc$ (\rrsc)  structure} on $(X,\tau_X,\mathbf{c})$ is a triple $(\mathfrak{s}_X,\sfT_X,\R_X)$ consisting of a $\spinc$ structure $\mathfrak{s}_X$, an isomorphism $\sfT_X \colon S|_{\mathbf{c}}\to \underline{\mathbf{C}}^4$ or $\Lambda^{0,*}$, as Clifford modules of $TX|_{\mathbf{c}}$, and a $\spinc$-compatible real structure $\frr$ such that over each component of $\mathbf{c}$ and :
    \[\sfT_X \circ \mR_X = \textsf{conj}_i \circ \sfT_X.	\]

\end{defn}
\begin{notats}
    Denote the set of real $\spinc$ structures by $\text{RSpin}^c(M,\tau)$, and,
    following~\cite{HFR2} and~\cite{OzSz2008multi}, denote the set of $\rrsc$ structures by $\underline{\text{RSpin}^c}(M,\tau,\mathbf{c})$.
\end{notats}
\begin{notats}
    A {\rrscs} will often be written as $(\mathfrak{s},\tilde{\mR})$ where $\tilde{\mR}$ denotes its underlying real structure but emphasizes that fact the notion of equivalence is in the framed sense.
    To further simplify notations, the underlying $\spinc$ structure $\mathfrak{s}$ may be suppressed and a $\rrsc$ structure will be written as $\tilde{\mR}$.
\end{notats}
\begin{rem}
    The data of the pairs $(\mathfrak{s},\sfT)$ and $(\mathfrak{s}_X,\sfT_X)$ in Definition~\ref{defn:3d-rrsc} and~\ref{defn:4d-rrsc} are \emph{relative $\spinc$} structures on $(Y,\mathbf{p})$ and $(X,\mathbf{c})$, respectively.
\end{rem}
\begin{rem}
    \label{rem:choice_of_trivilization_T}
    The definition of {\rrsc} structures has left some apparent ambiguity in the choices of trivializations.
    Indeed, post-composition with a locally constant function $\mathbf{c} \to \{\pm 1\}$ can yield inequivalent $\sfT$'s.
    This ambiguity is immaterial and a framing $\sfT$ will be always chosen once and for all.
\end{rem}
\subsection{First Stiefel-Whitney classes of real spin\textsuperscript{c} structures}
\subsubsection{Complex structure from a tangent framing}
\begin{lem}
    Let $(Y,\tau)$ be a real 3-manifold equipped with a real spin\textsuperscript{c} structure $(\mathfrak{s},\mR)$. The $w_1$ of the real subbundle $S^{\mR}$ is zero over every component $C_i$ of $C = \fix{\tau}$.
    In fact, an orientation of $C_i$ determines a complex structure on the $\mathbf{R}^2$-bundle $S^{\mR}|_{C_i}$.
\end{lem}
\begin{proof}
    Let $\hat{\text{t}}$ be a framing of the tangent bundle of $TC_k$; clearly $\tau_*(\hat{\text{t}}) = \hat{\text{t}}$.
    Then $\rho(\hat{\text{t}})^2=-1$ over $S|_{C_k}$ and commutes with $\mR$.
    This means precisely that $\rho(\hat{\text{t}})$ defines a complex structure of $S^{\mR}|_{C_k}$.
\end{proof}
\subsubsection{The first Stiefel--Whitney class from a normal framing}
Let $(Y,\tau)$ be a marked real 3-manifold, let $C$ be the fixed point set, and let $(\mathfrak{s},\mR)$ be a real spin\textsuperscript{c} structure.
Suppose $\hat{n}$ is a nonzero section of the unit normal bundle of $C$.
Then
\[S|_C \cong W^+ \oplus W^-\]
where $W^{\pm}$ is the $(\pm 1)$-eigenspace of $i\rho(\hat{n})$.
The involution $\mR$ commutes with $i\rho(\hat{n})$ and
$\rho(\hat{t})$ anti-commutes with  $i\rho(\hat{n})$.
Hence $\mR$ acts on each $W^{\pm}$ and $\rho(\hat{t})$ defines an complex-linear isomorphism $\rho(\hat{t})\colon W^{\pm} \to W^{\mp}$.
For every component $C_k$, the $\mR$-fixed subbundle $(W^+)^{\mR}$ is a $\mathbf{R}$-line bundle whose first Stiefel--Whitney class
\[w_1((W^+|_{C_k})^{\mR})=w_1((W^-|_{C_k})^{\mR})\]
gives rise to a $(\mathbf{Z}/2)^{r}$-valued invariant of $(\mathfrak{s},\mR)$ by evaluation over $[C_k]$, $1 \le k \le r$:
\[
w_1(\mR,\hat{n}) = (\langle w_1(W^+|_{C_1})^{\mR},C_1 \rangle, \dots,
\langle w_1(W^+|_{C_r})^{\mR} ,C_r\rangle ).
\]
Let $W^+_t$ be the eigen-line bundle obtained from $W^+$ by introducing a $(\pm 1)$-twist to $\hat{n}$ along $C_k$.
Then
\[
    \langle w_1(W^+_t|_{C_k})^{\mR}, C_k\rangle=
    \langle w_1(W^+|_{C_k})^{\mR} ,C_k\rangle + 1.
\]
\begin{lem}
    Assume $\hat{n}$ is a normal framing of an oriented nullhomologous surface $\Sigma$ that contains $C$, fixed set-wise by $\tau$, and $\Sigma/\tau$ is orientable.
    Then
    \[\sum_{i=1}^r \langle w_1(W^+|_{C_k})^{\mR},C_k \rangle = r \ (\text{mod 2}).\]
\end{lem}
\begin{proof}
    Since $\Sigma$ is nullhomologous, $\langle c_1(S), [\Sigma]\rangle = 0$ and thus $\langle c_1(W^{\pm}), [\Sigma]\rangle = \pm(g-1)$, where $g$ is the genus of $\Sigma$.
    If $h$ is the genus of $\Sigma/\tau$, then $g = 2h + r -1$.
    The equality $\langle w_1(L^{\mR}),C \rangle = \langle c_1(L), [\Sigma] \rangle $, for every real line bundle $L$, established in \cite[Appendix]{OkonekTeleman2013} gives:
    \[\sum_{k=1}^r \langle w_1(W^+|_{C_k})^{\mR},C_k \rangle = \langle c_1(W^+), [\Sigma]\rangle = g-1 = (2h + r - 1) -1 = r \ (\text{mod 2}). \qedhere\]
\end{proof}
\subsubsection{The first Stiefel--Whitney class of a {\rrscs} on a closed 4-manifold} \medskip
Given a marked real 4-manifold $(X,\tau_X,\mathbf{c})$ and a real $\spinc$ structure $(\mathfrak{s},\mR)$.
The restriction
$S^{\pm}|_{\mathbf{c}}$ is a \textsc{r}eal rank-$2$ vector bundle,
not necessarily trivial --- the invariant $\mathbf{R}^2$-subbundle over the basepoint circles is classified by its first Stiefel--Whitney class:
\[
 w_1((S^{\pm}|_{\mathbf{c}})^{\mR}).
\]
In particular, the model $(\mathbf{C}^4,\con_0)$ in Section~\ref{sec:rrsc} has trivial $w_1$ whereas the model $(\Lambda^{0,*},\con_1)$ has nontrivial $w_1$.

Enumerate the components of $\mathbf{c}$ as $c_1,\dots,c_n$; let $w_1(\mR,\mathbf{c})$ be the following algebraic-topological invariant:
\[
w_1(\mR,\mathbf{c}) = (\langle w_1(S^+|_{\mathbf{c}_1})^{\mR},\mathbf{c}_1 \rangle, \dots,
\langle w_1(S^+|_{\mathbf{c}_n})^{\mR} ,\mathbf{c}_n\rangle ).
\]
\subsection{The sets of relative real spin\textsuperscript{c} structures}
\begin{prop}
	\label{prop:rrsc_classify}
	Let $(M,\tau,\mathbf{c})$ be a marked real manifold. If $M$ is 4-dimensional,
	fix a model of $S^+|_{\mathbf{c}}$ along each component of $\mathbf c$.
	Fix a relative spin\textsuperscript{c} structure $(\mathfrak{s},\sfT)$ admitting at least one compatible relative real structure.
	The set of compatible relative real structures on $(\mathfrak{s},\sfT)$, modulo relative automorphisms, is a torsor over
    \begin{equation*}
        \frac{H^1(M,\mathbf{c};\mathbf{Z})^{\tau^*}}{\text{im}(1+\tau^*)},
    \end{equation*}
    Moreover, fixing the underlying absolute real structures $(\mathfrak{s},\mR)$ that occurs this way, the set of {\rrscs} is a torsor over:
    \begin{equation}
        \label{eq:ker_H1_rel_mod_im}
        \text{ker} \left(\frac{H^1(M,\mathbf{c};\mathbf{Z})^{\tau^*}}{\text{im}(1+\tau^*)} \to \frac{H^1(M;\mathbf{Z})^{\tau^*}}{\text{im}(1+\tau^*)} \right).
    \end{equation}
\end{prop}
\begin{proof}
    Apply Lemma~\ref{lem:rrsc_on_line_bundles} and \cite[Lem.~3.12]{ljk2022}.
    See also \cite[Cor.~3.27]{HFR2}.
\end{proof}
\begin{rem}
    The above kernel depends on the topology of $(M,\tau)$.
    In particular, while the quotient group
    \begin{equation*}
        \frac{H^0(\mathbf{c};\mathbf{Z})^{\tau^*}/\mathrm{im}(H^0(M; \mathbf{Z}))}{\text{im}(1+\tau^*)},
    \end{equation*}
    which is isomorphic to $(\mathbf{Z}/2)^{|\mathbf{c}|-b_0(M)}$, maps onto the group \eqref{eq:ker_H1_rel_mod_im}, this map may be not be injective.
\end{rem}
\begin{lem}
For a marked real 3-manifold $(Y,\tau,\mathbf{p})$, the set $\underline{\text{RSpin}^c}(Y,\tau,\mathbf{p})$ of {\rrscs}s  surjects onto the set $\text{RSpin}^c(Y,\tau)$ of real spin\textsuperscript{c} structures.
\end{lem}
\begin{proof}
    This is \cite[Lem.~3.31]{HFR2}.
    (The definition in \cite{HFR2} are relative to the boundary of punctures at the basepoints, which is equivalent to Definition~\ref{defn:3d-rrsc}.)
\end{proof}
Back to dimension four, recall the long exact sequence of cohomologies of the pair $(X,\mathbf{c})$:
\begin{equation*}
    \begin{tikzcd}
        {H^0(X)} & {H^0(\mathbf{c})} & {H^1(X,\mathbf{c})} & {H^1(X)} & {H^1(\mathbf{c})} & {H^2(X,\mathbf{c})} & \dots.
        \arrow[from=1-1, to=1-2]
        \arrow[from=1-2, to=1-3]
        \arrow[from=1-3, to=1-4]
        \arrow[from=1-4, to=1-5]
        \arrow[from=1-5, to=1-6]
        \arrow[from=1-6, to=1-7]
    \end{tikzcd}
\end{equation*}
It is no longer the case that $H^1(\mathbf{c}) = 0$, as $\mathbf{c}$ is a closed 1-manifold.
Hence the map $H^1(X,\mathbf{c}) \to H^1(X)$ needs not be surjective, and may remain so after taking $(\cdot)^{\tau^*}/\text{im}(1+\tau^*)$.
This suggests that
not every two absolute real structures can be related via an automorphism $u$ such that $u|_{\mathbf{c}} \equiv 1$.
For example, those with local model $\underline{\mathbf{C}}^4$ and with $\Lambda^{0,*}$.
Given two absolute real structure $(\mathfrak{s},\mR_1)$ and  $(\mathfrak{s},\mR_1)$ such that
\[w_1(\mR_1,\mathbf{c}) = w_1(\mR_2,\mathbf{c}),\]
the difference between $\mR_1$ and $\mR_1$
is an element of $H^1(M)^{\tau^*}$,
by the transversality argument in the proof of \cite[Lem.~3.31]{HFR2}.
\subsection{Examples}
\begin{exmp}[$S^1 \times S^2$, $\id \times \text{rot}, \{p_1,p_2\}$]
	Let the involution act on $S^1 \times S^2$ by rotating the $S^2$ around two points.
	There are in total four {\rrscs}s and two absolute real structures, all on the unique torsion $\spinc$ structure.
    Moreover, there are two inequivalent {\rrscs}s underlying the each absolute real structure.
\end{exmp}
\begin{exmp}[$\mathbf{RP}^3$, Hopf, $\{p_1,p_2\}$]
	View $\mathbf{RP}^3$ as the branched double cover of the Hopf link.
	There are four {\rrscs}s.
	Over $\mathbf{RP}^3$, there are two $\spinc$ structures $\mathfrak{s}_1, \mathfrak{s}_2$ arising from the two spin structures.
	Each $\mathfrak{s}_i$ admits two {\rrscs}s that are equivalent as absolute real structures.
\end{exmp}
\begin{exmp}[$S^1 \times S^2, U_2, \{p_1,p_2\}$]
	View $S^1 \times S^2$ as the branched double cover of the  2-component unlink.
	Consider the unique torsion $\spinc$ structure $\mathfrak{s}_0$.
	If $p_1$ and $p_2$ lie on different components, a cohomology calculation shows that there is a unique {\rrscs}.
\end{exmp}
\begin{exmp}[$\#^{n-1}(S^1 \times S^2), U_n, \{p_1,\dots,p_n\}$]
	Consider the branched double cover of the $n$-component unlink.
	Assume every component of the unlink receives precisely one basepoint.
    For the torsion $\spinc$ structure, there exists a unique {\rrsc} structure.
    The number of {\rrsc} structures changes if a component of $U_n$ is marked with multiple basepoints.
\end{exmp}
\begin{exmp}[Standard 4-ball]
	Let $B^4$ be the unit 4-ball in $\mathbf{R}^4 \cong \mathbf{C}^2$.
	The complex conjugation $(z_1,z_2) \mapsto (\overline{z}_1,\overline{z}_2)$ on $\mathbf{C}^2$ induces an involution on $B^4$, fixing a 2-disc $D^2$ in the plane of real coordinates.
	This involution will be referred to as the \emph{standard involution on $B^4$}, denoted as $\tau_{B^4}$.
	To mark $B^4$, let
	\[\mathbf{a}_{B^4}=\{(x_1,0):-1 \le x_1 \le +1\}.\]
	There is a unique $\rrsc$ structure $(\mathfrak{s}_{B^4},\tilde{\mR}_{B^4})$ on $(B^4,\tau_{B^4},\mathbf{a}_{B^4})$.
\end{exmp}
\begin{exmp}[Standard 3-sphere]
	The sphere $(S^3,\tau_{S^3})$ is viewed as the boundary of $(B^4,\tau_{B^4})$ as above.

	Let $q_{\pm} = (\pm 1,0)$.
	The marked real 3-manifold with a single basepoint $(S^3,\tau_{S^3},q_{\pm})$ has a unique $\rrsc$.
	However, $(S^3,\tau_{S^3},\{q_-,q_+\})$ admits two $\rrsc$'s, denoted as $(\mathfrak{s}_{S^3},\tilde{\mR}_+)$ and $(\mathfrak{s}_{S^3},\tilde{\mR}_-)$.
	As a convention, $(\mathfrak{s}_{S^3},\tilde{\mR}_-)$ is the boundary restriction of the $\rrsc$ on $(B^4,\tau_{B^4},\mathbf{a}_{B^4})$.
\end{exmp}
\begin{exmp}[The punctured cylinder]
	Let $(Y,\tau,\mathbf{p})$ be a marked real 3-manifold.
	Let $I \subset \mathbf{R}$ be an interval and 	$(Z,\tau_Z,\mathbf{p}_I)=(I \times Y,\id_I \times \tau,\mathbf{p}_I)$ where $\mathbf{p}_I=I \times \mathbf{p}$.
	Suppose $I = [-1,+1]$.
	Given any sub-collection of points $\mathbf{q} \subset \mathbf{p}$, remove from the cylinder $Z$ a $\tau$-invariant neighbourhood at $\{0\} \times \mathbf{q}$ to obtain the cobordism
	\[
	    Z_{\mathbf{q}} = (I \times Y) \setminus \bigcup_{q \in \mathbf{q}}(\nbhd(\{0\} \times q)).
	\]
	where each $\nbhd(\{0\} \times q)$ is identified with $(B^4,\tau_{B^4})$.
	Set $\mathbf{a}_{\mathbf{q}} = \mathbf{a} \setminus \nbhd(\{0\} \times q)$.
	Thus the boundary $\del( Z_{\mathbf{q}},\tau_Z, 	\mathbf{a}^{\mathbf{q}})$ comprises
	\[
    (-Y,\tau, \mathbf{p}) \sqcup \bigsqcup_{q \in \mathbf{q}}(-S^3,\tau_{S^3}, \{q_-,q_+\}) \sqcup (+Y,\tau, \mathbf{p})
	\]
	where $q_-,q_+$ are the boundary points of $\nbhd(\{0\} \times q) \cap \mathbf{a}$.
In other words, $Z_{\mathbf{q}}$ can be viewed as a cobordism with $|\mathbf{q}|$ additional incoming ends.
\end{exmp}
\begin{exmp}[Punctured cylinder with $|\mathbf{p}|=|\mathbf{q}|=1$]
    Assume there is precisely one basepoint $p$.
	Over the 3-manifold $Y$ there are as many $\rrsc$ structures as absolute real $\spinc$ structures, and such is true for $Z$.
	However, there are two $\rrsc$'s for each absolute real $\spinc$ structure over the punctured cobordism $Z^p$.
	Exactly one of the relative real structures extend over $Z$.
\end{exmp}
\begin{exmp}[Cobordism from homology actions]
    Let $\gamma \subset Y$ be an arc in $Y$ joining $p,q \in \mathbf{p}$ and $\tilde{\gamma} = \gamma \cup \tau(\gamma)$.
    Suppose $\gamma$ is chosen so that $\gamma \cap \tau(\gamma) = \{p,q\}$, and thus $\tilde{\gamma}$ is an embedded circle.
    Let
    \begin{equation*}
        Z_{\gamma} = (I \times Y) \setminus \left(\nbhd{\{0\} \times \tilde{\gamma}}\right)
    \end{equation*}
    $\mathbf{a}_{\gamma} = \mathbf{p}_I \cap Z_{\gamma}$.
    Note $\left(\nbhd{\{0\} \times \tilde{\gamma}}\right)$ is diffeomorphic to $S^1 \times D^3$ where $\tau$ acts as $\text{conj}_{S^1} \times \text{conj}_{D^2} \times \text{Id}$ where $S^1 \times D^2$ is viewed as the product of the unit circle and the unit disc in $\mathbf{C}^2$.
    The fixed point set is the disjoint union of two discs, and the basepoint set is the disjoint union of two arcs.
    Moreover,
    \begin{equation*}
        \del (Z_{\gamma}, \tau_{Z_{\gamma}}, \mathbf{a}_{\gamma}) \cong (S^1 \times S^2, \tau_{U_2}, \{p_-,p_+,q_-,q_+\})
    \end{equation*}
    where $\tau_{U_2}$ is the covering of transformation as the branched double cover of the 2-component unlink, with each component marked with two points.
    The pairs of points $p_{\pm},q_{\pm}$ can be thought of as the boundaries of a small neighbourhoods of $p,q$ on $\fix(\tau) \subset Y$.
\end{exmp}
\section{Framed Seiberg--Witten configuration spaces}
\label{sec:framed_config_space}
In this section, fix once and for all a marked real manifold $(M,\tau,\mathbf{c})$ and a {\rrsc} structure $(\mathfrak{s},\sfT, \frr)$, often written as $(\mathfrak{s},\tfrr)$ or simply $\tilde{\frr}$.
Let $W$ denote $S^+$ if $M$ is a 4-manifold, or $S$ if $M$ is a 3-manifold.
The first few definitions are concerned with the marked manifold $(M,\mathbf{c})$, without involutions.

Let $\underline{\mathcal{A}}(M,\mathfrak{s})$ be the space of $\spinc$ connections of $\mathfrak{s}$.
The \emph{ordinary framed Seiberg--Witten configuration space} is
\begin{equation*}
    \underline{\mathcal{C}}(M,\mathfrak{s}) =
    \underline{\mathcal{A}}(M,\mathfrak{s}) \times \Gamma(M,W).
\end{equation*}
The ``framing'' places no restriction on configurations, but limits the type of automorphisms allowed:
Define the \emph{ordinary framed gauge group} $\underline{\mathcal{G}}(M,\mathbf{c})$ as the group of maps $u$ from $M$ to $U(1)$ for which $u(\mathbf{c}) \equiv 1$.
Such an $u$ acts by
\begin{equation*}
    (A,\Phi) \mapsto (A - u^{-1}du, u\Phi).
\end{equation*}

A \emph{real spin\textsuperscript{c} connection} is a $\spinc$ connection that is invariant under pullback by $\frr$.
The space of real $\spinc$ connection will be denoted as $\mathcal{A}(M,\mathbf{c},\frr)$.
A \emph{real spinor} $\Phi$ is a $\tilde{\frr}$-invariant section of the spinor bundle $W$.
The \emph{framed real configuration space} $\mathcal{C}(M,\tau,\mathbf{c},\tilde{\mR})$ is the $\tilde{\mR}$-invariant subspace of the ordinary configuration space:
\[
\mathcal{C}(M,\tau,\mathbf{c},\tilde{\mR})=
\underline{\mathcal{C}}(M,\mathfrak{s})^{\tilde{\mR}}.
\]
In this viewpoint, there is a \textbf{canonical} identification of $\mathcal C(M,\tau,\frr)$ with  $\mathcal C(M,\tau,\mathbf{c},\tfrr)$ as sets since they have the same underlying real structure.

The \emph{real gauge group} $\mathcal{G}(M,\tilde{\mR})$ consists of $\tau$-skew-invariant gauge transformations, that is, an element $g \in \mathcal{G}(M,\tilde{\mR})$ is a map $g\colon M \to S^1$ satisfying $\tau^* g=\overline{g}$, thought of as an $\spinc$-automorphism.
The \emph{framed real gauge group} $\mathcal{G}(M,\tau,\mathbf{c})$ consists of elements $g$ of the real gauge group $\mathcal{G}(M,\tau)$ for which
\[
g|_{\mathbf{c}} \equiv 1.
\]
The elements of $\mathcal{G}(M,\tau,\mathbf{c})$  are precisely the automorphisms of a {\rrsc} structure that preserves the trivializations at the basepoints.

The \emph{framed real configuration space of gauge equivalence classes} over $(M,\tau,\mathbf{c},\tilde{\mR})$ is the quotient space:
\[
\wtilde{\mathcal{B}}(M,\tau,\mathbf{c},\tilde{\mR})=
\mathcal{C}(M,\tau,\mathbf{c},\tilde{\mR})/\mathcal{G}(M,\tau,\mathbf{c})
\]
Alternatively, the marked configurations and framed gauge groups will be denoted as $\tilde{\mathcal{B}}_{\mathbf{c}}$ and $\mathcal{G}_{\mathbf{c}}$, respectively.
The following can be proved by the argument in \cite[Prop.~9.3.1]{KMbook2007}.
\begin{prop}
    Suppose $2(k+1) > \dim M$.
Then the framed quotient space $\tilde{\mathcal{B}}_{\mathbf{c}}$ is Hausdorff. \hfill \qedsymbol
\end{prop}
In what follows, assume $M$ is connected.
The assumption that $\mathbf{c}$ is nonempty implies the framed real gauge group $\mathcal{G}(M,\tau,\mathbf{c})$ is a finite-index subgroup of $\mathcal{G}(M,\tau)$.
Indeed, $\mathcal{G}(M,\tau,\mathbf{c})$ is precisely the kernel of the homomorphism (not necessarily surjective)
\[\textsf{ev}_{\mathbf{c}}: \mathcal{G}(M,\tau) \to \{\pm 1\}^{|\mathbf{c}|}\]
where $\textsf{ev}_{\mathbf{c}}$ evaluates every $g \colon M \to S^1$ at each connected component of $\mathbf{c}$ (the values are necessarily $\{\pm 1\}$ by skew-invariance).
In particular, the constant function $(-1)$ lies outside the kernel.
Hence $\mathcal{G}(M,\tau,\mathbf{c})$ is a proper subgroup and acts freely on $\mathcal{C}(M,\tau,\mathfrak{s},\tilde{\mR})$.
As an example of non-epimorphism $\mathsf{ev}_{\mathbf{c}}$, if all the components of $\mathbf{c}$ lie on the same connected component of $\text{Fix}(\tau)$, then their evaluations coincide.

It follows that $\tilde{\mathcal{B}}_{\mathbf{c}}$ is a finite (branched) covering of its unframed counterpart $\mathcal{B}$, where the deck transformation is isomorphic to $(\mathbf{Z}/2)^k$ for some $k \ge 1$.
In particular, the reducible locus $\mathcal{B}^{red}$ has stabilizer isomorphic to $\mathbf{Z}/2$, generated by the constant gauge transformation $(-1)$.

Later in Section~\ref{sec:cylinder_functions}, it will be useful to consider a \emph{based gauge group} $\mathcal{G}^o = \mathcal{G}^o(M,\tau,c)$ which is a special case of the framed gauge group when $|\mathbf{c}|=1$.
This is always an index-$2$ subgroup of $\mathcal{G}(M,\tau)$ that excludes the $(-1)$-constant gauge transformation.
If $(M,\tau,\mathbf{c})$ is marked, $\mathcal{G}^o$ depends on a choice of some component $c$ of $\mathbf{c}$.

To summarize, there are nested subgroups
\[
\mathcal{G}(M,\tau,\mathbf{c}) \le \mathcal{G}^o(M,\tau,c) \le \mathcal{G}(M,\tau)
\]
corresponding to coverings
\begin{equation*}
    \wtilde{\mathcal{B}}_{\mathbf{c}}(M,\tfrr) \to \mathcal{B}^o(M,\tfrr) \to \mathcal{B}(M,\frr).
\end{equation*}
The roles of $\mathcal{G}^o$ and $\mathcal{B}^o$ in this article will be minimal.
In particular, the choice of $p \in \mathbf{p}$ is immaterial for Section~\ref{sec:cylinder_functions}.
\subsection{Homotopy types of the configuration spaces}
In real monopole Floer theory, the subset of reducible configurations $\mathcal{B}^{\red}(M,\frr)$ has the homotopy type of
\emph{the real Picard torus} of the real manifold $(M,\tau)$: \[\mathbf{T}_R=H^1(M;\mathbf{R})^{-\tau^*}/H^1(M;\mathbf{Z})^{-\tau^*}\]
which parametrizes all real projectively flat {$\spinc$} connections on $(M,\tau)$.
This torus has the following framed analogue, as a torus of the same rank, but more naturally viewed as a covering space of the unframed Picard torus.
In what follows, assume $M$ is connected.
\begin{defn}
    \label{defn:framed_real_Picard_torus_Y}
The \emph{framed real Picard torus}  is
\[
\wtilde{\mathbf{T}}_{R,\mathbf{c}}=\frac{H^1(M;\mathbf{R})^{-\tau^*}}{\Upgamma(M,\tau,\mathbf{c})},
\]
where $\Upgamma(M,\tau,\mathbf{c})$ is a lattice of $H^1(M;\mathbf{R})^{-\tau^*}$ isomorphic to $\pi_0(\mathcal{G}(M,\tau,\mathbf{c})$), defined as follows.
Enumerate the components of $\mathbf{c}$ as $c_0,c_1,\dots,c_n$.
Every $H^1(M;\mathbf{Z})^{-\tau^*}$ gives rise to an $\tau$-equivariant map $g\colon Y \to S^1$, up to an overall $\pm 1$ factor.
Let
\[
ev_{\mathbf{c}} \colon H^1(M;\mathbf{Z})^{-\tau^*} \to \{\pm 1\}^{n}, \quad ev_{\mathbf{c}}(g) = \big(g(c_0)^{-1}g(c_1),\dots,g(c_0)^{-1}g(c_n)\big)
\]
and
\begin{equation}
    \Upgamma(M,\tau,\mathbf{c}) = \ker(ev_{\mathbf{c}}).
\end{equation}
\end{defn}

\begin{lem}
    $\Upgamma(M,\tau,\mathbf{c})$ is well-defined.
\end{lem}
\begin{proof}
    By $\tau$-skew-invariance, the value of $g$ is constant on each $c_i$ (and on the component of fixed points containing $c_i$);
    the $g(c_0)^{-1}g(c_1) = g(c_0)g(c_1)$ resolves the $(\pm 1)$ ambiguity.
    Hence each component of $ev_{\mathbf{c}}$ is a well-defined homomorphism to $\{\pm 1\} \cong \mathbf{Z}/2$.
    While $ev_{\mathbf{c}}$ may depend on the ordering of $\mathbf{c}$, lying in its kernel means
    \[g(c_0) = g(c_1) = \dots = g(c_n).\qedhere\]
\end{proof}
\begin{rem}
Here is an alternative definition of $\Upgamma(M,\tau,\mathbf{c})$.
For each $i$, choose an arc $\gamma_i$ from $c_0$ to $c_i$ in $(M \setminus \mathbf{c})$ such that $\tau(\gamma_i) \cap \gamma_i$ consists of exact the two endpoints.
The circle $[\gamma_i \cup \tau(\gamma_i)]$ defines a $\tau_*$-skew-invariant 1-cycle in $H_1(M;\mathbf{Z})$.
Evaluations of each $H^1(M;\mathbf{Z})^{-\tau^*}$ over the $n$ 1-cycles $([\gamma_i \cup \tau(\gamma_i)])$ modulo-$2$ defines the homomorphism
\[ev_{\boldsymbol{\gamma}} \colon H_1(M;\mathbf{Z})^{-\tau^*} \to (\mathbf{Z}/2)^n,\]
As a consequence of the above lemma,
    the kernel of $ev_{\boldsymbol{\gamma}}$ is independent of the ordering of $c_i$ and the choices of $\gamma_i$'s.

To see this more geometrically, note $\gamma_i$ is an element of the chain group $C_1(M,c_0 \cup c_i)$ and any two $\gamma_i$'s differ by an element $\beta$ in $C_1(M)$.
    Evaluations of $H^1(M;\mathbf{Z})^{-\tau^*}$ over $[\beta - \tau_*(\beta)]$ are even, so the $ev_{\boldsymbol{\gamma}}$ is independent of the $\upgamma_i$.
    To see the kernel is independent of the reference component $c_0$, define a new system of $(\gamma_i)$ by concatenation from a new reference $c_0$.
\end{rem}
The torus $\tilde{\mathbf{T}}_{R,\mathbf{c}}$ is the framed version of the real Picard torus and $\tilde{\mathbf{T}}_{R,\mathbf{c}} \to \mathbf{T}_R $ is a finite covering of index $2^{k-1}$ if $\mathcal{G}(M,\tau,\mathbf{c}) \le \mathcal{G}(M,\tau)$ has index $2^k$ for some $k > 0$.
Denote the corresponding group of deck transformations as
\[
    \mathsf{D}_{\mathbf{c}} \cong  \frac{H^1(M;\mathbf{Z})^{-\tau^*}}{\Upgamma(M,\tau,\mathbf{c})}
    \cong \frac{\mathcal{G}(M,\tau)}{\pm \mathcal{G}(M,\tau,\mathbf{c})},
\]
where $\pm \mathcal{G}(M,\tau,\mathbf{c})$ is the subgroup generated by $\{\pm 1\}$ and $\mathcal{G}_{\mathbf{c}})$.
Thus $\mathsf{D}_{\mathbf{c}}$ fits in the exact sequence below:
\[
    0 \to \Upgamma(M,\tau,\mathbf{c}) \to H^1(M;\mathbf{Z})^{-\tau^*}
    \to \mathsf{D}_{\mathbf{c}} \to 0.
\]
Thus $\Upgamma(M,\tau,\mathbf c)$ is a sublattice.
The sequence need not split.
The proof of the following lemma is deferred to Section~\ref{sec:tangeng_space}.
\begin{lem}
\label{lem:config_homotopy_type}
    The group of connected components of $\mathcal{G}(M,\tau,\mathbf{c})$ is isomorphic to $\Upgamma(M,\tau,\mathbf{c})$.
    Thus there is a homotopy equivalence
    \begin{equation*}
        \tilde{\mathcal{B}}(M,\mathbf{c},\wtilde{\frr}) \simeq
        \tilde{\mathbf{T}}_{R,\mathbf{c}}.
    \end{equation*}
    The cohomology ring of $\tilde{\mathcal{B}}(M,\mathbf{c},\frr)$ is isomorphic to the exterior algebra $\Lambda^*(\Upgamma^*(M,\tau,\mathbf{c}))$ over $\mathbf{F}$.
\end{lem}
\subsection{Tangent spaces}
\label{sec:tangeng_space}
In what follows, let $k \ge 0$ be an integer or half-integer satisfying $2(k+1) > \dim M$.
Let $A_0$ be a smooth reference real $\spinc$ connection, let $\tfrr$ act on the forms by $(-\tau^*)$, and define Sobolev space:
\begin{equation*}
    \mathcal{C}_k(M,\tilde{\frr}) = (A_0,0) + L^2_k(M;iT^*M \oplus W)^{\tilde{\frr}}.
\end{equation*}
For $j \le k$, let
\begin{equation*}
    \mathcal{T}_j = L^2_j(M;iT^*M \oplus W)^{\tilde{\frr}} \times
    \mathcal{C}_k(M,\tilde{\frr})
\end{equation*}
be the trivial vector bundle over $\mathcal{C}_k(M,\tfrr)$.
Given a configuration $\upgamma = (A_0,\Phi_0)$, the linearization of the gauge group action at $\upgamma$ defines an operator
\begin{align}\label{eq:defn_d_gamma}
    \mathbf{d}_{\upgamma} \colon L^2_{j+1}(M,\mathbf{c};i\mathbf{R})^{-\tau^*}
    &\to \mathcal{T}_{j,\upgamma}\\
    \mathbf{d}_{\upgamma}(\xi)
    &= (-d\xi, \xi \Phi_0),
\end{align}
where, by definition, a function in $L^2_{j+1}(M,\mathbf{c};i\mathbf{R})^{-\tau^*}$ vanishes over $\mathbf{c}$.
Denote the corresponding formal adjoint by
\begin{align}\label{eq:defn_d_star_gamma}
    \mathbf{d}^*_{\upgamma}:
    \mathcal{T}_{j,\upgamma} &\to L^2_{j-1}(M,\mathbf{c};i\mathbf{R})^{-\tau^*}\\
    (a,\phi) &\mapsto
    -d^*a + i \re\langle i\Phi_0,\phi\rangle.
\end{align}
Let $\mathcal{J}_{j,\upgamma} \subset \mathcal{T}_{j,\upgamma}$ be the image of $\mathbf{d}_{\upgamma}$, and $\mathcal{K}_{j,\upgamma} \subset \mathcal{T}_{j,\upgamma}$ be the $L^2$-orthogonal complement of $\mathcal{J}$, that is,
\begin{equation*}
    \mathcal{K}_{j,\upgamma} = \left\{  (a,\phi) \ | \ \mathbf{d}^*_{\upgamma}(a,\phi) = 0 \text{ and } \langle a, \overrightarrow{n}\rangle = 0 \text{ at } \del M \right\}.
\end{equation*}
\begin{prop}
    \label{prop:tangent_K_J_decomp}
  Over the entire $\mathcal{C}_k(M,\wtilde{\frr})$, the subspaces $\mathcal{J}_{j,\upgamma}$ and $\mathcal{K}_{j,\upgamma}$ define a smooth decomposition (of \emph{closed} subbundles)
    \begin{equation*}
        \mathcal{T}_j
        = \mathcal{J}_j \oplus \mathcal{K}_j.
    \end{equation*}
\end{prop}
\begin{proof}
    If $\upgamma$ is irreducible, this follows from \cite[Prop.~9.3.4]{KMbook2007}.
    If $\upgamma$ is reducible, the conditions become (cf.~\cite[p.~150]{KMbook2007}):
    \begin{equation*}
        \begin{cases}
            \Delta \xi = d^*a\\
            \langle d\xi, \overrightarrow{n} \rangle = \langle a, \vec{n}\rangle \quad \text{at }\del M
        \end{cases}
    \end{equation*}
    which has an unique solution $\xi \in L^2_{j+1}(M,\mathbf{c};i\mathbf{R})$ subject to $\int_M \xi = 0$.
    But this integral constraint is vacuous in the the real case, since every $\tau$-anti-invariant function integrates to zero.
\end{proof}
\begin{proof}[Proof of Lemma~\ref{lem:config_homotopy_type}]
	Choose a reducible reference configuration $\upgamma = (A_0,0)$.
    The proof of Proposition~\ref{prop:tangent_K_J_decomp} suggests that the gauge group action provides a diffeomorphism
    \begin{align*}
        \mathcal{G}^{\perp}_{j+1}(M,\tau,\mathbf{c}) \times \mathcal{K}_{j,\upgamma}
        &\to \mathcal{C}_j \\
        (e^{\xi}, (a,\phi)) &\mapsto
        (A_0 + (a-d\xi)\otimes 1, e^{\xi}\phi),
    \end{align*}
    where
    \[\mathcal{G}^{\perp}_{j+1}(M,\tau,\mathbf{c}) = \{e^{\xi}:\xi \in L^2_{j+1}(M,\mathbf{c};i\mathbf{R})^{-\tau^*}\]
    comprises the nullhomotopic gauge transformations.
    Moreover, $\mathcal{G}_{j+1}(M,\tau,\mathbf{c}) = \mathcal{G}^h_{\mathbf{c}} \times \mathcal{G}^{\perp}_{j+1}$ such that $\mathcal{G}^h$ consists of harmonic gauge transformations.
    But the harmonic subgroup is discrete in the real case --- there is an isomorphism
    \begin{equation*}
         \mathcal{G}^h_{\mathbf{c}} \cong \Upgamma(M,\tau,\mathbf{c}).
    \end{equation*}
    Compare this with \cite[\S 5.7]{ljk2022}, where the harmonic group contains a $\{\pm 1\}$ subgroup of constant gauge transformations, and with \cite[\S 9.7]{KMbook2007} where the harmonic group contains an $S^1$.
\end{proof}
The \emph{Coulomb-Neumann slice} is the affine subspace  $\mathcal{S}_{k,\upgamma} \subset \mathcal{C}_k(M,\tilde{\frr})$ with tangent space $\mathcal{K}_{j,\upgamma}$.
Indeed, every $\upgamma \in \tilde{\mathcal{B}}_{k}(M,\tilde{\frr})$ there is an open neighbourhood $U \subset \mathcal{S}_{k,\upgamma}$ such that the composition of inclusion and quotient maps
\[
\mathcal{U} \to \mathcal{C}_k(M,\tilde{\frr}) \to \tilde{\mathcal{B}}_{k}(M,\wtilde{\frr})
\]
is a diffeomorphism onto its image.
Compare with \cite[Cor.~9.3.8]{KMbook2007}.
Since $\mathcal{G}_{\mathbf{c}}$ acts freely and the linearization of its action has closed range:
\begin{cor}
    When $2(k+1) > \dim M$, the quotient space $\tilde{\mathcal{B}}_k(M,\tfrr)$ is a Hilbert manifold without boundary.
\end{cor}

Let $\mathcal{V}_j \to \mathcal{C}_k$ be the trivial vector bundle with fibre $L^2_j(M;i\mathfrak{su}(S^-)\oplus S^+)^{\tfrr}$ if $M$ is a 4-manifold, and with fibre $L^2_j(M;i\mathfrak{su}(S)\oplus S)^{\tfrr}$ if $M$ is a 3-manifold.
Let $\mathfrak{F}$ denote either the \emph{4-dimensional Seiberg--Witten operator}
\[
    (A,\Phi) \mapsto \left(\frac{1}{2}\rho(F^+_A) - (\Phi\Phi^*)_0, D_A^+\Phi\right),
\]
or the \emph{3-dimensional Seiberg--Witten operator}
\[
    (B,\Psi) \mapsto \left(\frac{1}{2}\rho(F_B) - (\Psi\Psi^*)_0, D_B\Psi\right).
\]
Such an operators is $\tfrr$-equivariant, and extends to a section of
$\mathcal{V}_j \to \mathcal{C}_k(M,\mathbf{c},\tfrr)$ as a ($\mathcal{G}_{\mathbf{c},k-1}$-equivariant) section of $\mathcal{V}_{k-1}$.
In general, the framed gauge group $\mathcal{G}_{\mathbf{c},k+1}$ acts smoothly on $\mathcal{V}_j$ whenever $j \le k$.
\subsection{Configurations on cylinders}
\label{sec:config_on_cylinders}
Assume $I \subset \mathbf{R}$ is an interval.
Denote by $(Z,\tau,\mathbf{p}_I)$ the product $I \times (Y,\tau,\mathbf{p})$.
In the framed setting, there is no need to blow up on $Y$ or $Z$.
A 4-dimensional configuration $\upgamma = (A,\Phi)$ on $Z$ gives rise to a path of  3-dimensional configurations $(\breve A(t),\breve \Phi(t))$ for which
\[
(A,\Phi) = (\check A + c \otimes 1_{S_Z} dt,\breve \Phi),
\]
where $t \in I$ is the time direction and $c$ is a time-dependent imaginary-valued function.
If $c = 0$, then $A$ is in \emph{temporal gauge}.
In the path notation, the 4-dimensional Seiberg--Witten equations takes the form of a flow equation:
\begin{align}\label{eq:grad_flow_1}
    \frac{d}{dt}\breve{A}- dc &= -\left(\frac{1}{2} *F_{\check{A}^t} + \rho^{-1}(\check \Phi\check \Phi^*)_0,
    \right) 1_S, \\
     \label{eq:grad_flow_2} \frac{d}{dt}\breve{\Phi} + c\breve{\Phi} &= -D_{\breve{A}}(\breve \Phi).
\end{align}

The Clifford multiplication $\rho$ identifies $\mathfrak{su}(S)$ with $T^*Y$, and consequently identifies $\mathcal{T}_j(Y)$ with $\mathcal{V}_j(Y)$.
It follows that, in temporal gauge, equations~\eqref{eq:grad_flow_1} and~\eqref{eq:grad_flow_2} are precisely the downward gradient flow equations of the \emph{Chern--Simons--Dirac functional} $\mathcal{L} \colon \underline{\mathcal{C}}(Y,\mathfrak{s}) \to \mathbf{R}$, defined over all ordinary configurations:
\[
\mathcal{L}(B,\Psi) = -\frac{1}{8}\int_Y(B^t-B_0^t) \wedge (F_{B^t} + F_{B_0^t}) + \frac{1}{2} \int_Y \langle D_B \Psi, \Psi \rangle d\text{vol}.
\]
The formal $L^2$ gradient $\grad \mathcal{L}$ of $\mathcal{L}$ is precisely the 3-dimensional Seiberg--Witten operator $\mathfrak{F}$.
By compatibility of a {\rrscs}, $\grad \mathcal{L}$ is $\tfrr$-equivariant and thus a well-defined vector field on the real subspace $\mathcal{C}(Y,\tfrr) \subset \underline{\mathcal{C}}(Y,\mathfrak{s})$.
In what follows, view the $\mathcal{L}$ as a functional over $ \mathcal{C}(Y,\tilde{\frr})$.

In Section~\ref{sec:perturbations}, the transversality and regularity of critical points and trajectories will be achieved for the \emph{perturbed CSD functional}
$\pertL \colon \mathcal{C}(Y,\tfrr) \to \mathbf{R}$
obtained by adding to $\mathcal{L}$ a $\mathcal{G}_{\mathbf p}$-gauge-invariant continuous function $f \colon \mathcal{C}(Y,\tilde{\frr}) \to \mathbf{R}$.
Denoting the formal gradient of $f$ by $\mathfrak{q}$, let
the \emph{perturbed gradient} be
\begin{equation}
    \label{eq:grad_pertL}
    \grad \pertL = \grad \mathcal{L} + \mathfrak{q}.
\end{equation}
By the identification $\mathcal{T}_j(Y) \cong \mathcal{V}_j(Y)$, write $\mathfrak{q}$ as $(\mathfrak{q}^0,\mathfrak{q}^1)$ in $L^2(Y;iT^*Y) \oplus L^2(Y;S)$.
Using this, the operator~\eqref{eq:grad_pertL} becomes
\begin{equation*}
    \mathfrak{F}_{\frakq}(B,\Psi) = \left(*\frac12F_{B^t} + \rho^{-1}(\Psi\Psi^*)_0 + \mathfrak{q}^0(B,\Psi),D_B\Psi + \mathfrak{q}^1(B,\Psi)\right),
\end{equation*}
which will be referred to as \emph{perturbed 3-dimensional Seiberg--Witten operator}.
A critical point of $\pertL$ with zero spinorial component is a \emph{reducible}; otherwise, it is \emph{irreducible}.
Framed-gauge invariance implies $\hat{\mathfrak{q}}^1(u(B,\Psi)) = u\hat{\mathfrak{q}}^1(B,\Psi)$.
The component $\mathfrak{q}^1(B,\Psi)$ needs not vanish at $(B,0)$, as $-1$ does not lie in $\mathcal{G}_{\mathbf{p}}$.

Suppose $I$ is compact and consider the configuration space $\mathcal{C}_{k}(Z,\mathbf{p}_I,\tfrr_{Z})$.
Via the Clifford multiplication, identify the fibre $L^2_j(Z;i\mathfrak{su}(S^+)\oplus S^-)^{\tfrr_Z}$ of the trivial bundle $\mathcal{V}_j\to \mathcal{C}_k$ with $L^2_j(Z;iT^*Y\oplus S)^{\tfrr_Z}$.
The 3-dimensional perturbation $\frakq$, viewed as a section $\mathcal{C}_k \to \mathcal{T}_0$, defines a section $\frakq(\check{A},\check{\Phi})$ over a path $(\check A(t),\check \Phi(t))$, and thus a section
\[
    \hat{\frakq}\colon \mathcal{C}(Z,\tfrr_Z) \to \mathcal{V}_0(Z,\tfrr_Z),
\]
where $(\check A(t),\check \Phi(t))$ is interpreted as an element of  $\mathcal{C}(Z,\tfrr_Z)$.

Denote $\hat{\mathfrak{q}} = (\hat{\mathfrak{q}}^0,\hat{\mathfrak{q}}^1)$ where $\hat{\mathfrak{q}}^0 \in L^2(Z;i\mathfrak{su}(S^+))^{\tfrr_Z}$ and $\hat{\mathfrak{q}}^1 \in L^2(Z;S^-)^{\tfrr_Z}$.
The \emph{4-dimensional
perturbed Seiberg--Witten operator} $\mathfrak{F}_{\hat{\frakq}} = \mathfrak{F}+ \hat{\frakq}$ is the left-hand side of the \emph{4-dimensional perturbed Seiberg--Witten equations}:
\begin{align*}
    \rho_Z(F_{A^t}^+) - 2(\Phi\Phi^*)_0 + 2\hat{\mathfrak{q}}^0(A,\Phi)
    &= 0,\\
    D_A^+ \Phi+ \hat{\mathfrak{q}}^1(A,\Phi) &= 0.
\end{align*}
\subsection{Cylinder functions}
\label{sec:cylinder_functions}
The discussion in this subsection will be exclusively 3-dimensional.
Let $(Y,\tau,\mathbf{p})$ be marked real 3-manifold and fix a relative real $\spinc$ structure $(\mathfrak{s},\tilde{\frr})$.
In monopole Floer homologies, transversality is established by introducing ``tame perturbations'', which in turn, arise as ``cylinder functions''~\cite[\S11.1]{KMbook2007}.
In short, the framed version of cylinder functions are built from
$p$ of the form
\[p \colon \tilde{\mathcal{B}}_{k}(Y,\mathbf{p},\tilde{\mR}) \to \mathbf{R}^n \times \mathbf{T}^t \times \mathbf{R}^m,\]
which in turn descend from $\mathcal{G}_{\mathbf{p}}$-invariant functions $\mathcal{C}_{k}(Y,\tfrr) \to \mathbf{R}$.
These functions arise in two sorts, described below.
\begin{itemize}[leftmargin=*]
    \item $r_c:\mathcal{C}_{k}(Y,\tilde{\frr}) \to \mathbf{R}$, where
    \[r_c(B_0+b\otimes 1,\Psi) = \langle b,c \rangle_Y,\]
    for a coclosed form $c \in \Omega^1(Y;i\mathbf{R})^{-\tau^*}$.
    Since
    \[r_c(u(B_0 + b\otimes 1),\Psi) = r_c(B_0+b \otimes 1, \Psi) + (h \cup [*\bar c])[Y],\]
    the function $r_c$ is fully gauge-invariant only if $c$ is co-exact.
    \item $q_{\Upsilon}:\mathcal{C}_{k}(Y,\tilde{\frr}) \to \mathbf{R}$, where
    \[q_{\Upsilon}(B,\Psi) = \langle \Psi,\tilde{\Upsilon}^{\dag}\rangle_{Y},\]
    for a $\frr$-invariant section $\Upsilon$ of $\mathbb{S}$.
    The definitions of
    $\Upsilon$ and $\Upsilon^{\dag}$
    are as follows:

    Consider the \emph{universal spinor bundle} $\mathbb{S} \to \tilde{\mathbf{T}}_{R,\mathbf{p}} \times Y$, obtained as the quotient of  $H^1(Y;i\mathbf{R})^{-\tau^*} \times S$ by $\Upgamma(Y,\tau,\mathbf{p})$.
    A $\frr$-invariant section $\Upsilon$ of $\mathbb{S}$ can be lifted to
    a  quasi-periodic  $\frr$-invariant section of the bundle
    \[H^1(Y;i\mathbf{R})^{-\tau^*} \times S \to H^1(Y;i\mathbf{R})^{-\tau^*} \times Y,\]
    in the sense that
    \(\tilde{\Upsilon}_{\alpha + \kappa}(y) = u_{\kappa}(y)\tilde{\Upsilon}_{\alpha}(y),\) for
    every $\kappa \in \Upgamma(Y,\tau,\mathbf{c})$.
    Here, $u_{\kappa}$ is the unique harmonic representative of $\kappa$.
    Hence every $\Upsilon$ defines a $\mathcal{G}_{\mathbf{p}}$-equivariant map
    \(\Upsilon^{\dag}:\mathcal{C}(Y,\tilde{\mR}) \to C^{\infty}(S)^{\tilde{\mR}}\).
    By letting $b_{harm}$ be the harmonic part of $b$ and $G$ be the Green's function, define
    \begin{equation}
        \label{eq:Upsilon-section-defn}
    \Upsilon^{\dag}(B_0+b\otimes 1, \Psi) = e^{-Gd^*b}\tilde{\Upsilon}_{b_{harm}}.
    \end{equation}
\end{itemize}

Recall that $\Upgamma(Y,\tau,\mathbf{p}) \le H^1(Y;i\mathbf{Z})^{-\tau^*}$ is a lattice of full rank in $H^1(Y;i\mathbf{R})^{-\tau^*}$.
Choose harmonic forms $\eta_1,\dots,\eta_t$ representing an integral basis of $H^1(Y;i\mathbf{Z})^{-\tau^*}$.
Let $\omega_1,\dots,\omega_t$ be the dual basis, satisfying
$\langle \eta_j,\omega_i\rangle_Y=\delta_{ij}$.
Define the period lattice
\[
P_{\mathbf p} = 2\pi\left\{
(n_1,\dots,n_t)\in\mathbf Z^t:
\sum_j n_j[\eta_j]\in\Upgamma(Y,\tau,\mathbf p)
\right\} \subset 2\pi\mathbf Z^t.
\]
View the map
$(B,\Psi) \mapsto (r_{\omega_1}(B,\Psi),\dots,r_{\omega_t}(B,\Psi))$
as onto $\mathbf{T}^t =  \mathbf{R}^t/P_{\mathbf{p}}$.
The harmonic coordinates define a map:
\[
(B,\Psi)\longmapsto
\bigl(r_{\omega_1}(B,\Psi),\dots,r_{\omega_t}(B,\Psi)\bigr)
\bmod{P_{\mathbf p}}
\]
to the torus
$\mathbf T^t=\mathbf R^t/P_{\mathbf p}$.

\begin{defn}
A ($\mathcal{G}_{\mathbf{p}}$-invariant) function $f$ on $\mathcal{C}(Y,\mathbf{p},\tilde{\frr})$ is a \emph{framed real cylinder function} if it arises as $h \circ p$, where:
\begin{itemize}[leftmargin=*]
    \item the function $h$ is a compactly supported smooth function on $\mathbf{R}^n \times \mathbf{T}^t \times \mathbf{R}^m$;
    \item the map $p:\mathcal{C}(Y,\tilde{\frr}) \to \mathbf{R}^n \times \mathbf{T}^t \times \mathbf{R}^m$ has the form
    \[
    p(B,\Psi) = (r_{b_1}(B,\Psi),\dots,r_{b_n}(B,\Psi), r_{\omega_1}(B,\Psi),\dots,r_{\omega_t}(B,\Psi), q_{\Upsilon_1}(B,\Psi),\dots,
    q_{\Upsilon_m}(B,\Psi) ),
    \]
    where $b_1,\dots,b_n$ is any collection of $\tau$-skew-invariant coexact forms, and $\Upsilon_1,\dots,\Upsilon_m$ is any collection of $\frr$-invariant sections of $\mathbb{S}$.
\end{itemize}
\end{defn}
\begin{rem}
	Unlike the real cylinder functions in \cite{ljk2022},
	the framed counterpart does not require full real gauge invariance under $\mathcal{G}(Y,\tau)$.
	In particular, since $(-1) \not\in \mathcal{G}_{\mathbf{p}}$, the framed cylinder functions need not be even
	on the spinor part.
	Hence, a perturbation $\mathfrak{q} = \grad f$ may not vanish on the reducible locus, i.e. $\mathfrak{q}^1(B,0) \ne 0.$
\end{rem}
\begin{prop}
    \label{prop:enough_embeddings}
    Suppose $K$ is a compact subset of a finite-dimensional $C^1$-submanifold $M$ in $\tilde{\mathcal{B}}_k(Y,\mathbf{p},\tilde{\frr})$.
    Then there exists a collection of $\tau^*$-skew-invariant coclosed forms $c_{\nu}$, a collection of $\tfrr$-invariant sections $\Upsilon_{\mu}$ of $\mathbb{S}$, and a neighbourhood $U$ of $K$ in $M$, for which the corresponding map
    \[
    p \colon \tilde{\mathcal{B}}_k(Y,\mathbf{p},\tilde{\frr}) \to \mathbf{R}^n \times \mathbf{T}^t \times \mathbf{R}^m
    \]
    induces an embedding of $U$.
    Furthermore, given any $[B,\Psi]$ in $\tilde{\mathcal{B}}_{\mathbf{p},k}$ and any nonzero vector $v$ to $\tilde{\mathcal{B}}_{\mathbf{p},k}$, there exists a cylinder function $f$ whose differential $\mathcal{D}_{[B,\Psi]}f(v)$ is nonzero.
\end{prop}
\begin{proof}
    This follows from the same argument in \cite[Prop.~6.5]{ljk2022} applied to $\widetilde{B}_{\mathbf{p},k}$
	rather than the based configuration space $\mathcal{B}^o$.
	To highlight what is new in the framed case, suppose two configurations $\mathfrak{a},\mathfrak{b}$ lie above the same point in $\mathcal{B}_k$;
	in other words, $u\mathfrak{a} = \mathfrak{b}$ for some $u \in \mathcal{G} \setminus \mathcal{G}_{\mathfrak{p}}$.
	Then either $[u] \not\in \Gamma(Y, \tau, \mathbf{p})$, or $u \in -\mathcal{G}_{\mathbf{p}}$.
	In the former case,
	$\mathfrak{a}$ and $\mathfrak{b}$ can already be separarted by the function $(r_{\omega_1}, \dots, r_{\omega_t})$.
	In the latter case, $\mathfrak{b}$ is framed gauge equivalent to $(B,-\Psi)$ where $\mathfrak{a} = (B,\Psi)$.
	Then these two points can be separarted by an $\Upsilon$ for which $q_{\Upsilon}(B,\Psi) \ne 0$.
	Furthermore, assume $v = (0,\psi)$ at a reducible $[B,0]$.
	Then one can choose $\Upsilon$ such that $\langle \psi , \Upsilon^{\dag}(B)\rangle \ne 0$ and a compactly supported $h$ such that $h'(0) \ne 0$.
	Then $\mathcal{D}(h \circ \mathfrak{q}_{\Upsilon}) \ne 0$.
\end{proof}
\subsection{Perturbations}
\label{sec:perturbations}
Essentially all estimates regarding perturbed gradients from \cite[\S 7]{ljk2022} (which in turn are consequences of \cite[\S 10.5]{KMbook2007}) carry over in the framed setting.

\begin{defn}
    \label{defn:tame_pert}
    Given any integer $k \ge 2$,
	a perturbation $\frakq \colon \mathcal{C}_k(Y,\tilde{\frr}) \to \mathcal T_0$ is \emph{$k$-tame} if $\frakq$ is the formal gradient of a $\mathcal{G}_{\mathbf{p}}$-invariant continuous function on $\mathcal{C}_k(Y,\tilde{\frr})$, of the following significance:
    \begin{enumerate}[noitemsep,itemindent=*,label=(\roman*)]
        \item the induced 4-dimensional perturbation $\hat{\mathfrak{q}}$ defines a section $\hat{\mathfrak{q}} \in C^{\infty}(\mathcal{C}_k(Z,\tilde{\frr}_Z),\mathcal{V}_k(Z,\tilde{\frr}_Z))$;
        \item for all $1 \le j \le k$, the perturbation $\hat{\mathfrak{q}}$ defines a section $\hat{\mathfrak{q}} \in C^0(\mathcal{C}_j(Z,\tilde{\frr}_Z),\mathcal{V}_j(Z,\tilde{\frr}_Z))$;
        \item for all $-k \le j \le k$, the derivative $\mathcal{D}\hat{\mathfrak{q}} \in C^{\infty}(\mathcal{C}_k(Z,\tilde{\frr}_Z),\text{Hom}(T\mathcal{C}_k(Z,\tilde{\frr}_Z),\mathcal{V}_k(Z,\tilde{\frr})))$ extends to
        \[\mathcal{D}\hat{\mathfrak{q}} \in C^{\infty}\left(\mathcal{C}_k(Z,\tilde{\frr}_Z),\text{Hom}(T\mathcal{C}_j(Z,\tilde{\frr}_Z),\mathcal{V}_j(Z,\tilde{\frr}_Z))\right);\]
        \item there is a constant $m_2$ such that
        \(\|\frakq(B,\Psi)\|_{L^2} \le m_2(\|\Psi\|_{L^2} + 1)\) for every $(B,\Psi) \in \mathcal{C}_k(Y,\tilde{\frr})$;
        \item for all $A_0$ smooth, there is a real function $\mu_1$ such that for every $(A,\Phi) \in \mathcal{C}_k(Z,\tilde{\frr}_Z)$
        \[\|\frakq(A,\Phi)\|_{L^2_{1,A}} \le \mu_1\left((\|(A-A_0),\Psi\|_{L^2_{1,A_0}}\right);\]
        \item the perturbation $\mathfrak{q}$ defines a $C^1$ section $ \colon \mathcal{C}_1(Y,\tilde{\frr}) \to \mathcal{T}_0$.
    \end{enumerate}
\end{defn}

\begin{prop}\label{prop:framed_cylinders_tame}
The gradient of every framed real cylinder function is tame.
\end{prop}

\begin{proof}
The coordinate estimates in \cite[\S11.3--11.4]{KMbook2007} apply to real sections over the framed harmonic torus.
These arguments do not require $\{\pm 1\}$-invariance of $g$.
\end{proof}

The following is the downstairs compactness result on finite cylinders, easily adapted from \cite[Prop.~10.7.1]{KMbook2007}.

\begin{prop}
Suppose $\frakq$ is a $k$-tame perturbation and $Z = [t_1,t_2] \times Y$ is a finite cylinder.
If $\upgamma_n \in \mathcal{C}_k(Z,\tilde{\frr}_Z)$ is a sequence of solutions to $\mathfrak F_{\frakq}(\upgamma) = 0$ for which there is a uniform bound
\[\pertL(\check \upgamma_n(t_1)) - \pertL(\check \upgamma_n(t_2)) \le C.\]
Then there exists a sequence $u_n \in \mathcal{G}(Z,\tau,\mathbf{a})$ such that on a smaller cylinder $Z'=[t_1',t_2'] \times Y$ and after passing to a subsequence, $u_n(\upgamma_n)$ converges in $\mathcal{C}_{k+1}(Z,\tilde{\frr}_Z)$.
Furthermore, the framed solutions in $\tcalB_k(Y,\tfrr)$ to the perturbed 3-dimensional Seiberg--Witten equations is compact. \hfill \qedsymbol
\end{prop}
The notion of a large Banach of tame perturbation and its existence follows from the Theorem below.
See the dicussion in \cite[\S 6.5]{ljk2022}.

\begin{thm}\label{thm:largebanachspace}
Let $\{\mathfrak q_i\}$ be a countable collection of tame perturbations arising from cylinder functions on $\mathcal C(Y,\tfrr)$.
Then there exists a separable Banach space $\mathcal{P}$ and a linear map $\mathfrak O:\mathcal P \to C^0(\mathcal C(Y,\uptau),\mathcal T_0)$,
		$\lambda \mapsto \mathfrak q_{\lambda}$,
	such that
	\begin{enumerate}[noitemsep,leftmargin=*,label=(\roman*)]
    \item for every $\lambda \in \mathcal P$, the perturbation $\mathfrak q_{\lambda}$ is tame;
		\item the image of $\mathfrak O$ contains the family $\{\mathfrak q_i\}$;
		\item for all $k \ge 2$, the map $\mathcal P \times \mathcal C_k(Z,\tfrr_Z) \to \mathcal V_k(Z,\tfrr_Z)$ over $Z = I \times Y$
			given by $(\lambda,\upgamma) \mapsto \hat{\mathfrak q}_{\lambda}(\upgamma)$ is smooth; and
		\item the map
			$\mathcal P \times \mathcal C_1(Y,\tfrr) \to \mathcal T_1(Y,\tfrr)$ given by $(\lambda,\beta) \mapsto \mathfrak q_{\lambda}(\beta)$ is continuous and satisfies
			\begin{align*}
				\|\mathfrak q_{\lambda}(B,\Psi)\|_{L^2} &\le \|\lambda\|m_2(\|\Psi\|_{L^2} + 1), \text{ and}\\
				\|\mathfrak q_{\lambda}(B,\Psi)\|_{L^2_{1,A_0}} &\le \|\lambda\|\mu_1\left(\|B-B_0,\Psi\|_{L^2_{1,A_0}}\right),
			\end{align*}
			for some constant $m_2$ and some continuous real function $\mu_1$.
	\end{enumerate}
\end{thm}

\begin{proof}
	Same construction as ~\cite[Thm.~11.6.1]{KMbook2007}.
\end{proof}

The countable family of $\{\mathfrak{q}_i\}$ will be re-enumerated from the following countable family labelled by:
\begin{itemize}[noitemsep]
	\item positive integers $n,m$;
	\item $(-\tau^*)$-invariant coexact forms $c_1,\dots, c_n$ and $\tfrr$-invariant sections $\Upsilon_1,\dots, \Upsilon_m$ of $\mathbb S$;
	\item a compact subset $K$ of $\mathbf{R}^n \times \mathbf T^t \times \mathbf{R}^m$;
	\item a smooth function $g$ on $\mathbf{R}^n\times\mathbf{T}^t\times \mathbf{R}^m$ supported in $K$.
\end{itemize}
For any $(n,m)$ choose a countable collection of $(n+m)$-tuples $(c_1,\dots,c_n,\Upsilon_1,\dots,\Upsilon_m)$ which are dense in the $C^{\infty}$-topology in the space of all such $(n+m)$-tuples.
Next, choose a countable collection of compact subset $K$ of $\mathbf{R}^n\times \mathbf{T}^t \times \mathbf{C}^m$ that is dense in the Hausdorff topology.
The smooth functions $g$ need not be $\{\pm 1\}$-invariant, $(-1)$ is not an element of the framed gauge group.
For each $K$ choose a collection of functions $g_{\alpha} = g(n,m,K)_{\alpha}$ with the following properties
\begin{itemize}[noitemsep,leftmargin=*]
	\item each $g_{\alpha}$ is supported in $K$;
	\item the collection $\{g_{\alpha}\}$ is dense in the $C^{\infty}$-topology of smooth functions supported in $K$;
\end{itemize}

\begin{defn}
\label{defn:large_banach_space}
A large Banach space of tame perturbations is a pair $(\mathcal{P},\mathfrak{O})$ of a separable Banach space $\mathcal{P}$ and a linear map $\mathfrak{O}\colon \mathcal{P} \to C^0(\mathcal{C}(Y,\tfrr), \mathcal{T}_0)$ satisfying the properties in Theorem~\ref{thm:largebanachspace} and containing a countable collections of tame perturbations $\{\mathfrak{q}_i\}$ described above.
\end{defn}

\subsection{Hessians}
Recall $\mathcal{T}_j = \mathcal{J}_j \oplus \mathcal{K}_j$ over the whole $\mathcal{C}_k(Y,\tfrr)$.
In particular, given a 3-dimensional configuration $\alpha$, the Coulomb slice $\mathcal{S}_{k,\alpha}= \alpha + \mathcal{K}_{k,\alpha}$ is orthogonal to the gauge group action;
$\mathcal{K}_{k,\alpha}$ can be identified with the tangent space of $\tilde{\mathcal{B}}_{k}$ at $[\alpha]$.
The \emph{Hessian} $\Hess_{\mathfrak{q},\alpha}:\mathcal{K}_{k,\alpha} \to \mathcal{K}_{k-1,\alpha}$ is the operator
\begin{equation*}
    \Pi_{\mathcal{K}_{k-1}} \circ \mathcal{D}_{\alpha}
    \grad \pertL:
    \mathcal{T}_{k,\alpha} \to \mathcal{K}_{k-1,\alpha}
\end{equation*}
where $\Pi_{\mathcal{K}_{k-1}}$ is the projection onto the subspace $\mathcal{K}_{k-1,\alpha} \subset \mathcal{T}_{k-1,\alpha}$.
The Hessians over all points of $\mathcal{C}_k(Y,\tfrr)$ define a smooth $\mathcal{G}(Y,\tau,\mathbf{p})$-equivariant bundle map $\Hess_{\mathfrak{q}} \colon \mathcal{K}_k \to \mathcal{K}_{k-1}$, which in turn can be interpreted as a map $[\mathcal{T}_k] \to [\mathcal{T}_{k-1}]$.
It is convenient to introduce the \emph{extended Hessian} operator
\begin{equation*}
    \widehat{\Hess}_{\mathfrak{q},\alpha}:
    \mathcal{T}_{k,\alpha} \oplus L^2_k(Y,\mathbf{p};i\mathbf{R})^{-\tau^*} \to
    \mathcal{T}_{k-1,\alpha} \oplus L^2_{k-1}(Y,\mathbf{p};i\mathbf{R})^{-\tau^*}.
\end{equation*}
given by
\begin{equation*}
    \widehat{\Hess}_{\mathfrak{q},\alpha}
    =
    \begin{bmatrix}
        \mathcal{D}_{\alpha} \grad \pertL & \mathbf{d}_{\alpha}\\
        \mathbf{d}^*_{\alpha} & 0
    \end{bmatrix},
\end{equation*}
where $\mathbf{d}_{\alpha}$ is defined in~\eqref{eq:defn_d_gamma}.
At $\alpha = (B_0,\Psi_0)$, the extended Hessian in $\mathcal{T}_{j,\alpha} \oplus L^2_k(Y,\mathbf{p};i\mathbf{R})^{-\tau^*}$ can be written in full as
\begin{align*}
    (b,\Psi,\xi) \mapsto
    &\big(*db-d\xi-2\rho^{-1}(\Psi\Psi^*_0 + \Psi_0\Psi^*)_0 +2\mathcal{D}_{\alpha}\mathfrak{q}^0(b,\Psi),\\
    &D_{B_0}\Psi + \rho(b)\Psi_0 + \xi\Psi_0 +  \mathcal{D}_{\alpha}\mathfrak{q}^1(b,\Psi),
    \\
    &-d^*b + i\text{Re}\langle i\Psi_0, \Psi\rangle \big) .
\end{align*}

In the decomposition $\mathcal{T}_{j,\alpha} = L^2_j(Y;S)^{ \tilde{\mR}} \oplus L^2_j(Y;iT^*Y)^{-\tau^*}$, the operator $\widehat{\Hess}_{\mathfrak{q},\alpha}$ can be re-expressed as a sum
\begin{equation}
    \label{eqn:ext_hessian_1th}
    \widehat{\Hess}_{\mathfrak{q},\alpha} =
    \begin{bmatrix}
        D_{B_0} & 0 & 0\\
        0 & *d & -d\\
        0 & -d^* & 0
    \end{bmatrix} + h
\end{equation}
where the term $h$ is the sum of a zeroth-order operator and terms arising from the perturbation $\mathfrak{q}$.
On the other hand, $\mathcal{T}_j = \mathcal{J}_j \oplus \mathcal{K}_j$;
let $x = \Pi_{\mathcal{J}_{k-1}} \circ \mathcal{D}_{\alpha} \grad \pertL|_{\mathcal{K}_{k,\alpha}}$,
and $y = \Pi_{\mathcal{J}} \mathcal{D}_{\alpha} \grad \pertL|_{\mathcal{J}_{\alpha}}$.
The above operator in block form becomes
\begin{equation}
    \label{eqn:ext_hessian_with_x}
    \widehat{\Hess}_{\mathfrak{q},\alpha} =
    \begin{bmatrix}
        y & x & \mathbf{d}_{\alpha} \\
        x^* & \Hess_{\mathfrak{q},\alpha} & 0\\
        \mathbf{d}^*_{\alpha} & 0 & 0
    \end{bmatrix}.
\end{equation}
By~\cite[Lem.~12.3.2]{KMbook2007} and Hypothesis~(iii) in Definition~\ref{defn:tame_pert},
the operator $x \colon \mathcal{K}_k \to \mathcal{J}_{k-1}$ and its formal adjoint $x^*:\mathcal{J}_k \to \mathcal{K}_{k-1}$ both vanish at critical points.
They extend to bounded operators $\mathcal{K}_j \to \mathcal{J}_j$ and $\mathcal{J}_j \to \mathcal{K}_j$ for $0 \le j \le k$, respectively.
Similarly, $y$ vanishes at critical points.
Using this form of the extended Hessian, one can conclude (compare with \cite[Lem.~12.3.2]{KMbook2007}):
\begin{prop}
    $\Hess_{\mathfrak{q},\alpha} \colon \mathcal{K}_{k,\alpha} \to \mathcal{K}_{k-1,\alpha}$ is symmetric.
    There is a complete orthonormal system $\{e_n\}$ in $\mathcal{K}_{0,\alpha}$ for which
    \begin{equation*}
        \Hess_{\mathfrak{q},n} e_n = \lambda_n e_n
    \end{equation*}
    such that $e_n$ is smooth and $\lambda_n$ is real.
    The span of the eigenvectors is dense in $\mathcal{K}_{k,\alpha}$ for all $k$.
    The number of eigenvectors $\lambda_n$ in any bounded interval is finite.
    In particular, $\Hess_{\mathfrak{q},\alpha}$ is Fredholm of index zero, and is therefore surjective if and only if it is injective.
    \hfill \qedsymbol
\end{prop}
\subsection{Transversality of 3-dimensional critical points}
\begin{defn}
    A critical point $\mathfrak{a} \in \mathcal{C}_k(Y,\tfrr)$ of the vector field $\grad \pertL$ is \emph{non-degenerate} if the smooth section $\grad \pertL$ of $\mathcal{T}_{k-1}$ is transverse to the subbundle $\mathcal{J}_{k-1}$.
\end{defn}
The block form in $\mathcal{J} \oplus \mathcal{K}$
\begin{equation*}
    \mathcal{D}_{\mathfrak{a}}(\grad \pertL) = \begin{bmatrix}
        0 & 0\\
        0 & \Hess_{\mathfrak{q},\mathfrak{q}}
    \end{bmatrix}
\end{equation*}
readily implies the following characterization of non-degeneracy:
\begin{lem}
    \label{lem:nondeg_Hess_suj}
    A critical point $\mathfrak{a}$ is non-degenerate if and only if $\Hess_{\mathfrak{q},\mathfrak{a}}$ is surjective. \hfill \qedsymbol
\end{lem}
\begin{theorem}
    Let $\mathcal{P}$ be a large Banach space of tame perturbations.
    Then there is a residual subset such that,
	for every $\mathfrak{q}$ in this subset,
	the zeros of the section $\grad
    \pertL$ of $\mathcal{T}_{k-1} \to \mathcal{C}_k(Y,\tfrr)$ are non-degenerate at elements of this subset.
\end{theorem}
\begin{proof}
    The proof is identical to the irreducible case of \cite[Thm.~12.5.1 \& Lem.~12.5.2]{KMbook2007}, using the Sard--Smale theorem~\cite[Lem.~12.5.1]{KMbook2007}.
    It suffices to prove the following map is transverse to the zero section of $\mathcal{K}_{k-1}$:
\begin{align*}
    \mathfrak{g} \colon \mathcal{C}_k \times \mathcal{P} &\to \mathcal{K}_{k-1}\\
    (\alpha,\mathfrak{q}) &\mapsto
    \grad \mathcal{L}(\alpha) + \mathfrak{q}(\alpha)
\end{align*}
 which is equivalent to the surjectivity of the following operator over $\mathcal{K}_{k,\alpha} \times T_{\mathfrak{q}}\mathcal{P} \to \mathcal{K}_{k-1,\alpha}$:
\begin{equation*}
    ((b,\psi),\delta \mathfrak{q}) \mapsto \text{Hess}_{\mathfrak{q},\alpha}(b,\psi) + \delta \mathfrak{q}(\alpha).
\end{equation*}
This follows from the denseness assumption in Definition~\ref{defn:large_banach_space},
and the last part of Proposition~\ref{prop:enough_embeddings}.
Indeed, one can deduce that any nonzero $v$ in the kernel of $\Hess_{\mathfrak{q},\alpha}$,
there exists $\delta \mathfrak{q} \in \mathcal{P}$ such that the $L^2$ inner product of $\delta \mathfrak{q}(\alpha)$ with $v$ is nonzero.
This includes the case when $\alpha = (B,0)$ and $v = (0, \psi)$, as the family $g$ for the countable family are not required to be even.
\end{proof}
\section{The analysis of framed Seiberg--Witten trajectories}
\label{sec:analysis_framed}
This section draws heavily from the downstairs analysis in~\cite[Part IV]{KMbook2007}.
The framed moduli spaces of trajectories will be denoted using $\wtilde{N}$,
as they are the framed analogues of the downstairs moduli spaces.
\S\ref{sec:transversality-traj} requires $\mathcal{G}_{\mathbf{p}}$-invariant perturbations,
which is new in the framed setting.
\subsection{Moduli spaces on finite cylinders}
Let $I = [t_1,t_2]$.
Fix $\mathfrak{a}_1,\mathfrak{a}_2$ in $\mathcal{C}_k(Y,\tfrr)$ and two tame perturbations $\mathfrak{q_1},\mathfrak{q_2}$.
Suppose $(\upgamma,\mathfrak{p})$ is a pair where
\begin{itemize}[noitemsep,leftmargin=*]
    \item $\upgamma$ is a configuration in $\mathcal C_k(I \times Y)$ whose restriction to $\{t_i\} \times Y$ is gauge-equivalent to $\mathfrak{a}_i$ for $i=1,2$; and
    \item $\mathfrak{p}$ is a continuous path in the Banach space $\mathcal{P}$ of tame perturbations, with $\mathfrak{p}(t_i) =\mathfrak{p}_i$, for $i=1,2$.
\end{itemize}
Denote the linearized Seiberg--Witten operator as
\[
\mathcal{Q}_{\upgamma,\mathfrak{p}} \colon \mathcal{E} \to \mathcal{F}
\]
where
\[
\mathcal{E}=\mathcal{T}_{1,\upgamma}(I \times Y,\tfrr), \quad
\mathcal{F}=\mathcal{V}_{0,\upgamma}(I \times Y,\tfrr) \oplus
L^2(I \times Y, \mathbf{p}_I; i\mathbf{R}).
\]
\subsection{Moduli spaces over infinite cylinders}
Let $I \subset \mathbf{R}$ be an infinite interval and $(Z,\tau,\mathbf{p}_I)=I \times (Y,\tau,\mathbf{p})$.
Let $(\mathfrak{s},\tfrr)$ be a {\rrscs} on $Z$ pulled back from $(Y,\tau,\mathbf{p})$.
Without blowing-up,
define
\begin{align*}
    \mathcal{C}_{k,\text{loc}}(I \times Y,  \tilde{\mR}) &\subset
    \left( A_0 + L^2_{k,\text{loc}}(I \times Y; iT^*Z)^{-\tau^*} \right) \times
    L^2_{k,\text{loc}}(I \times Y;S^+)^{ \tilde{\mR}} \\ &=
    \mathcal{A}_{k,\text{loc}}(I \times Y, \tilde{\mR}) \times
    L^2_{k,\text{loc}}(I \times Y;S^+)^{ \tilde{\mR}}
\end{align*}
Unlike the $\uptau$-blow-up situation, the framed perturbed Seiberg--Witten operator can be regarded as a section
\begin{equation*}
    \mathfrak{F}_{\mathfrak{q}}: \mathcal{C}_{k,\text{loc}}(I \times Y, \tilde{\mR})
    \to  \mathcal{V}_{k-1,\text{loc}}(I \times Y, \tilde{\mR})
\end{equation*}
\emph{of the trivial vector bundle} $\mathcal{V}_{k,\text{loc}} \to \mathcal{C}_{k,\text{loc}}$,
having fibre
\begin{equation*}
    \mathcal{V}_{j,\text{loc},\upgamma} =
    L^2_{j,\text{loc},\upgamma}(I \times Y; i\mathfrak{su}(S^+))^{-\tau^*} \oplus
    L^2_{j,\text{loc}}(I \times Y;S^-)^{ \tilde{\mR}}.
\end{equation*}
Compare this with \cite[p.219]{KMbook2007}, where the relevant bundle is not locally trivial.

The \emph{framed real $L^2_{k+1,\text{loc}}$-gauge group} $\mathcal{G}_{k+1,\text{loc}}(Z,\tau,\mathbf{p}_I)$ is defined as the group of all $L^2_{k+1,\text{loc}}$ maps to $S^1 \subset \mathbf{C}$ for which $\bar u =\tau^* u$ and $u|_{\mathbf{p}_I} \equiv 1$.
The \emph{configuration space of equivalence classes} is defined to be the quotient
\begin{equation*}
    \tilde{\mathcal{B}}_{k,\text{loc}}(I \times Y,\tilde{\mR}) =
    \mathcal{C}_{k,\text{loc}}(I \times Y,\tilde{\mR})/
    \mathcal{G}_{k+1,\text{loc}}(I \times Y,\tilde{\mR}).
\end{equation*}

Every critical point $\mathfrak{b}$ of $\pertL$ corresponds to a translation-invariant element $\upgamma_{\mathfrak{b}} \in \mathcal{C}_{k,\text{loc}}(Z, \tilde{\mR})$.
Write $[\upgamma_{\mathfrak{b}}]$ for its gauge-equivalence class.
A configuration $[\upgamma] \in \tcalB_{k,\text{loc}}(Z,\tfrr)$ is \emph{asymptotic} to a critical point $[\mathfrak{b}]$ as $t \to \pm \infty$  if, as $t \to \pm \infty$,
\[ [\uptau_t^*\upgamma] \to [\upgamma_{\mathfrak{b}}]\text{ in }\tcalB_{k,\text{loc}}(Z,\tfrr),\]
where $\uptau_t \colon Z \to Z$ translates $(s,y)$ to $(s+t,y)$.
Moreover, if $[\upgamma]$ is asymptotic to $[\frakb]$ as $t \to +\infty$ or $t \to -\infty$, one writes
\[\lim_{\rightarrow}[\upgamma] = [\mathfrak{b}] \quad \text{or} \quad
\lim_{\leftarrow}[\upgamma] = [\mathfrak{b}], \]
respectively.
\begin{defn}
    Let $\wtilde{N}([\mathfrak{a}],[\mathfrak{b}])$ be the space of all configurations in $\tilde{\mathcal{B}}_{k,\text{loc}}(Z,\tfrr)$ which are asymptotic to $[\mathfrak{a}]$ and $[\mathfrak{b}]$ as $t$ tends to $-\infty$ and $+\infty$, respectively, and which solve the perturbed Seiberg--Witten equations:
    \begin{equation*}
      \wtilde{N}([\mathfrak{a}],[\mathfrak{b}]) =
        \left\{ [\upgamma] \in \tcalB_{k,\text{loc}}(Z, \tilde{\mR}) \ \big| \
        \mathfrak{F}_{\mathfrak{q}}(\upgamma) = 0,
        \lim_{\leftarrow}[\upgamma] =[\mathfrak{a}],
        \lim_{\rightarrow}[\upgamma]=[\mathfrak{b}]\right\}.
    \end{equation*}
\end{defn}
Each element $[\upgamma]$ of $\tilde{N}([\mathfrak{a}],[\mathfrak{b}])$ defines a path in $\tcalB_k(Y,\tilde{\mR})$ from $[\mathfrak{a}]$ to $[\mathfrak{b}]$.
Denote by $\tilde{N}_z([\mathfrak{a}],[\mathfrak{b}])$ the component of $\tilde{N}([\mathfrak{a}],[\mathfrak{b}])$ corresponding to paths in the relative homotopy classes of the paths $z \in \pi_0(\tcalB_k(Y, \tfrr),[\mathfrak{a}],[\mathfrak{b}])$;
\begin{equation*}
    \wtilde{N}([\mathfrak{a}],[\mathfrak{b}]) = \bigcup_{z} \wtilde{N}_z([\mathfrak{a}],[\mathfrak{b}]).
\end{equation*}

\begin{prop}
    Suppose $\mathfrak{b}$ is a non-degenerate critical point of the perturbed functional $\pertL$.
    Let $I$,$I'$ be two compact intervals where $I=[t_1,t_2]$, $I'=[t_1',t_2']$, and $t_1 < t_1' < t_2' < t_2$.
    There is a constant $C_1$ and a gauge-invariant neighbourhood $U$ of the constant solution $\upgamma_{\frakb}$ in $\mathcal{C}_k(I \times Y,\tfrr)$ of the following significance.
    For every $\upgamma \in \mathcal{C}_k(I \times Y, \tfrr)$ lying in $U$ that solves the perturbed Seiberg--Witten equations $\mathfrak{F}_{\frakq}(\upgamma) = 0$,
    there is a gauge transformation $u \in \mathcal{G}_{k+1}(I \times Y,\tau,\mathbf{p}_I)$ such that the squared norm of $u(\upgamma) - \upgamma_{\frakb}$ on the smaller interval $I'$ is bounded by the change in $\pertL$ on the larger interval $I=[t_1,t_2]$:
    \[\|u(\upgamma)-\upgamma_{\frakb}\|^2_{L^2_{k+1,B}(I' \times Y)}
    \le C_1(\pertL(t_1)-\pertL(t_2)).\]
\end{prop}
\begin{proof}
    This follows from the proof of \cite[Lem.~13.4.3]{KMbook2007}, except the $\frakb$ needs not be irreducible.
    Indeed, the map
    \[\mathcal{C}_1(Y,\tfrr) \times L_1^2(Y,\mathbf{p};i\mathbf{R})^{-\tau^*} \to \mathcal{J}_0 \oplus \mathcal{K}_0 \oplus L^2(Y,\mathbf{p};i\mathbf{R})^{-\tau^*}\]
    given by
    \[(\frakb + v,c) \mapsto (\mathbf{d}_{\frakb + v}c, (\grad \pertL)(\frakb + v), \mathbf{d}^*_{\frakb}v)\]
    has linearization
    \[(v^J, v^K,c) \mapsto (\mathbf{d}_{\frakb}c, \Hess_{\frakq,\mathfrak{b}}v^K, \mathbf{d}_{\frakb}^*v^J)\]
    which is \emph{invertible regardless whether $\frakb$ is reducible or irreducible} as long as $\frakb$ is nondegnerate.
    It follows that there is an $L^2_1$ neighbourhood $U^Y$ of $(\frakb,0)$ as above and an estimate for all $(\frakb + v,c)$ in $U^Y$:
    \[\|(v,c)\|^2_{L^2_1(Y)} \le
    C\left(\|\mathbf{d}_{\frakb + v}c\|^2 + \|\mathbf{d}^*_{\frakb}v\|^2 + \|(\grad \pertL + v)\|^2\right)\]
    The $L^2_1$ version of the bound in the statement can be proved exactly the same way as \cite[Lem.~13.4.4]{KMbook2007} using the energy identity in \cite[Lem.~13.4.3]{KMbook2007} which is identical in the real case.
    The $L^2_{k+1}$ bound is obtained by standard bootstrapping.
\end{proof}

Let now $Z = \mathbf{R} \times Y$.
Let $[\mathfrak{a}],[\mathfrak{b}]$ be critical points in $\tilde{\mathcal{B}}_k(Y,\tfrr)$.
Suppose $\upgamma_0 \in \mathcal{C}_{k,\text{loc}}$ is a smooth base configuration which agrees near $\pm \infty$ with the corresponding translation invariant configurations $\upgamma_{\mathfrak{a}}$ and $\upgamma_{\mathfrak{b}}$.
By choosing appropriate lifts of $[\mathfrak{a}],[\mathfrak{b}]$ in $\mathcal{C}_k(Y, \tilde{\mR})$, one can arrange $[\check \upgamma_0]$ to be in any given homotopy class $z$.

Define
\begin{align*}
    \mathcal{C}_k(\mathfrak{a},\mathfrak{b})
    &=
    \left\{ \upgamma \in \mathcal{C}_{k,\text{loc}}(Z, \tilde{\mR}) \ \big| \
    \upgamma - \upgamma_0 \in L^2_k(Z;iT^*Z \oplus S^+)
    \right\},\\
    \mathcal{G}_{k+1}(Z,\tfrr)
    &= \left\{ u: Z \to S^1 \ \big| \
    1-u \in L^2_{k+1}(Z;\mathbf{C}), \ u(\tau(x))=\overline{u(x)}, \
    u(\mathbf{p}_I) \equiv 1\right\},\\
    \widetilde{\mathcal{B}}_k(\mathfrak{a},\mathfrak{b}) &=
    \mathcal{C}_k(\mathfrak{a},\mathfrak{b})/\mathcal{G}_{\mathbf{p}_I,k+1}(Z,\tau).
\end{align*}
These $L^2_k$-completions provide an equivalent definition of the moduli space of trajectories on $\mathbf{R} \times Y$:
\begin{theorem}
    Let $\upgamma \in \mathcal{C}_{k,\text{loc}}(Z)$ represent an element $[\upgamma] \in \tilde{N}_z([\mathfrak{a}],[\mathfrak{b}])$.
    Let $\mathfrak{a}=\mathfrak{a}_z,\mathfrak{b}=\mathfrak{b}_z$ be suitable lifts for which path $\upgamma$ lies in homotopy class $z$.
    Then there exists a gauge transformation $u$ in $\mathcal{G}_{k+1,\text{loc}}(Z,\tau,\mathbf{p}_I)$ such that $u(\upgamma)$ belongs to $\mathcal{C}_k(\mathfrak{a},\mathfrak{b})$. If $u,u'$ are two such gauge transformations, then $u^{-1} u'$ belongs to $\mathcal{G}_{k+1}(Z,\tau,\mathbf{a})$.
    The resulting bijection is a homemorphism
    \begin{equation*}
      \wtilde{N}_z([\mathfrak{a}],[\mathfrak{b}]) \to
        \left\{ [\upgamma] \in \tilde{\mathcal{B}}_{k,z}([\mathfrak{a}],[\mathfrak{b}]) \ | \
        \mathfrak{F}_{\mathfrak{q}}(\upgamma) = 0\right\}.
    \end{equation*}
\end{theorem}
\begin{proof}
The theorem can be proved the same way as
%%%
which relies on the following exponential decay result:
\cite[Prop.~13.5.1]{KMbook2007}, which is stated for $\uptau$-blow-ups.
\end{proof}
\begin{prop}
    Suppose $\frakb \in \mathcal{C}_k(Y,\tfrr)$ is a nondegenerate critical point.
    Then there exists a $\delta > 0$ such that:
	for every $\upgamma$ satisfying $\mathfrak{F}(\upgamma) = 0$ on $[0,\infty) \times Y$ and $\lim_{\to} [\upgamma] = \frakb$,
	there exists $t_0 > 0$ for which
    \[
    \pertL(\upgamma(t)) - \pertL(\frakb) \le Ce^{-\delta (t-t_0)},
    \]
    for all $t \ge t_0$, where $C = \pertL(\upgamma(t_0)) - \pertL(\frakb)$. \hfill \qedsymbol
\end{prop}
For $\upgamma = (A_0,\Phi_0)$, let $\mathcal{T}_{j} \to \mathcal{C}_k(\mathfrak{a},\mathfrak{b})$ be the infinite-cyliner versions of the tangent bundle, having fibre
\begin{equation*}
    \mathcal{T}_{j,\upgamma}= L^2_j(Z,\mathbf{p}_I;iT^*Z)^{-\tau^*} \oplus L^2_{j,A_0} (Z;S^+)^{\tfrr}.
\end{equation*}
Let $\mathcal{J}_{j,\upgamma} = \text{Im}(\mathbf{d}_{\upgamma})$ and $\mathcal{K}_{j,\upgamma}=\text{Ker}(\mathbf{d}_{\upgamma}^*)$ the subspaces, where $\mathbf{d}_{\upgamma}$ and $\mathbf{d}^*_{\upgamma}$ are as in ~\eqref{eq:defn_d_gamma} and~\eqref{eq:defn_d_star_gamma}.
The Coulomb slice $\mathcal{S}_{k,\upgamma}$ at $\upgamma = (A_0,\Phi_0)$ can be defined analogously:
\begin{equation*}
    \mathcal{S}_{k,\upgamma}
    =\left\{ (A_0 + a) \in \mathcal{C}_k(\mathfrak{a},\mathfrak{b}) \ \big\vert  \boldsymbol{d}^*_{\upgamma}(a,\phi) = 0\right\}
    =\left\{ (A_0 + a) \in \mathcal{C}_k(\mathfrak{a},\mathfrak{b}) \ \big\vert  -d^*a + i\text{Re}\langle i\Phi_0,\phi\rangle = 0\right\}
\end{equation*}
The following proposition for every $\upgamma \in \mathcal{C}_k(\mathfrak{a},\mathfrak{b})$, there is an open neighbourhood $U \subset \mathcal{S}_{k,\upgamma}$ for which the natural quotient map provides a diffeomorphism
\[U \subset \mathcal{S}_{k,\upgamma} \to \tilde{\mathcal{B}}_k(\mathfrak{a},\mathfrak{b})\]
onto its image.
Hence $\tilde{\mathcal{B}}_{k}([\mathfrak{a}],[\mathfrak{b}])$ is a Hilbert manifold.
(cf.~\cite[Prop.~14.3.3]{KMbook2007}.)

\begin{prop}
The bundles $\mathcal{J}_j$ and $\mathcal{K}_j$ defines a smooth bundle decomposition over $\mathcal{C}_k(\mathfrak{a},\mathfrak{b})$ and $\mathcal{J}_j = \mathcal{J}_0 \cap \mathcal{T}_j$.
\end{prop}

\begin{proof}
    This is the analogue of \cite[Prop.~14.3.2]{KMbook2007} in the $\uptau$-blow-up.
    Following the proof there, use the same integration-by-parts to show
    $\mathbf{d}$ has closed image.
   In the framed case, one needs to instead prove the following operator is an isomorphism:
   \begin{align*}
       \mathbf{d}_{\upgamma}^*\mathbf{d}_{\upgamma}:L^2_{j+1}(Z,\mathbf{p}_I;i\mathbf{R})^{-\tau^*} &\to L^2_{j-1}(Z,\mathbf{p}_I;i\mathbf{R})^{-\tau^*}\\
       \xi &\mapsto \Delta \xi + |\Phi_0|^2\xi
   \end{align*}
   Indeed, since anti-invariant functions on $Y$ have zero mean on each component, Poincar\'{e} inequality implies
   \begin{equation*}
   	\langle \Delta \xi + |\Phi_0|^2\xi, \xi) \rangle_{L^2(Z)} =
	\|d\xi\|^2_{L^2(Z)} + \|\Phi_0 \xi\|^2_{L^2(Z)} \ge \|d_Y \xi\|^2_{L^2(Z)} \ge \lambda \|\xi\|_{L^2(Z)}^2.
   \end{equation*}
The above holds regardless whether $\Phi_0$ is zero or not.
The rest of the proof of \cite[Prop.~14.3.2]{KMbook2007} apply verbatim.
\end{proof}

Let $\mathcal{V}_j \to \mathcal{C}_k([\mathfrak{a}],[\mathfrak{b}])$ be the trivial vector bundle having fibre $\mathcal{V}_{j,\upgamma} = L^2_j(Z;i \mathfrak{su}(S^+))^{\tfrr} \oplus L^2_{j,A_0}(Z;S^-)^{\tfrr}$ for every $\upgamma = (A_0,\Phi_0)$.
As in the case of finite cylinders, the perturbed Seiberg--Witten operator
defines a smooth section
\begin{equation*}
    \mathfrak{F}_{\mathfrak{q}} = \mathfrak{F} + \hat{\mathfrak{q}} \colon \mathcal{C}_k(\mathfrak{a},\mathfrak{b}) \to \mathcal{V}_{k-1}.
\end{equation*}
(This can be deduced from the finite cylinder case following the arguments of~\cite[Lem.~14.4.1]{KMbook2007}.)
To analyze its linearization
\begin{equation}
    \mathcal{D}\mathfrak{F}_{\frakq}: \mathcal{T}_j \to \mathcal{V}_{j-1},
\end{equation}
at $\upgamma_0$, let $\breve{\upgamma}_0(t)=(B_0(t),\phi_0(t))$ be the corresponding path $\mathbf{R} \to \mathcal{C}_{k-1}(Y, \tilde{\mR})$.
Let $(a,\psi)$ be an element of $\mathcal{T}_j$ and express $a$ as $b + cdt$ where $b$ is in temporal gauge, so that $(b(t),\psi(t),c(t))$ is a path of sections of
\begin{equation*}
    \mathcal{T}(Y, \tilde{\mR}) \oplus L^2(Y,\mathbf{p};i\mathbf{R})^{-\tau^*}
\end{equation*}
In this path notation, the linearized perturbed Seiberg--Witten operator $\mathfrak{F}_{\mathfrak{q}}$ becomes
\begin{equation*}
    (V,c) \mapsto
    \frac{d}{dt}V + \mathcal{D}(\grad \pertL)(V) + \mathbf{d}_{\upgamma_0(t)}(c)
\end{equation*}
Moreover, the Coulomb-type gauge-fixing condition $\mathbf{d}_{\upgamma_0}^*(V,c) = 0$
becomes:
\begin{equation*}
    \frac{d}{dt}c + \mathbf{d}^*_{\upgamma_0(t)}(V) = 0.
\end{equation*}
Define the coupled operator:
\begin{align*}
    \mathcal{Q}_{\upgamma_0}
    &=
    \mathcal{D}_{\upgamma_0}(\mathfrak{F}_{\frakq}) \oplus \mathbf{d}^*_{\upgamma_0}
    \\
    \mathcal{Q}_{\upgamma_0}:\mathcal{T}_{j,\upgamma_0}
    &\to
    \mathcal{V}_{j-1,\upgamma_0} \oplus L^2_{j-1}(Z,\mathbf{p}_I;i\mathbf{R})^{-\upiota^*}
\end{align*}
which is, in the path notation,
\begin{equation}
\label{eq:Q_in_path_notat}
    (V,c) \mapsto \frac{d}{dt}(V,c) + L_{\upgamma_0}(V,c)
\end{equation}
where
\begin{equation*}
    L_{\mathfrak{a}}:
    \mathcal{T}_{j,\mathfrak{a}}(Y,\tilde{\mR}) \oplus L^2_j(Y,\mathbf{p};i\mathbf{R})^{-\upiota^*} \to
    \mathcal{T}_{j-1,\mathfrak{a}}(Y,\tilde{\mR}) \oplus L^2_{j-1}(Y,\mathbf{p};i\mathbf{R})^{-\upiota^*}
\end{equation*}
is the extended Hessian
\begin{equation*}
    L_{\mathfrak{a}} =
    \begin{bmatrix}
        \mathcal{D}_{\mathfrak{a}}(\grad \pertL) & \mathbf{d}_{\mathfrak{a}} \\
        \mathbf{d}^*_{\mathfrak{a}} & 0
    \end{bmatrix}.
\end{equation*}
\begin{theorem}
    Suppose $\mathfrak{a},\mathfrak{b}$ are non-degenerate critical points.
  For every $\upgamma_0$ in $\mathcal{C}_k(\mathfrak{a},\mathfrak{b})$, the linear operator
    \begin{equation*}
        \mathcal{Q}_{\upgamma_0}: \mathcal{T}_{j,\upgamma}(Z,\tilde{\mR}) \to
        \mathcal{V}_{j-1,\upgamma_0}(Z,\tilde{\mR}) \oplus L_{j-1}^2(Z,\mathbf{p}_I;i\mathbf{R})^{-\tau^*}
    \end{equation*}
    is Fredholm for all $j$ in the range $1 \le j \le k$, whose index is independent of $j$; it satisfies a G\aa rding inequality
    \begin{equation*}
        \|u\|_{L^2_j} \le C_1\|\mathcal{Q}_{\upgamma_0} u\|_{L^2_{j-1}} + C_2\|u\|_{L^2_{j-1}}.
    \end{equation*}
 In particular, the restriction $\mathcal{D}\mathfrak{F}_{\frakq}:
\mathcal{K}_{j,\upgamma} \to
\mathcal{V}_{j-1,\upgamma}$ is Fredholm and has the same index as $\mathcal{Q}_{\upgamma}$.
Lastly,
\begin{equation*}
    \mathsf{ind}(\mathcal{Q}_{\upgamma}) = \mathsf{SF} \left\{ \widehat{\Hess}_{\mathfrak{q},\breve{\upgamma}(t)} \right\}.
\end{equation*}
\end{theorem}
\begin{proof}
    This theorem can be deduced from the proof of~\cite[Thm.~14.4.2]{KMbook2007} in the $\uptau$-blowup.
    In fact, the framed version is technically simpler since $\mathcal{V}_j$ is a trivial Hilbert vector bundle.
    Besides this, the two other ingredients of the proof have their real counterparts: hyperbolicity of the extended Hessian at a non-degenerate critical point (Lemma~\ref{lem:nondeg_Hess_suj}) and Fredholmness of~\eqref{eq:Q_in_path_notat} (\cite[Prop.~8.10]{ljk2022}).
    The last part of the statement can be proved the same way as~\cite[Prop.~14.4.3]{KMbook2007}.
\
\end{proof}
\begin{defn}
    A moduli space $\tilde{N}_z([\mathfrak{a}],[\mathfrak{b])}$ is \emph{regular at $\upgamma$} if $\mathcal{Q}_{\upgamma}$ is surjective; if $\tilde{N}_z([\mathfrak{a}],[\mathfrak{b])}$ is regular at every $\upgamma$, then $\tilde{N}_z([\mathfrak{a}],[\mathfrak{b])}$ is \emph{regular} as a moduli space.
\end{defn}

\begin{defn}
    Given two critical points $\mathfrak{a},\mathfrak{b}$ in $\mathcal{C}_k(Y,\tfrr)$, the \emph{relative grading} $\gr(\mathfrak{a},\mathfrak{b})$ is the index of $\mathcal{Q}_{\upgamma}$, where $\upgamma$ is any element of $\mathcal{C}_k(\mathfrak{a},\mathfrak{b})$.
    If $[\mathfrak{a}],[\mathfrak{b]}$ are gauge orbits of $\mathfrak{a},\mathfrak{b}$ in the relative homotopy class of $z = \pi \circ \check \upgamma$, then $\gr(\mathfrak{a},\mathfrak{b})$ will also be written as $\gr_z([\mathfrak{a}],[\mathfrak{b}])$.
\end{defn}
By the spectral flow interpretation, the relative grading satisfies additivity: if $\mathfrak{a},\mathfrak{b},\mathfrak{c}$ are critical points, then
\begin{equation*}
    \gr(\mathfrak{a},\mathfrak{c}) =
    \gr(\mathfrak{a},\mathfrak{b}) + \gr(\mathfrak{b},\mathfrak{c}).
\end{equation*}
Moreover, the following Lemma about dependence of $\gr_z$ on $z$ is a special case of~\cite[Lem.~11.7]{ljk2022}.
\begin{lem}
    \label{lem:loop_and_12}
Let $z_u$ be a closed loop based at $[\mathfrak{a}]$ in $\tilde{\mathcal{B}}_k(Y,\tfrr)$, and view $[u]$ as the homotopy class of $u:Y \to S^1$.
Then
\begin{equation*}
    \gr_{z_u}(\mathfrak{a},\mathfrak{a}) = \frac{1}{2}([u] \cup c_1(S))[Y],
\end{equation*}
\end{lem}
The factor $1/2$ is special to the real case.
(Compare with the ordinary formula~\cite[Lem.~14.4.6]{KMbook2007}.)
The homotopy classes $z$ lives in $\pi_1(\tilde{B}_{\mathbf{p}}) \cong \Upgamma(Y,\tau,\mathbf{p})$.
(Compare with~\cite[\S 5.7]{ljk2022}, where $\pi_1(\tilde{B}_{\mathbf{p}})$ is isomorphic to $ H^1(Y;\mathbf{Z})^{-\tau^*}$.)
\begin{prop}
    If $\wtilde{N}_z((\mathfrak{a}],[\mathfrak{b}])$ is regular at $[\upgamma]$, then $\wtilde{N}_z([\mathfrak{a}],[\mathfrak{b}])$ is a smooth manifold of dimension $\gr_z ([\mathfrak{a}],[\mathfrak{b}])$ in a neighbourhood of $[\upgamma]$. \hfill \qedsymbol
\end{prop}
\subsection{Transversality of trajectory spaces}
\label{sec:transversality-traj}
Recall that a large Banach space of tame perturbations $\mathcal{P}$ has been chosen.
The following is the main result of this subsection, adapted from \cite[Thm.~15.1.1 \& Prop.~15.1.3]{KMbook2007}.
\begin{theorem}
    \label{thm:transversality_cylinder}
    There is a $\mathfrak{q} \in \mathcal{P}$ such that (i) all the critical points $\mathfrak{a} \in \mathcal{C}_k(Y, \tilde{\mR})$ are nondegenerate, and (ii) for each pair of critical points $(\mathfrak{a},\mathfrak{b})$,
	and each relative homotopy class $z \in \pi_1(\tcalB_k(Y, \tfrr),[\mathfrak{a}],[\mathfrak{b}])$ the moduli space $\tilde{N}_z([\mathfrak{a}],[\mathfrak{b}])$ is regular.
\end{theorem}
\begin{proof}
	Choose $\mathfrak q_0\in\mathcal P$ for which all critical
points are nondegenerate.
    Apply Proposition~\ref{prop:enough_embeddings} to choose a map $p_0 \colon \tilde{\mathcal{B}}_k(Y,\tfrr) \to \mathbf{R}^n \times \mathbf{T}^t \times \mathbf{R}^m$ that embeds the set critical points (which is a finite set).
    For each critical point $[\mathfrak{a}]$,
	choose an open neighbourhood $\mathcal{O}_{[\mathfrak{a}]}$ of $[\mathfrak{a}]$, and
	denote
	\begin{equation*}
        \mathcal{O} = \bigcup_{[\mathfrak{a}]} \mathcal{O}_{[\mathfrak{a}]} \subset \wtilde{\mathcal{B}}(Y,\tfrr).
    \end{equation*}
    By choose $\mathcal{O}_{[\mathfrak{a}]}$ small,
	assume that the image $p_0(\mathcal{O}_{[\mathfrak{a}]})$ have disjoint closures and that no essential loop based at any $p_0([\mathfrak{a}])$ is contained in $\overline{p_0(\mathcal{O})}$.
    Let $\mathcal{P}_{\mathcal{O}} \subset \mathcal{P}$ be the closed affine subspace of perturbations:
    \begin{equation*}
        \mathcal{P}_{\mathcal{O}} = \{ \mathfrak{q} \in \mathcal{P} \ \vert \ \mathfrak{q}|_{\mathcal{O}} = \mathfrak{q}_0|_{\mathcal{O}}\}.
    \end{equation*}
    By properness of perturbed gradients, there exists an open neighbourhood of $\mathfrak{q}_0$ in $\mathcal{P}_{\mathcal{O}}$ such that for all $\mathfrak{q}$ in this neighbourhood, the perturbed vector field $\grad \pertL_{\mathfrak{q}}$ has no critical points outside $\mathcal{O}$ (cf.~\cite[Lem.~15.1.2]{KMbook2007}).

    It remains to show that the set of perturbations in $\mathcal{P}_{\mathcal{O}}$ that satisfies (ii) of the theorem for all $(\mathfrak{a},\mathfrak{b})$ whose images belong to $\mathcal{O} \subset \tilde{\mathcal{B}}_k(Y,\tfrr)$ is a residual subset of $\mathcal{P}_{\mathcal{O}}$.
    To this end, consider the parametrized moduli space:
    \begin{equation*}
        \mathfrak{M}_z([\mathfrak{a}],[\mathfrak{b}]) \subset \widetilde{\mathcal{B}}_{k,z}([\mathfrak{a}],[\mathfrak{b}]) \times \mathcal{P}_{\mathcal{O}},
    \end{equation*}
    as the zero set of
    \begin{equation*}
        \mathfrak{W} \colon \mathcal{C}_k(\mathfrak{a},\mathfrak{b}) \times \mathcal{P}_{\mathcal{O}} \to \mathcal{V}_{k-1}(Z), \quad
        (\upgamma,\mathfrak{q}) \mapsto \mathfrak{F}_{\mathfrak{q}}(\upgamma).
    \end{equation*}

    In order to apply the Sard--Smale theorem, one needs to verify the derivative $\mathcal{D}\mathfrak{W}$ is surjective at all $(\upgamma,\mathfrak{q})$ in $\mathfrak{W}^{-1}(0)$.
    Let $(\upgamma,\mathfrak{q})$ be a zero of $\mathfrak{W}$ and suppose the corresponding path $\breve{\upgamma}$ in $\mathcal{C}_k(Y,\tfrr)$ has non-constant image in $\tilde{\mathcal{B}}_k(Y,\tfrr)$.
    In other words, either $[\mathfrak{a}] \ne [\mathfrak{b}]$ or when $[\mathfrak{a}] = [\mathfrak{b}]$, the homotopy class $z$ is non-trivial.
    Together with the hypothesis on $\overline{p_0(\mathcal{O})}$, this assumption ensures there is an open interval $J \subset \mathbf{R}$ for which
    \begin{equation*}
        p_0(\breve{\upgamma}(J)) \cap \overline{p_0(\mathcal{O})} = \emptyset.
    \end{equation*}

	By unique continuation for the perturbed equations, after shrinking $J$, the path $\breve{\upgamma}$ embeds the coloure $\bar J$ into $\tilde{\mathcal B}_k(Y,\tfrr)$.
\begin{rem}
In the framed setting, this argument applies uniformly to both irreducible and reducible trajectories.
\end{rem}

	Suppose for contradiction that $\mathcal D_{(\upgamma,\mathfrak q)}\mathfrak W$ is not surjective;
	let $0\ne V\in\mathcal V_{0,\upgamma}(Z,\tfrr)$ be an element of the $L^2$ orthogonal complement of the image of $\mathcal D_{(\upgamma,\mathfrak q)}\mathfrak W$.
    Hence $(V,0) \in \mathcal{V}_{0,\upgamma}(Z) \times L^2(Z,\mathbf{p}_I;i\mathbf{R})^{-\tau^*}$ is $L^2$-orthogonal to the image of $\mathcal{Q}_{\upgamma}$.
    By a similar unique continuation argument as in the proof of \cite[Prop.~15.1.3]{KMbook2007}, the restriction of $V$ to $J$ is nonzero.

    Using the isomorphism
    \begin{equation*}
        i\mathfrak{su}(S^+)\oplus S^- \to iT^*(Y) \oplus S,
    \end{equation*}
    let $\breve V$ be the corresponding $L^2_1$ section of $\mathcal{T}_0(Y,\tfrr)$ along the path $\breve{\upgamma}(t)$.
    The following lemma can be proved using the same integration-by-parts argument as \cite[Lem.~15.1.4]{KMbook2007} (stated for the $\uptau$-blowup):
    \begin{lem}
        \label{lem:breveV_non_zero}
        For all $t \in \mathbf{R}$, the element $\breve{V}(t)$ in $\mathcal{T}_0$ belongs to the orthogonal complement of the tangent space to the $\mathcal{G}_{k+1,\mathbf{p}}(Y,\tau)$-orbit through $\breve{\upgamma}(t)$, with respect to the $L^2$ inner product in $\mathcal{T}_0(Y,\tfrr)$.
        In particular, the image of $\breve{V}(t)$ in $[\mathcal{T}_0]$ is nonzero for all $t$.\qed
    \end{lem}

	To reach a contradiction to $V$ being orthogonal to the image of $\mathcal{D}\mathfrak{W}$, choose $t_0 \in J$,
	and seek a cylinder function $f$ for which $\delta \mathfrak{q} =\grad f$ satisfies
	\begin{equation*}
		\langle \delta \mathfrak{q}(t), \breve{V}(t)\rangle_{\mathcal{T}_{0,\breve{\upgamma}}(t)} \ge 0
	\end{equation*}
	 for all $t$, with strict inequality at some $t_0$.
	 Apply Proposition~\ref{prop:enough_embeddings} to find a large collection of $c_i$ and $\Upsilon_j$ extending $p_0$
	 to define $p \colon \mathcal{C}_k(Y,\tfrr) \to \mathbf{R}^{n'} \times \mathbf{T}^t \times \mathbf{R}^{m'}$
	 that embeds $\breve{\upgamma}(S)$, where
	 \begin{equation*}
		 S = \{ t \in \mathbf{R} \ \vert \ p(\upgamma(t)) \not\in p(\mathcal{O})^+ \}
	 \end{equation*}
	 and $p(\mathcal{O})^+$ is an open neighbourhood of $\overline{p(\mathcal{O})}$ whose closure is disjoint from $p(\upgamma(\bar J))$.
     By Lemma~\ref{lem:breveV_non_zero} and the last part of Proposition~\ref{prop:enough_embeddings},
	 one may assume, after enlarging the collection, that $p_*(\breve{V})$ along $\bar J$ is non-zero.
     One may also assume, by unique continuition and Lemma~\ref{lem:breveV_non_zero}, that $p_*(\breve{V}(t_0))$ is not tangent to the arc $p(\breve{\upgamma}(\bar{J}))$.
	 Proposition~\ref{prop:enough_embeddings}\footnote{
		 Here we use that the $\mathcal{G}_{\mathbf{p}}$-invariant framed cylinder functions need not be even in the spinor variable.
		 Indeed, at a reducible configuration $(B,0)$,
		 every even $f$ satisfies $(\mathcal Df)_{(B,0)}(0,\psi)=0$,
		 which cannot detect a purely spinorial element of the cokernel.}
	 provides a cylinder function $f = g \circ p$ arising from some
	 $g\colon \mathbf{R}^{n'} \times \mathbf{T}^t \times \mathbf{R}^{m'} \to \mathbf{R}$
	 that vanishes along $p(\breve{\upgamma}(\bar J))$ and
	 satisfies $(\mathcal{D}f)(\breve{V}) \ge 0$ for all $t \in J$ with strict inequality at $t_0$.
     By multiplying $g$ with a cut-off function supported near $p(\breve{\upgamma}(t_0))$,
	 one arranges that $g$ vanishes on $p(\mathcal{O})^+$, and that
	 $(\mathcal{D}f)(\breve{V}) = 0$ for $t \in \mathbf{R}\setminus J$.
	 Finally, by the denseness assumption in Definition~\ref{defn:large_banach_space}, $f$ can be taken so that $\delta \mathfrak{q} = \grad f$.

    The remaining case, in which $z$ is trivial, so that $\upgamma$ is the constant trajectory,
	is a consequence of the non-degeneracy of the critical point $[\mathfrak{a}]$.
    Indeed, an element $V$ in the cokernel of $\mathcal{D}_{\upgamma}\mathfrak{F}_{\mathfrak{q}}$ gives rise to a $\breve{V}(t)$
    which satisfies translation-invariant equation of the form
     \begin{equation*}
         \left(\frac{d}{dt}+L_0+h\right)\breve{V}(t) = 0.
     \end{equation*}
     By the non-degeneracy assumption of $\mathfrak{a}$, the 3-dimensional operator $(L_0+h)$ is hyperbolic.
     But by \cite[Prop.~14.1.2]{KMbook2007}, the 4-dimensional operator $(d/dt + L_0 + h) \colon L^2_{j+1} \to L^2_j$ has trivial kernel and thus $\breve{V}(t) = 0$.
\end{proof}
\subsection{Compactness of trajectory spaces}
The trajectory spaces $\wtilde{N}_z([\mathfrak{a}],[\mathfrak{b}])$ are compactified by adding broken trajectories.
\begin{defn}
    A trajectory is \emph{nontrivial} if it is not invariant under the $\mathbf{R}$-translation.
  An \emph{unparametrized trajectory} is an equivalence class of nontrivial trajectories in $\tilde{N}_z([\mathfrak{a}],[\mathfrak{b}])$ under translation.
  Denote by $\breve{N}_z([\mathfrak{a}],[\mathfrak{b}])$ the space of unparametrized trajectories.
\end{defn}
\begin{defn}
    An \emph{unparametrized broken trajectory} joining $[\mathfrak{a}]$ to $[\mathfrak{b}]$ is an $(n+1)$-tuple of critical points
    \[
    [\mathfrak{a}_0],[\mathfrak{a}_1],\dots,[\mathfrak{a}_n]\]
    with $[\mathfrak{a}_0]=[\mathfrak{a}]$ and $[\mathfrak{a}_n]=[\mathfrak{b}]$, together with an unparametrized $[\check \upgamma_i]$ in $\breve N_{z_i}([\mathfrak{a}_{i-1}],[\mathfrak{a}_i])$ for $1 \le i \le n$.
    The homotopy class of the broken trajectory is obtained by concatenating representatives of the classes $z_i$.
\end{defn}
The set $\breve{N}_z([\mathfrak{a}],[\mathfrak{b}])$ will be topologized routinely (cf.~\cite[p.276]{KMbook2007}) by introducing a neighbourhood base.
With this topology, the ``Proof of compactness downstairs'' that occupies \cite[\S 16.2]{KMbook2007} translates (cf. \cite[\S 9]{ljk2022}) to the proof of the following proposition in the framed case.
\begin{theorem}
    \label{thm:traj_space_compact}
    The space of unparametrized broken trajectories $\breve N_z^+([\mathfrak{a}],[\mathfrak{b}])$ is compact.
    \hfill \qedsymbol
\end{theorem}
\begin{theorem}
    \label{thm:energy_bound_traj_space_compact}
    For any $C > 0$ and any $[\mathfrak{a}],[\mathfrak{b}]$, there are only finitely many $z$ with energy $\mathcal{E}_{\mathfrak{q}}(z) \le C$ for which $N_z^+([\mathfrak{a}],[\mathfrak{b}])$ is nonempty.
    \hfill \qedsymbol
\end{theorem}
Assume perturbations are chosen so that moduli spaces a regular.
The framed compactifications $\breve N_z^+([\mathfrak{a}],[\mathfrak{b}])$ are ``stratified'':
A space $N$ is a \emph{$d$-dimensional space stratified by manifolds} if there are closed subsets
\begin{equation*}
    N = N^d \supset N^{d-1} \supset \cdots \supset N^0 \supset N^{-1} = \emptyset.
\end{equation*}
such that $N \ne N^{d-1}$ and each $N^e \setminus N^{e-1}$ (for $0 \le e \le d$) is homeomorphic to manifold of dimension $e$, possibly empty.
The difference $N^e \setminus N^{e-1}$ is the \emph{$e$-dimensional stratum}.
Compare the following result with \cite[Prop.~16.5.2]{KMbook2007} and \cite[Prop.~9.11]{ljk2022}.
\begin{prop}
    \label{prop:traj_space_compact_strata}
    Suppose that $\tilde{N}_z([\mathfrak{a}],[\mathfrak{b}])$ is non-empty and of dimension $d$.
    Then $\breve N^+_z([\mathfrak{a}],[\mathfrak{b}])$ is a $(d-1)$-dimensional space stratified by manifolds.
    The $(d-\ell)$-dimensional stratum consists of spaces of the form
    \begin{equation*}
        \breve N_{z_1}([\mathfrak{a}_0],[\mathfrak{a}_1]) \times \cdots
        \breve N_{z_{\ell}}([\mathfrak{a}_{\ell-1}],[\mathfrak{a}_{\ell}]).
    \end{equation*}
\end{prop}
\subsection{Gluing trajectories}
This subsection analyzes the boundary strata of framed moduli spaces.
In short, gluing in the framed setup is always unobstructed and $\delta$-structures (cf.~\cite[Defn.~19.5.3]{KMbook2007}) are no longer needed due to the absence of boundary-obstructed cases.
The real version of the gauge-theoretic gluing result \cite[\S 18.3]{KMbook2007} was adapted in \cite[\S 10]{ljk2022}.

\begin{theorem}
    Given a moduli space $N([\mathfrak{a}_0],[\mathfrak{a}_n])$ and a stratum of the form
    \begin{equation}
        \label{eq:stratum_in_compac}
        \prod_{i=1}^n \breve N([\mathfrak{a}_{i-1}],[\mathfrak{a}_i]) \subset N^+([\mathfrak{a}_{i-1}],[\mathfrak{a}_i])
    \end{equation}
    there exists a neighbourhood $\breve W$ of \eqref{eq:stratum_in_compac} and a map $\mathbf{S}\colon \breve W \to (0,\infty]^{n-1}$ such that $S^{-1}(\infty,\dots,\infty)$ is precisely \eqref{eq:stratum_in_compac} and $\mathbf{S}$ is a topological submersion (cf.~\cite[Defn.~19.2.7]{KMbook2007}).
    \hfill \qedsymbol
\end{theorem}
\begin{theorem}
    Suppose $\breve N^+_z([\mathfrak{a}],[\mathfrak{b}])$ is a $(d-1)$-dimensional space stratified by manifolds.
    Let $N' \subset \check N^+([\mathfrak{a}],[\mathfrak{b}])$ be any component of the codimension-$1$ stratum.
    Then along $M'$ the moduli space $\check N^+([\mathfrak{a}],[\mathfrak{b}])$ is a $C^0$-manifold with boundary.
    \hfill \qedsymbol
\end{theorem}
\subsection{Construction of orientation sets}
This section uses the same title as \cite[\S 20.3]{KMbook2007} but the main goal is to introduce a Fredholm operator of index that shows up crucially in the construction of relative gradings.

Let $I = [t_1,t_2]$ be a finite interval.
Let $\mathfrak{a}_1,\mathfrak{a}_2$ be any two configurations (not necessarily critical points) in $\mathcal{C}_k(Y,\tfrr)$ and $\mathfrak{q}_1,\mathfrak{q}_2$ be two tame perturbations.
Consider the space $\mathcal{C}$ of all pairs $(\upgamma,\mathfrak{p})$ where
\begin{enumerate}[leftmargin=*,noitemsep]
    \item $\upgamma$ is a configuration in $\mathcal{C}(I \times Y,\tfrr)$ whose restriction to the end ${t_i} \times Y$ is gauge equivalent to $\mathfrak{a}_i$, for $i=1,2$; and
    \item $\mathfrak{p}$ is a continuous path in the Banach space $\mathcal{P}$ of tame perturbations, with $\mathfrak{p}(t_i) = \mathfrak{p}_i$, for $i=1,2$.
\end{enumerate}
Over the finite cylinder $Z = I \times Y$, the linearization of the Seiberg--Witten operator $\mathcal{Q}_{\upgamma,\mathfrak{p}}$ gives a map
\begin{equation*}
    \mathcal{Q}_{\upgamma,\mathfrak{p}}: \mathcal{E} \to \mathcal{F}
\end{equation*}
where
\begin{equation*}
    \mathcal{E} = \mathcal{T}_{1,\upgamma}(I \times Y, \tfrr),\qquad
    \mathcal{F} = \mathcal{V}_{0,\upgamma}(I \times Y, \tfrr) \oplus
    L^2(I \times Y, I \times \mathbf{p}; i\mathbf{R})^{-\tau^*}.
\end{equation*}
To make the problem over the finite cylinder Fredholm, impose boundary conditions as follows:
At ${t_i} \times Y$, decompose
\begin{equation*}
    \mathcal{T}_{1/2,\mathfrak{a}_i}(Y) \oplus L^2_{1/2}(Y, \mathbf{p};i\mathbf{R}) =
    \mathcal{J}_{1/2,\mathfrak{a}_i} \oplus \mathcal{K}_{1/2,\mathfrak{a}_i} \oplus
    L^2_{1/2}(Y, \mathbf{p};i\mathbf{R})^{-\tau^*}.
\end{equation*}
Denote
\begin{equation*}
    H_i^{\pm}
    = {0} \oplus \mathcal{K}^{\pm}_{1/2,\mathfrak{a}_i} \oplus L^2_{1/2}(Y, \mathbf{p};i\mathbf{R})^{-\tau^*},
\end{equation*}
using the spectral subspaces of $\Hess_{\mathfrak{q}_i,\mathfrak{a}_i}$ if the operator $\Hess_{\mathfrak{q}_i,\mathfrak{a}_i}$ is hyperbolic, or $\Hess_{\mathfrak{q}_i,\mathfrak{a}_i} - \epsilon$ otherwise.
Let $\Pi^-_{Y,\mathfrak{a}_i}$ and $\Pi^+_{Y,\mathfrak{a}_i}$ be the projections to $H^-_i$ and $H^+_i$ with kernels
\begin{equation*}
    \ker(\Pi^-_{Y,\mathfrak{a}_i}) =
    \mathcal{J}_{1/2,\mathfrak{a}_i} \oplus \mathcal{K}^+_{1/2,\mathfrak{a}_i} \oplus {0}, \quad
    \ker(\Pi^+_{Y,\mathfrak{a}_i}) =
    \mathcal{J}_{1/2,\mathfrak{a}_i} \oplus \mathcal{K}^-_{1/2,\mathfrak{a}_i} \oplus {0}.
\end{equation*}
Write
\begin{equation*}
    \Pi^{\pm}_i = \Pi^{\pm}_{Y,\mathfrak{a}_i} \circ r_i:
    \mathcal{E} \to H^{\pm}_i,
\end{equation*}
where $r_i$ is the restriction map onto ${t_i} \times Y$.
One arrives at the the following Atiyah--Patodi--Singer-type Fredholm operator:
\begin{align}
    \label{eq:grading_op_P}
    P_{\upgamma,\mathfrak{p}}
    &= (\mathcal{Q}_{\upgamma,\mathfrak{p}},-\Pi^+_1,\Pi^-_2),\\
    \mathcal{E}
    &\to \mathcal{F} \oplus H^+_1 \oplus H_2^-.
\end{align}
Recall $\mathcal{Q}_{\upgamma,\mathfrak{p}(t)} = \frac{d}{dt} + L(t)$,
where $L(t) = \widehat{\Hess}_{\upgamma(t),\mathfrak{p}(t)}$ is the extended Hessian~\eqref{eqn:ext_hessian_1th}, acting on $\mathcal{T}_{1,\upgamma(t)} \oplus L^2_1(Y,\mathbf{p};i\mathbf{R})^{-\tau^*}$.
\medskip

Now, suppose $\upgamma(t) = (B(t),0)$ is reducible and $\mathfrak{p}(t) = 0$ for all $t \in I$.
Decompose the domain of $L(t)$ as the sum
\begin{equation*}
    i\Omega^0_{-\tau^*} \oplus i\left(d\Omega^0_{-\tau^*}\right) \oplus i\left(\ker(d^*: \Omega^1 \to \Omega^0)\right)^{-\tau^*}
    \oplus \Gamma(S)^{\tfrr}.
\end{equation*}
Under this decomposition and by scaling the ``$h$'' terms in \eqref{eqn:ext_hessian_1th} to zero, $L(t)$ can be deformed into the block form:
\begin{equation*}
    \begin{bmatrix}
        0 & -d^* & 0 & 0\\
        -d & 0 & 0 & 0\\
        0 & 0 & *d & 0\\
        0 & 0 & 0 & D_{B(t)}
    \end{bmatrix}
\end{equation*}

In summary, the operator $P = (\mathcal{Q}_{\upgamma},-\Pi^+_1,\Pi^-_2)$ along a reducible path $\upgamma(t)$ is homotopic to an operator $P' = (\mathcal{Q}'_{\upgamma},-\Pi^+_1,\Pi^-_2)$, where $\mathcal{Q}'_{\upgamma}$ is the direct sum of the following pieces:
\begin{enumerate}[leftmargin=*]
    \item the operator
    \begin{equation*}
        \frac{d}{dt} +
        \begin{bmatrix}
            0 & -d^*\\
            -d & 0
        \end{bmatrix}
    \end{equation*}
    where for fixed $t$, the matrix is interpreted as an operator on pairs $(c,b)$, where $c$ is an imaginary-valued function skew-invariant under $\tau$ and relative to $\mathbf{p}$, and $b$ is an exact 1-form;
    \item the operator
    \begin{equation*}
        \frac{d}{dt}+*d
    \end{equation*}
    where $*d$ acts on the coclosed $\tau$-skew-invariant 1-forms on $Y$;
    \item the operator
    \begin{equation}
    \label{eq:DBt_in_grading_decomp}
        \frac{d}{dt} + D_{B(t)}(t)
    \end{equation}
\end{enumerate}
The first two summands are invertible.
The third operator is a real Dirac operator acting on the real subspace of spinors:
\begin{equation*}
    D_{B(t)} \colon \Gamma(S)^{\tfrr} \to \Gamma(S)^{\tfrr},
\end{equation*}
which, unlike the complex-linear Dirac operator,
the real part of the 4-dimensional operator \eqref{eq:DBt_in_grading_decomp} may have odd real index.
\section{Framed real monopole Floer homology}
\label{sec:Floer_homology}
\subsection{The framed chain complex}
\begin{defn}
    A tame perturbation $\mathfrak{q}$ is an \emph{admissible perturbation} for $(Y,\tau, g,\mathfrak{s},\tfrr)$ if all critical points of $\pertL$ are non-degenerate, all moduli spaces are regular, and there are no reducibles unless $c_1(\mathfrak{s})$ is torsion. \end{defn}
\begin{defn}
    Let $\mathfrak{C} \subset \tcalB_k(Y,\tfrr)$ be the set of critical points of $\pertL$.
    The \emph{framed real monopole Floer homology} is the homology $H(\tilde{C},\tilde{\del})$ of the
    \emph{framed real monopole Floer chain complex} $(\tilde{C},\tilde{\del})$, given by
\begin{equation*}
	\wtilde{C} = \bigoplus_{[\mathfrak{a}] \in \mathfrak{C}} \mathbf{F} [\mathfrak{a}], \quad \tilde{\del} [\mathfrak{a}] = \sum_{[\mathfrak{b}] \in \mathfrak{C}} \sum_{\substack{z \in \pi_1(\tilde{\mathcal{B}}_k [\mathfrak{a}],[\mathfrak{b}]),
	\\ \gr_z([\mathfrak{a}],[\mathfrak{b}])= 1}} \# \breve{N}_z([\mathfrak{a}],[\mathfrak{b}])[\mathfrak{b}].
\end{equation*}
\end{defn}
The well-definedness of $(\tilde{C},\tilde{\del})$ as a chain complex is the consequence of the following two results.
First, by Theorem~\ref{thm:traj_space_compact} and Theorem~\ref{thm:energy_bound_traj_space_compact}:
\begin{lem}
    If $\breve{N}_z([\mathfrak{a}],[\mathfrak{b}])$ has dimension zero, then it is a finite set.
	For fixed endpoints $([\mathfrak{a}],[\mathfrak{b}])$, there are only finitely many $z$ with $\gr_z([\mathfrak a],[\mathfrak b])=1$ for which $\breve{N}_z([\mathfrak{a}],[\mathfrak{b}])$ is non-empty.
    \hfill \qedsymbol
\end{lem}

\begin{prop}
    The square $(\tilde{\del})^2 = 0$, i.e. $\tilde{\del}$ defines a differential.
\end{prop}
\begin{proof}
    Fix a pair $([\mathfrak{a}],[\mathfrak{b}])$ of crticial points.
    The compactification $\breve{N}^+_z([\mathfrak{a}],[\mathfrak{b}])$ of a moduli space $\breve{N}_z([\mathfrak{a}],[\mathfrak{b}])$ for which $\gr_z([\mathfrak{a}],[\mathfrak{b}]) = 2$ is a 1-dimensional $C^0$-manifold, with boundary given by $0$-dimensional strata of the form
    \begin{equation*}
        \breve{N}_{z_1}([\mathfrak{a}],[\mathfrak{a}_1]) \times \breve{N}_{z_2}([\mathfrak{a}_1],[\mathfrak{b}]).
    \end{equation*}
    Varying over all such $z$, the sum of the boundary contributions is precisely the entry of $\tilde{\del}^2$ from $[\mathfrak{a}]$ to $[\mathfrak{b}]$, which is zero modulo two.
\end{proof}

\begin{lem}\label{lem:total_critical_set_finite}
	Given perturbations with a uniform bound $m_2$ as in Definition~\ref{defn:tame_pert}(iv) over the set of $\rrsc$ strcutures,
	only finitely many $\rrsc$ strcutures admit critical points.
	Thus the total set of critical points modulo $\mathcal{G}_{\mathbf{p}}$ is finite.
\end{lem}

\begin{proof}
	Following the argument of \cite[Lem.~10.7.5]{KMbook2007}, the uniform bound $m_2$ gives a uniform bound on $\|F_{B^t}\|$ at critical points.
	Hence only finitely many underlying $\spinc$ structures occur.
	By the classification in Proposition~\ref{prop:rrsc_classify},
	each underlying $\spinc$ structure admits finitely many compatible relative real structures.
	And for each such $\rrsc$ structure, the critical set in $\widetilde{\mathcal{B}}$ is compact.
	Because the critical points are isolated by nondegeneracy, this compact set is finite.
\end{proof}

Since $(-1) \not\in \mathcal{G}_{\mathbf{p}}$,
let $\sigma[B,\Psi]=[B,-\Psi]$.
The $\sigma$-fixed set is precisely the reducible configuration.
Write $f^\sigma=f\circ\sigma^{-1}$ and $\mathfrak q^\sigma=\grad f^\sigma$
(not to be confused with the $\sigma$-blowup).
There is a bijection of trajectory spaces
\[
 \sigma_* \colon \widetilde N_{\mathfrak q}([\mathfrak a],[\mathfrak b])
 \longrightarrow
 \widetilde N_{\mathfrak q^\sigma}
       (\sigma[\mathfrak a],\sigma[\mathfrak b]).
\]
Thus $\sigma$ defines an isomorphism between chain complexes for two possibly different perturbations.
However, it is an involution the same chain complex only when the $\mathfrak q^\sigma=\mathfrak q$.
\subsection{Grading}
Let $\mathbf{J}(Y,\mathfrak{s},\tfrr)$ be the quotient set
\[
\left(\widetilde{\mathcal{B}}_k(Y,\mathfrak{s},\tfrr) \times \mathcal{P} \times \mathbf{Z}\right)/\sim,
\]
where $\sim$ is the equivalence relation defined as follows.
Let $([\mathfrak{a}],\mathfrak{q}_1,m)$ and $([\mathfrak{b}],\mathfrak{q}_2,n)$ be two elements above.
Let $\zeta$ be a path joining $[\mathfrak{a}]$ and $[\mathfrak{b}]$ and $\mathfrak{p}$ be a 1-parameter family of perturbation.
Then $([\mathfrak{a}],\mathfrak{q}_1,m) \sim ([\mathfrak{b}],\mathfrak{q}_2,n)$ if there exists $(\zeta,\mathfrak{p})$ for which %\eqref{eq:grading_op_P}
\[\text{ind}(P_{\zeta,\mathfrak{p}}) = n-m.\]
The integers $\mathbf{Z}$ act on $\mathbf{J}$, via
$([\mathfrak{a}],\mathfrak{q}_1,m) \mapsto ([\mathfrak{a}],\mathfrak{q}_1,m+n)$
for $n \in \mathbf{Z}$, written additively as $j \mapsto j+n$ for $j \in \mathbf{J}$.
Given a fixed admissible perturbation, the \emph{grading} of a critical point $[\mathfrak{a}]$ is the equivalence class
\[\gr[\mathfrak{a}] = ([\mathfrak{a}],\mathfrak{q},0)/\sim \  \in \mathbf{J}(\mathfrak{s},\tfrr).\]
Denote by $\tilde{C}_j$ the subgroup generated by critical points in $j \in \mathbf{J}(\mathfrak{s},\tfrr).$
Then the boundary operator has degree $-1$, and
\[\thmr_*(Y,\mathfrak{s},\tfrr) = \bigoplus_{j \in \mathbf{J}(\mathfrak{s},\tfrr)}\thmr_j(Y,\mathfrak{s},\tfrr).\]
\subsection{Cohomology and duality}
With the same assumptions as above, the \emph{framed real monopole Floer cochain complex} is given by
\[\tilde{C}^j = \text{Hom}(\tilde{C}_j,\mathbf{F}),\]
graded by $\mathbf{J}(\mathfrak{s},\mR)$ , and equipped with the induced co-differential $\tilde{d}$.
The \emph{framed real monopole Floer cohomology} is
\[\thmr^j(Y,\mathfrak{s},\tfrr) = H^j(\tilde{C}^*,\tilde{\delta}).\]
\subsubsection*{The dotted unknot in the 3-sphere}
$(S^3,U,\mathbf{p})$
where $\mathbf{p}=\{p_1,\dots,p_n\}$.
There is a unique real spin\textsuperscript{c} structure on $(S^3,\tau_{S^3})$.
For $n$ basepoints, there are $2^{n-1}$ distinct {\rrscs}s, as a torsor over
\begin{equation*}
    \frac{H^1(S^3,\{p_1,\dots, p_n\})^{\tau^*}}{\mathrm{1+\tau^*}} \cong
    \frac{\mathbf{Z}^{n-1}}{2\mathbf{Z}^{n-1}}.
\end{equation*}
Fix one such $(\mathfrak{s},\tfrr)$.
Equip $S^3$ with the round metric that's invariant under the involution $\tau_{S^3}$.
Since this metric has non-negative scalar curvature, the only critical points of $\mathcal{L}$ are the reducibles and in this case is unique.
This critical point is already non-degenerate since the corresponding Dirac operator has no kernel.
\[
\thmr(S^3,\mathfrak{s},\tfrr) = \tcmr(S^3,\mathfrak{s},\tfrr) \cong \mathbf{F}.
\]
As the sum over all {\rrscs}s, $\thmr(S^3,\mathfrak{s},\mathbf{p})$ has rank $2^{n-1}$.
A priori, one cannot compare the gradings of these elements.
\section{Functoriality}
\label{sec:functor}
\begin{notat}
As a reminder, a marked real manifold of the form $(M,\tau,\mathbf{c})$ may be written as blackboard ``$\mathbb{M}$''.
\end{notat}
Let $\mathbb{Y}= (Y,\tau,\mathbf{p})$ be a marked real 3-manifold.
Fix a $\tau$-invariant Riemannian metric $g$.
For each {\rrscs} $(\mathfrak{s},\tfrr)$, let $\mathcal{P}(Y,\mathfrak{s},\tfrr)$ be a large Banach space of tame perturbations.
Consider the product over all {\rrscs}:
\[\mathcal{P}(Y,\tau,\mathbf{p}) = \prod_{(\mathfrak{s},\tfrr)} \mathcal{P}(Y,\mathfrak{s},\tfrr)\]
An element $\mathfrak{q} = \{\mathfrak{q}_{(\mathfrak{s},\tfrr)}\}$ is \emph{admissible} if and only if all $\mathfrak{q}_{(\mathfrak{s},\tfrr)}$ are admissible, with a uniform constant $m_2$ as in \ref{defn:tame_pert}.
Fixing an admissible perturbation, denote
\[\thmr_*(Y,\tau,\mathbf{p};g,\mathfrak{q}) =
\bigoplus_{(\mathfrak{s},\tfrr)} \thmr_*(Y,\mathfrak{s},\tfrr;g,\mathfrak{q}).\]
This group is graded by
\[\mathbf{J}(Y,\tau,\mathbf{p}; g,\mathfrak{q})
= \coprod_{(\mathfrak{s},\tfrr)} \mathbf{J}(\mathfrak{s},\tfrr; g, \mathfrak{q}).\]
\begin{defn}
    The category $\widetilde{\textsc{rcob}}_*$ has objects $5$-tuples $(Y,\tau,\mathbf{p}; g,\mathfrak{q}) = (\mathbb Y; g,\mathfrak{q})$ where $\mathbb{Y} = (Y,\tau,\mathbf{p})$ is a marked real 3-manifold, $g$ is a $\tau$-invariant Riemannian metric, and $\mathfrak{q}$ is an admissible perturbation.
    A morphism in $\widetilde{\textsc{rcob}}_*$ is a marked real cobordism between the underlying marked real 3-manifolds.
\end{defn}
The main result of this section is:
\begin{theorem}
    The framed real monopole Floer homology $\thmr_*(Y,\tau,\mathbf{p};g,\mathfrak{q})$ defines a covariant functor
    \[\thmr_* \colon \widetilde{\textsc{rcob}}_* \to \textsc{vect}_{\mathbf{F}},\]
    and the  framed real monopole Floer cohomology $\thmr^*(Y,\tau,\mathbf{p};g,\mathfrak{q})$ defines a contravariant functor
    \[\thmr^* \colon \widetilde{\textsc{rcob}}_* \to \textsc{vect}_{\mathbf{F}},\]
    where $\textsc{vect}_{\mathbf{F}}$ is the category of $\mathbf{F}$-vector spaces.
\end{theorem}
More precisely, every marked real cobordism $(W,\tau_W,\mathbf{a}) \colon (Y_-,\tau_-,\mathbf{p}_-) \to (Y_+,\tau_+,\mathbf{p}_+)$ defines a homomorphism
\[\thmr(W,\tau_W,\mathbf{a}) \colon
\thmr_*(Y_-,\tau_-,\mathbf{p}_-;g_-,\mathfrak{q}_-) \to \thmr_*(Y_+,\tau_+,\mathbf{p}_+;g_+,\mathfrak{q}_+)\]
summed over {\rrsc}s:
\[\thmr(W,\tau_W,\mathbf{a}) =
\sum_{\tfrr_W \in \underline{\rspinc}(W,\tau_W,\mathbf{a})}
\thmr(W,\tfrr_W).\]
Applying the functor to the cylinder $[0,1] \times (Y,\tau,\mathbf{p})$ yields the invariance of the Floer homology groups:
\begin{cor}
    For any two objects $(Y,\tau,\mathbf{p}; g,\mathfrak{q})$ and $(Y,\tau,\mathbf{p}; g',\mathfrak{q}')$, their framed real monopole Floer (c)homology groups are canonically isomorphic.
    \hfill \qedsymbol
\end{cor}
\begin{cor}
    The framed real monopole Floer homology (cohomology, resp.) $\thmr_*(Y,\tau,\mathbf{p};g,\mathfrak{q})$ ($\thmr^*(Y,\tau,\mathbf{p};g,\mathfrak{q})$, resp.) defines a covariant (contravariant, resp.) functor
    \[\thmr_* \ (\thmr^*, \text{resp.}) \colon \textsc{rcob}_* \to \textsc{vect}_{\mathbf{F}}.\]
\end{cor}
\subsection{Moduli spaces on manifolds with boundaries}
\label{sec:mod_space_on_4manifold_w_d}
Let $(X,\tau_X,\mathbf{a})$ be a compact, connected, oriented, marked real 4-manifold with non-empty boundary $(Y,\tau,\mathbf{p})$.
Suppose $\mathbb X$ contains an isometric copy of the cylinder $(-C,0] \times \mathbb Y$.
Denote by $\mathbb Y^{\alpha} = (Y^{\alpha},\tau^{\alpha},\mathbf{p}^{\alpha})$ a component of $\mathbb Y$ and assume $\mathbf{p}^{\alpha}$ is non-empty.
Let $(\mathfrak{s}_X,\tfrr_X)$ be a {\rrscs} relative to $\mathbf{a} \subset X$, and $(\mathfrak{s}^{\alpha},\tfrr^{\alpha})$ be the corresponding restriction to the $\alpha$-th boundary component.
Consider the configuration spaces and Banach spaces of tame perturbations on the boundary (notationally omitting the underlying $\spinc$ structures and involutions):
\begin{equation*}
    \tilde{\mathcal{B}}_k(Y,\tfrr) =
    \prod \tilde{\mathcal{B}}_k(Y^{\alpha},\tfrr^{\alpha}), \quad
    \mathcal{P}(Y,\tfrr) =  \prod \mathcal{P}(Y^{\alpha},\tfrr^{\alpha}).
\end{equation*}
Over the 4-manifold, consider the configuration $\mathcal{C}_k(X,\tfrr_X)$ and the \emph{closed} Hilbert manifold $\mathcal{B}_k(X,\tfrr_X)$.
The Seiberg--Witten operator is a section $\mathfrak{F}\colon \mathcal{C}_k \to \mathcal{V}_{k-1}$ given by~\eqref{eq:grad_pertL}.
Introduce a perturbation at the cylindrical region $I \times \mathbb Y$: Let $\mathfrak{q},\mathfrak{p}_0$ be two elements of $\mathcal{P}(Y,\tfrr)$ and $\beta,\beta_0$ be two bump functions on $I \times Y$.
Define
\begin{equation*}
    \mathfrak{F}_{\mathfrak{p}} = \mathfrak{F} + \widehat{\mathfrak{p}}, \quad
    \widehat{\mathfrak{p}} \colon \mathcal{C}(X,\tfrr_X) \to \mathcal{V}_k,\quad
    \widehat{\mathfrak{p}} = \beta \widehat{\mathfrak{q}} + \beta_0 \widehat{\mathfrak{p}}_0.
\end{equation*}
The moduli space of perturbed framed real Seiberg--Witten solutions is the solutions modulo $\mathcal{G}(X,\tau,\mathbf{a})$, denoted as
\begin{equation*}
    \tilde{N}(X,\tfrr_X).
\end{equation*}
Since everything is downstairs, there is always a well-defined (not just partially) restriction map for any interior domain $X' \subset X$:
\begin{equation*}
    \tilde{\mathcal{B}}_k(X,\tfrr_X) \to \tilde{\mathcal{B}}_k(X',\tfrr_{X})
\end{equation*}

\begin{prop}
    \label{prop:manifold_with_d_smooth_Hilbert}
On the marked real 4-manifold with boundary $(X,\tau_X,\mathbf{a})$, the section $\mathfrak{F}_{\mathfrak{p}}$ of $\mathcal{V}_{k-1}$ is transverse to zero,
and the subset $\tilde{N}(X,\tfrr_X)$ of $\tilde{\mathcal{B}}_k(X,\tfrr_X)$ is a smooth Hilbert submanifold.
\end{prop}

\begin{proof}
Let $\upgamma = (A,\Phi)$ be a solution.
The linearization of the unperturbed operator is the operator
\begin{equation*}
    \mathcal{D}_{\upgamma}\mathfrak{F} \colon
    L^2_k(X;iT^*X)^{-\tau^*_X} \oplus L^2_k(X;S^+)^{\tfrr_X} \to
    L^2_{k-1}(X;i\mathfrak{su}(S^+) \oplus S^-)^{\tfrr_X}
\end{equation*}
expressed as
\begin{equation*}
    (a,\phi) \mapsto \left(\rho_X(d^+a) - (\Phi\phi^* + \phi\Phi^*)_0, D^+_A \phi + \rho_X(a)\Phi.\right)
\end{equation*}
The linearization of the perturbation $\hat{\mathfrak{p}}$ is an operator
\begin{equation*}
    \mathcal{D}_{\upgamma}\hat{\mathfrak{p}}\colon
    L^2_k(X;iT^*X)^{-\tau^*_X} \oplus L^2_k(X;S^+)^{\tfrr_X} \to
    L^2_{k-1}(X;i\mathfrak{su}(S^+) \oplus S^-)^{\tfrr_X}
\end{equation*}
and its formal adjoint will be denoted as
\begin{equation*}
    (\mathfrak{t}_1,\mathfrak{t}_2)\colon L^2_{k-1}(X;i\mathfrak{su}(S^+) \oplus S^-)^{\tfrr_X} \to
    L^2_{k-1}(X;iT^*X)^{-\tau^*_X} \oplus L^2_{k-1}(X;S^+)^{\tfrr_X}.
\end{equation*}
It suffices to prove
$
    \mathcal{Q}_{\upgamma} = \mathcal{D}_{\upgamma} \mathfrak{F}_{\mathfrak{p}} \oplus \mathbf{d}_{\upgamma}^{*}
$
is surjective, where
\[\mathbf{d}_{\upgamma}^*: \mathcal{T}_{k,\upgamma} \to L^2_{k-1}(X,\mathbf{a};i\mathbf{R})^{-\tau_X^*},
    \quad (a,\phi) \mapsto -d^*a + i\text{Re}\langle i\Phi,\phi\rangle.\]

By \cite[Cor.~17.1.5]{KMbook2007}, it remains to show the formal adjoint equation $(\mathcal{Q}_{\upgamma})^*v=0$,
for which $v|_Y = 0$,
has only trivial solution by appealing to the unique continuation of the perturbed equations.
The equation $(\mathcal{Q}_{\upgamma})^*v=0$ can be expressed as
\begin{equation*}
    \begin{cases}
        0 = (d^+)^*\rho_X^*\eta + \rho_X^*(\pi \Phi^*) - d\xi + \mathfrak{t}_1(\eta,\uppi)\\
        0 = D^-_A\pi - 2\eta(\Phi) + \xi\Phi  + \mathfrak{t}_2(\eta,\uppi)
    \end{cases}
\end{equation*}
where $v=(\eta,\uppi,\xi) \in L^2_j(X;i\mathfrak{su}(S^+) \oplus S^- \oplus i\mathbf{R})^{\tfrr_X}$.
Right away, these equations has the following the form over the collar region:
\begin{equation}
    \label{eq:application_unique_continuation}
    \frac{d}{dt}v + (L_0 + h(t))v=0,
\end{equation}
where $L_0$ is self-adjoint elliptic and $h(t)$ satisfies the condition of \cite[Prop.~7.1.3]{KMbook2007} in its blow-down form, which can be readily applied to the real case,
where $L_0$ is self-adjoint elliptic and $h(t)$ is bounded in $L^2$ locally uniformly in $t$, by tameness.
Thus \cite[Lem.~7.1.3]{KMbook2007} applies directly.
Thus $v$ vanishes on $(-C,0]$, and hence on all of $X$ by unique continuation.
\end{proof}
Let $(X^*,\tau_X,\mathbf{a}^*)$ be the manifold with cylindrical-end  obtained from $X$ by attaching $\mathbb{Z} = [0,\infty) \times (Y,\tau,\mathbf{p})$.
\begin{defn}
    Given a critical point $[\mathfrak{b}]$ in $\mathcal{B}_k(Y,\tfrr)$, the moduli space
\begin{equation*}
    \wtilde{N}(X^*,\tau_X,\mathbf{a}^*,\tfrr_X;[\mathfrak{b}]) \subset \tilde{\mathcal{B}}_{k,loc}(X^*,\tau_X,\mathbf{a}^*,\tfrr_X)
\end{equation*}
is the set of all $[\upgamma]$ solving $\mathfrak{F}_{\mathfrak{p}}(\upgamma) = 0$ for which the restriction of $[\upgamma]$ is asymptotic to $[\mathfrak{b}]$ on the cylindrical end $Z$.
\end{defn}

To define regularity, interpret the moduli space over manifolds with cylindrical ends as the fibre product for the maps:
\[R_+\colon \wtilde N(X,\tfrr_X) \to \wtilde{\mathcal{B}}_{k-1/2}(Y,\tfrr), \quad
R_-\colon \wtilde N(Z,\tfrr;[\mathfrak{b}]) \to \wtilde{\mathcal{B}}_{k-1/2}(Y,\tfrr).\]

\begin{defn}
    The moduli space $\wtilde N(X^*,\tfrr_X;[\mathfrak{b}])$  is \emph{regular} at $[\upgamma] \in \wtilde N(X^*,\tfrr_X;[\mathfrak{b}])$ if $R_+$ and $R_-$ are transverse at $\rho[\upgamma]$.
\end{defn}
\begin{prop}
    If  $\wtilde N(X^*,\tfrr_X;[\mathfrak{b}])$ is nonempty and regular, then it is a smooth manifold. \hfill \qedsymbol
\end{prop}
Let $\wtilde{\boldsymbol{\mathcal{B}}}(X^*,\tau_X,\mathbf{a})$ be the union of the configuration spaces over all {\rrscs}s.
Denote by $\tilde{B}(X^*,\tau_X,\mathbf{a};[\mathfrak{b}])$ the fibre of the restriction to $\tilde{\boldsymbol{\mathcal{B}}}(Y,\tfrr)$.
This space can be written as the union over homotopy classes:
\begin{equation*}
    z \in \pi_0(\wtilde{\mathcal{B}}(X,\tau_X;[\mathfrak{b}])), \quad
    \wtilde{\mathcal{B}}_{k,\text{loc}}(X^*,\tau_X,\mathbf{a};[\mathfrak{b}]) =
    \bigcup_z \wtilde{\mathcal{B}}_{k,\text{loc},z}(X^*,\tau_X,\mathbf{a};[\mathfrak{b}]);
\end{equation*}
where $z$ plays the role of $\pi_1(\wtilde{\mathcal{B}}(Y,\tfrr);[\mathfrak{a}],[\mathfrak{b}])$.
The moduli space $\wtilde{N}(X^*,\tau_X;[\mathfrak{b}])$ can be written as the union
\begin{equation*}
    \wtilde{N}(X^*,\tau_X;[\mathfrak{b}]) = \bigcup_z \wtilde{N}_z(X^*,\tau_X;[\mathfrak{b}]).
\end{equation*}
Elements $z \in \pi_0(\wtilde{\mathcal{B}}(X,\tau_X;[\mathfrak{b}_0]))$ and $z_1 \in \pi_1(\wtilde{\mathcal{B}}(Y,\tfrr);[\mathfrak{b}_0],[\mathfrak{b}])$ can be concatenated to form $z_1 \circ z \in \pi_0(\wtilde{\mathcal{B}}(X,\tau_X;[\mathfrak{b}_0]))$.
Let
$[\upgamma] \in \wtilde{\mathcal{B}}(X,\tau_X;[\mathfrak{b}])$, represented by $\upgamma$, and $\mathcal{Q}_{\upgamma}$ be the operator $\mathcal{D}_{\upgamma} \mathfrak{F}_{\mathfrak{q}} \oplus \mathbf{d}_{\upgamma}^{*}$ on $X$.
Let $[\upgamma_{\mathfrak{b}}]$ be the constant trajectory corresponding to $\mathfrak{b}$, and $\mathcal{Q}_{\upgamma_{\frakb}}$ be the translation-invariant operator on $Z$.
Consider the restriction maps:
\begin{align*}
    r_+\colon \ker(\mathcal{Q}_{\upgamma})
    &\to L^2_{k-1/2}(Y;iT^*Y \oplus S \oplus i\mathbf{R})^{-\tfrr_X},\\
    r_-\colon \ker(\mathcal{Q}_{\upgamma_{\frakb}})
    &\to L^2_{k-1/2}(Y;iT^*Y \oplus S \oplus i\mathbf{R})^{-\tfrr_X}.
\end{align*}
Denote
\[
    \gr_z(X,\tau_X,\mathbf{p}; [\mathfrak{b}]) = \ind(r_+ - r_- \colon
    \ker(\mathcal{Q}_{\upgamma}) \oplus \ker(\mathcal{Q}_{\upgamma_{\frakb}}) \to
    L^2_{k-1/2}(Y;iT^*Y \oplus S \oplus i\mathbf{R})^{-\tfrr_X}).
\]
\begin{prop}
    If the moduli space $\wtilde N_z(X,\tau_X,\mathbf{a};[\mathfrak{b}])$ is non-empty and regular, then its dimension is $\gr_z(X,\tau_X,\mathbf{p}; [\mathfrak{b}])$. \hfill \qedsymbol
\end{prop}
Let $P$ be a smooth finite-dimensional manifold.
Suppose $g^p$ is a family of $\tau_X$-invariant Riemannian metrics containing isometric copies of the collar $I \times Y$.
Suppose $\mathfrak{p}_0^p \in \mathcal{P}(Y,\tfrr)$ is a smooth family of perturbations; let
\[\hat{p}^p=\beta(t)\mathfrak{q} + \beta_0(t)\hat{\mathfrak{p}}_0\]
The family version of the moduli space is defined as
\[\wtilde{N}(X^*,\tfrr_X;[\mathfrak{b}])_P = \bigcup_p \{p\} \times \wtilde{N}(X^*,\tfrr_X;[\mathfrak{b}]) \subset
P \times \wtilde{\mathcal{B}}(X^*,\tfrr_X).
\]
Once again, there are restriction maps
\[R_+\colon \wtilde N(X,\tfrr_X)_P \to \wtilde{\mathcal{B}}_{k-1/2}(Y,\tfrr), \quad
R_-\colon \wtilde N(Z,\tfrr;[\mathfrak{b}]) \to \wtilde{\mathcal{B}}_{k-1/2}(Y,\tfrr),\]
from where we define the family version of regularity:
\begin{defn}
    Given $(p,\upgamma) \in \wtilde N(X,\tfrr_X)_P$, and $\rho[\upgamma] = ([\upgamma_0],[\upgamma_1])$, the moduli space $\wtilde N(X,\tfrr_X)_P$ is \emph{regular} at $(p,[\upgamma])$ if the maps of Hilbert manifolds $R_+$ and $R_-$ are transverse at $((p,[\upgamma_0]),[\upgamma_1])$.
\end{defn}
\begin{prop}
    \label{prop:family_transversality_manifold_with_d}
    Fix an admissible perturbation $\mathfrak{q}$ on $Y$.
    Let $g^p$ be a smooth family of $\tau_X$-invariant Riemannian metrics, all containing an isometric copy of $I \times (Y,\tau)$.
    Let $\hat{p}^p$ be a family of perturbations of the form:
    \[\hat{p}^p=\beta(t)\mathfrak{q} + \beta_0(t)\hat{\mathfrak{p}}_0\]
    all supported on the collar $I \times Y$.
    Let $P_0 \subset P$ be a closed subset for which the parametrized moduli space
    $\wtilde N(X,\tau_X,\mathbf{a};[\mathfrak{b}])_P$
    is regular at all points $(p_0,[\upgamma])$ where $p_0 \in P_0$.
    Then there exists a family of perturbations $\tilde{\mathfrak{p}}^p$ with
    \[\tilde{\mathfrak{p}}^p = \mathfrak{p}^p, \quad \text{for all} \ p \in P_0,\]
    such that $\wtilde N(X,\tau_X,\mathbf{a};[\mathfrak{b}])_P$ is regular everywhere.
    In the case when $P$ is a single point and $P_0 = \emptyset$,
     there is a residual subset of $\mathcal{P}(Y,\tfrr_X)$ such that for all element $\mathfrak{p}$, the moduli space $\wtilde N_z(X,\tau_X,\mathbf{a};[\mathfrak{b}])$ is regular.
\end{prop}
\begin{proof}
    This proposition is the framed analogue of \cite[Prop.~24.4.10]{KMbook2007}.
    By the same argument, it suffices to
    consider the case when $P$ is a point, and the universal moduli space (a smooth Banach manifold by Proposition~\ref{prop:manifold_with_d_smooth_Hilbert}):
    \[\mathfrak{W}(X,\widetilde{\sr}_X) \subset \wtilde{\mathcal{B}}_k(X,\tfrr_X) \times \mathcal{P}(Y,\tfrr)\]
    defined as the zero set of
    \begin{equation*}
        \wtilde{\mathcal{C}}_k(X,\tfrr_X) \times \mathcal{P}(Y,\tfrr)
        \to \mathcal{V}_{k-1},\quad
        (\upgamma,\mathfrak{p}) \mapsto \mathfrak{F}_{\mathfrak{p}}(\upgamma).
    \end{equation*}
    modulo the framed gauge group $\mathcal{G}_{\mathbf{a}}(X,\tau_X)$.
    One can prove the auxiliary lemma below by adapting the irreducible case of \cite[Lem.~24.4.8]{KMbook2007}:
    \begin{lem}
    Let $([\upgamma],\mathfrak{p}_0) \in \mathfrak{W}(X,\tfrr_X)$ and let $[\mathfrak{c}] = R_+([\upgamma])$.
    Then the differential of the restriction map
    \[\mathcal{D}_{([\upgamma],\mathfrak{p}_0)} R_+ \colon
    T_{([\upgamma],\mathfrak{p}_0)} \mathfrak{W}(X,\tfrr_X) \to
    T_{[\mathfrak{c}]}\widetilde{\mathcal{B}}_{k-1/2}(Y,\tfrr).\]
    has dense range in the $L^2_{1/2}$ topology.
    \end{lem}
    The proof of the above Lemma uses the proof of Theorem~\ref{thm:transversality_cylinder} in place of \cite[Prop.~15.1.3]{KMbook2007}.
	Here Proposition~\ref{prop:enough_embeddings} detects spinor directions also at reducible configurations,
	because the cylinder functions are not required to satisfy $f(B,\Psi)=f(B,-\Psi)$.

	One can conclude, using the above the lemma and the argument in \cite[Prop.~24.4.7]{KMbook2007} that
    the restriction
    \[
	R_+ \times R_- \colon
    \mathfrak{W}(X,\tfrr_X) \times \wtilde{N}(Z,\widetilde{\sr};[\mathfrak{b}])
    \to
    \wtilde{\mathcal{B}}_{k-1/2}(Y,\tfrr) \times \wtilde{\mathcal{B}}_{k-1/2}(Y,\tfrr) \]
    is transverse to the diagonal.
    Proposition~\ref{prop:family_transversality_manifold_with_d} then follows from the standard transversality argument~\cite[Lem.~12.5.1]{KMbook2007}.
\end{proof}
\subsubsection*{Compactness}
Denote
\[X' \Subset \big(X \setminus (-\epsilon,0] \times Y\big).\]
Suppose $\mathfrak{q}$ is the fixed perturbation on $\del X = Y$ for which $\pertL_{\mathfrak{q}} = \mathcal{L} + f$.
The \emph{perturbed topological energy} is
\[\mathcal{E}^{\text{top}}_{\mathfrak{q}}(\upgamma) :=
\mathcal{E}^{\text{top}}(\upgamma) - 2f(\upgamma).\]
\begin{theorem}
    Let $\upgamma_n \in \mathcal{C}_k(X,\tfrr_X)$ be a sequence of solutions to $\mathfrak{F}_{\mathfrak{p}}(\upgamma) = 0$.
    Suppose there is a uniform bound on the energy:
    \[\mathcal{E}^{\text{top}}_{\mathfrak{q}}(\upgamma_n) \le C_1.\]
    Then there exists a sequence of gauge transformations $u_n \in \mathcal{G}_{\mathbf{a}}(X,\tau_X)$ such that
    after passing to a subsequence, the restrictions $u_n(\upgamma_n)|_{X'}$ converges to a solution $\upgamma \in \mathcal{C}_k(X',\tfrr_X)$.
\end{theorem}
\begin{proof}
    This follows from the proof of the downstairs case of \cite[Thm.~24.5.2]{KMbook2007} which essentially a consequence of
    \cite[Thm.~5.1.1]{KMbook2007}.
\end{proof}
\begin{defn}
    Let $[\mathfrak{b}]$ be a critical point and $[\mathfrak{b}^{\alpha}]$ be its restriction to the $\alpha$-th component of $Y$.
    A \emph{broken $X$-trajectory} $([\upgamma_0],[\breve{\boldsymbol{\upgamma}}])$ \emph{asymptotic to} $[\mathfrak{b}]$ consists of the data of
    \begin{itemize}[leftmargin=*,noitemsep]
        \item an element $[\upgamma_0]$ in $\tilde{N}_{z_0}(X^*,[\mathfrak{b}_0])$, and
        \item for each $Y^{\alpha}$, an unparametrized broken trajectory $[\breve{\boldsymbol{\upgamma}}^{\alpha}]$ in a moduli space $\tilde{N}_{z_{\alpha}}^+([\breve{\boldsymbol{\upgamma}}^{\alpha}_0],[\breve{\boldsymbol{\upgamma}}^{\alpha}])$, where $[\breve{\boldsymbol{\upgamma}}^{\alpha}_0]$ is the restriction of $[\mathfrak{b}_0]$ to $Y^{\alpha}$.
    \end{itemize}
    Let $z_1$ be the homotopy class of the paths from $[\mathfrak{b}_0]$ to $[\mathfrak{b}]$ that restricts to $z^{\alpha}$ on $Y^{\alpha}$.
    The \emph{homotopy class} of the broken $X$-trajectory $([\upgamma_0],[\breve{\boldsymbol{\upgamma}}])$ is
    \[z = z_1 \circ z_0 \in \pi_0(\wtilde{\mathcal{B}}(X,\tau_X,[\mathfrak{b}])).\]
    Let $\tilde{N}_{z}^+(X^*,[\mathfrak{b}])$ denote the space of broken $X$-trajectories in the homotopy class $z$.
    Topologize the space $\tilde{N}_{z}^+(X^*,[\mathfrak{b}])$ will be in the same way as \cite[p.485-486]{KMbook2007}.
    \hfill $\diamondsuit$
\end{defn}
Extend the definition of the perturbed topological $\mathcal{E}^{\text{top}}_{\mathfrak{q}}(\upgamma)$ energy to broken $X$-trajectories $\boldsymbol{\upgamma}$ as the sum of all the energies of the components.

\begin{prop}
    \label{prop:bound_C_X_traj_broken_compact}
    Fix $C > 0$ and a critical point $[\mathfrak{b}]$.
    The space of broken $X$-trajectories $[\breve{\boldsymbol{\upgamma}}] \in \cup_z \tilde{N}_z^+(X^*,[\mathfrak{b}])$ satisfying $\mathcal{E}^{\text{top}}_{\mathfrak{q}}(\breve{\boldsymbol{\upgamma}}) \le C$ is compact.
\end{prop}

\begin{proof}
    This is a simpler version of the proof \cite[Prop.~24.6.4]{KMbook2007}, as there is no need to control the function $\Lambda(\upgamma_n)$.
\end{proof}
The follow finiteness result can be proved the same way as
\cite[Prop.~24.6.6]{KMbook2007}:
\begin{prop}
    \label{prop:compactness_d_0_bound}
    Suppose $\mathfrak{q}$ is admissible and $\mathfrak{p}$ is chosen so that all $\tilde{N}_z(X^*,[\mathfrak{b}])$'s are regular.
    For any $d_0 \ge 0$, there are only finitely many $(z,[\mathfrak{b}])$ for which $\tilde{N}_z(X^*,[\mathfrak{b}])$ is non-empty and having at most dimension $d_0$.
\end{prop}
\begin{proof}
    The Proposition is the framed case of \cite[Prop.~24.6.5]{KMbook2007}, and it is the counterpart of the cylinder case Theorem~\ref{thm:energy_bound_traj_space_compact}.
    The proof is again simpler, which involves a bound $C$ uniform for $([\mathfrak{b}],z,[\boldsymbol{\upgamma}])$ of the form
    \[ \mathcal{E}^{\text{top}}_{\mathfrak{q}}([\boldsymbol{\upgamma}]) \le C,\]
    without the term $\iota([\mathfrak{b}])$ that measures the height in the blowup $\mathbf{CP}^{\infty}$~\cite[Lem.~16.4.4]{KMbook2007}.
\end{proof}

Let $P$ be a compact smooth manifold parametrizing a family of metrics and perturbations, all agreeing on a neighbourhood of the boundary.
The family moduli space of broken $X$-trajectories is defined as the union
\[\tilde{N}_{z}^+(X^*,[\mathfrak{b}])_P
= \bigcup_{p} \{p\} \times \tilde{N}_{z}^+(X^*,[\mathfrak{b}])_p,\]
which can be topologized similarly as the a single moduli space.
This compactification is stratified by subspaces of the form
\begin{equation}
    \label{eq:family_strat_X}
    \tilde{N}_{z_0}(X^*,[\mathfrak{b}_0])_P \times
    \prod_{\alpha} \breve{\tilde{N}}_{z^{\alpha}}^+([\mathfrak{b}_0^{\alpha}],[\mathfrak{b}^{\alpha}]).
\end{equation}
%where $z_0 \in \pi_0(\wtilde{\mathcal{B}}(X,\tau_X,[\mathfrak{b}_0]))$.
\begin{theorem}
	\label{thm:family_proper}
    Suppose $\tilde{N}_{z}(X^*,[\mathfrak{b}])_P$ are regular for all critical points $[\mathfrak{b}]$.
    For each $[\mathfrak{b}]$, the map $\tilde{N}^+_{z}(X^*,[\mathfrak{b}])_P \to P$ is proper.
	For a fixed  $[\mathfrak{b}]$, there are finitely many $z$ for which $\tilde{N}^+_{z}(X^*,[\mathfrak{b}])_P$ is non-empty.
\end{theorem}
\begin{proof}
    This can be essentially deduced from the case when $P$ is a point, which are Proposition~\ref{prop:bound_C_X_traj_broken_compact} and Proposition~\ref{prop:compactness_d_0_bound}.
    See \cite[Thm.~24.6.7]{KMbook2007}.
\end{proof}
The following is the analogue of Proposition~\ref{prop:traj_space_compact_strata}
\begin{prop}
    Suppose $\tilde{N}_{z}(X^*,[\mathfrak{b}])_P$ is a non-empty and $d$-dimensional.
    The space $\tilde{N}^+_{z}(X^*,[\mathfrak{b}])_P$ is a $d$-dimensional space stratified by manifolds, proper over $P$, and whose $e$-dimensional stratum is formed by subsets of the form \eqref{eq:family_strat_X}, for $e \ge 0$. \hfill  \qedsymbol
\end{prop}
\subsubsection*{A smaller compactification}
\begin{defn}
    The space $\overline{N}_z(X^*,\tau_X,[\mathfrak{b}])$ is the image of $N^+_z(X^*,\tau_X, [\mathfrak{b}])$ under the map
    \[r \colon N^+_z(X^*, \tau_X, [\mathfrak{b}]) \to
    \widetilde{\boldsymbol{\mathcal{B}}}_{k,\text{loc}}(X^*,\tau_X), \quad
    ([\upgamma_0],[\breve{\boldsymbol{\upgamma}}]) \mapsto
    [\upgamma_0].\]
    For a moduli spaces parametrized by a family $P$ of metrics and perturbations, the space $\overline{N}_z(X^*,\tau_X,[\mathfrak{b}])_P$ is the image of $N^+_z(X^*, \tau_X, [\mathfrak{b}])_P$ under the map
    \[r \colon N^+_z(X^*, \tau_X, [\mathfrak{b}])_P \to
    P \times \widetilde{\boldsymbol{\mathcal{B}}}_{k,\text{loc}}(X^*,\tau_X), \quad
    (p,[\upgamma_0],[\breve{\boldsymbol{\upgamma}}]) \mapsto
    (p,[\upgamma_0]).\]
\end{defn}
\begin{prop}
    \label{prop:strat_moduli_on_X_codimension}
    Assume $\widetilde{N}(X^*,[\mathfrak{b}])_P$ is nonempty and of dimension $d$.
    Then both $N^+_z(X^*, [\mathfrak{b}])_P$ and $\overline{N}_z(X^*,[\mathfrak{b}])_P$ are $d$-dimensional spaces stratified by manifolds.
    The $(d-1)$-dimensional stratum in $N^+_z(X^*, [\mathfrak{b}])_P$ consists of elements of the following two types:
    \begin{itemize}[leftmargin=*]
        \item The elements \eqref{eq:family_strat_X} for which precisely one $n^{\alpha}=1$ for some $\alpha = \alpha_*$, or
        \item the unbroken solutions lying over $\del P$ when $P$ has boundary.
    \end{itemize}
    The $(d-1)$-dimensional strata in $\overline{N}_z(X^*,[\mathfrak{b}])_P$ are the image of $r$ of the strata in the first case above, such that the $[\upgamma^{\alpha_*}_1]$ belongs to a $1$-dimensional moduli space, and unbroken solutions over $\partial P$.
    \hfill \qedsymbol
\end{prop}
\subsubsection*{Gluing on 4-manifolds with boundary}
Once again, $\delta$-structures are not needed in the gluing of framed trajectories.
The following theorem describes the stratification of the compactified moduli spaces.
Compare with \cite[Thm.~24.7.2]{KMbook2007}.
\begin{theorem}
    Suppose $\tilde{N}_{z}(X^*,[\mathfrak{b}])_P$ is a non-empty and $d$-dimensional.
    The top stratum of $\tilde{N}^+_{z}(X^*,[\mathfrak{b}])_P$ is $\tilde{N}_{z}(X^*,[\mathfrak{b}])_P$.
    Moreover, $\tilde{N}^+_{z}(X^*,[\mathfrak{b}])_P$ is a $C^0$-manifold with boundary along $M'$, for any component of the codimension-$1$ stratum $M'$.
\end{theorem}
\begin{proof}
    The boundary-unobstructed case of \cite[Thm.~24.7.2]{KMbook2007} can be adapted to the real framed case with no essential changes.
\end{proof}
\subsection{Moduli spaces over cobordisms}
Suppose $(W,\tau_W,\mathbf{a}) \colon (Y_-,\tau_-,\mathbf{p}_-) \to (Y_+,\tau_+,\mathbf{p}_+)$ is a pointed real cobordism, where the boundary $\del W$ is oriented as $-Y_- \sqcup Y_+$.
Let $\mathfrak{q}_{\pm}$ be admissible perturbations on $\mathbb Y_{\pm}$.
Let
\[
\mathbb W^* = (-\infty,0] \times \mathbb Y_- \cup \mathbb W \cup [0,\infty) \times \mathbb Y_+,
\]
and
\[\widetilde{\boldsymbol{\mathcal{B}}}_{k-1/2}(Y_{\pm},\tau_{\pm}) =
\bigcup_{(\mathfrak{s}_{\pm},\tfrr_{\pm},\mathbf{p}_{\pm})} \widetilde{\mathcal{B}}_{k-1/2}(Y_{\pm},\tau_{\pm}, \mathfrak{s}_{\pm},\tfrr_{\pm}).\]
Given $[\mathfrak{a}] \in \widetilde{\mathcal{B}}_{k-1/2}(Y_{-},\tau_{-})$ and $[\mathfrak{b}] \in \widetilde{\mathcal{B}}_{k-1/2}(Y_{+},\tau_{+})$, the moduli space over the cobordism is denote as
\[\widetilde{N}([\mathfrak{a}]; W^*,\tau_W, \mathbf{a}; [\mathfrak{a}] )
\subset \widetilde{\boldsymbol{\mathcal{B}}}_{k,\text{loc}}(W^*,\tau_W,\mathbf{a})
=\bigcup_{\mathfrak{s}_{W},\tfrr_{W}} \widetilde{\mathcal{B}}_{k,\text{loc}}(W^*,\tau_W,\mathbf{a}).\]
\begin{defn}
    A $W$-path from $[\mathfrak{a}]$ to $[\mathfrak{b}]$ is a an element $[\upgamma]$ in $\tilde{\boldsymbol{\mathcal{B}}}(W)$ which restricts to $r[\upgamma] = ([\mathfrak{a}],[\mathfrak{b}])$ on the boundaries.
    Two $W$-paths are \emph{homotopic} if they belong to the same path component of $r^{-1}([\mathfrak{a}],[\mathfrak{b}])$.
    Denote the set of homotopy classes of $W$-paths as
    $\boldsymbol{\pi}([\mathfrak{a}], W,[\mathfrak{b}]) = \boldsymbol{\pi}([\mathfrak{a}]; W, \tau_W,\mathbf{a}; [\mathfrak{b}])$.
\end{defn}
Analogous to the decomposition over $ \pi_0(\wtilde{\mathcal{B}}(X,\tau_X;[\mathfrak{b}]))$ in Section~\ref{sec:mod_space_on_4manifold_w_d},  the space $\widetilde{N}([\mathfrak{a}], W^*, [\mathfrak{b}])$ is a union over $z \in \boldsymbol{\pi}([\mathfrak{a}], W,[\mathfrak{b}])$:
\[\bigcup_{\mathfrak{s}_{W},\tfrr_{W}} \widetilde{N}([\mathfrak{a}]; W^*,\tau_W,\mathbf{a};\ \mathfrak{s}_{W},\tfrr_{W} ; [\mathfrak{b}] )
= \bigcup_z \widetilde{N}_z([\mathfrak{a}]; W^*,\tau_W,\mathbf{a}; [\mathfrak{b}]).
\]
The integer $\gr_z([\mathfrak{a}],W,[\mathfrak{b}])$ and the compactification $\tilde{N}^+([\mathfrak{a}], W^*, [\mathfrak{b}])$ can defined as special cases of $\gr_z(W,[\mathfrak{b}])$ and $\tilde{N}^+( W^*, [\mathfrak{b}])$ in Section~\ref{sec:mod_space_on_4manifold_w_d}.
A typical element of $\tilde{N}^+([\mathfrak{a}], W^*, [\mathfrak{b}])$ will be written as
\[([\breve{\boldsymbol{\upgamma}}_-],[\upgamma_0],[\breve{\boldsymbol{\upgamma}}_+]),\]
where
\[
[\breve{\boldsymbol{\upgamma}}_-] \in \breve{N}^+([\mathfrak{a}],[\mathfrak{a}_0]), \quad
[\breve{\boldsymbol{\upgamma}}_+] \in \breve{N}^+([\mathfrak{b}_0],[\mathfrak{b}]), \quad
[\upgamma_0] \in \tilde{N}^+([\mathfrak{a}_0], W^*, [\mathfrak{b}_0]).
\]
The following is a restatement of Proposition~\ref{prop:strat_moduli_on_X_codimension}.
Compare with \cite[Prop.~25.1.1]{KMbook2007}, which has three times as many stratum types.
\begin{prop}
    \label{prop:strat_on_W_d}
    Suppose $\tilde{N}_z([\mathfrak{a}], W^*, [\mathfrak{b}])$ is non-empty and of dimension $d$.
    Then both compactifications $\tilde{N}_z^+([\mathfrak{a}], W^*, [\mathfrak{b}])$ and $\bar{N}_z([\mathfrak{a}], W^*, [\mathfrak{b}])$ are $d$-dimensional spaces stratified by manifolds.
    The $d$-dimensional stratum is $\tilde{N}_z([\mathfrak{a}], W^*, [\mathfrak{b}])$.
    The $(d-1)$-dimensional stratum consists of subsets of two types:
    \begin{align*}
        \breve{N}_{-} &\times N_0,\\
        N_0 &\times \breve{N}_{+},
    \end{align*}
    where $N_0$ denotes a typical moduli space on $W^*$, and $\breve{N}_{\pm}$ denotes a typical unparametrized moduli space on $Y_{\pm}$.
    The image of one of the above strata has codimension-$1$ in the smaller compactification $\bar{N}_z([\mathfrak{a}], W^*, [\mathfrak{b}])$ when the unparametrized moduli spaces on the cylinders $\breve{N}_{\pm}$ are 0-dimensional.
    \hfill \qedsymbol
\end{prop}
\subsection{Defining cobordism maps}
The goal of this subsection is to, given a cohomology class
$u \in H^*(\boldsymbol{\mathcal{B}}(W,\tau_W,\mathbf{a});\mathbf{F})$,
define the cobordism map that evaluates $u$ over the moduli spaces on cobordisms
\[\thmr(u|W,\tau_W,\mathbf{a}).\]

Consider the smaller compactification $\bar{N}_z([\mathfrak{a}], \mathbb W^*,[\mathfrak{b}])$, which fits in
\[
    \wtilde{N}_z([\mathfrak{a}], \mathbb W^*,[\mathfrak{b}]) \subset
    \bar{N}_z([\mathfrak{a}], \mathbb W^*,[\mathfrak{b}]) \subset
    \wtilde{\boldsymbol{\mathcal{B}}}_z(\mathbb W^*).
\]
Since $\tilde{\boldsymbol{\mathcal{B}}}_z(\mathbb W^*)$ is weakly homotopic equivalent to $\tilde{\boldsymbol{\mathcal{B}}}_z(\mathbb W)$, the class $u$ can be viewed as a class over $\tilde{\boldsymbol{\mathcal{B}}}_z(\mathbb W^*)$.
Following \cite[\S 21]{KMbook2007}, one evaluates the class $u$ using the \v{C}ech model as follows.
Fix a positive integer $d_0$.
The set of all triples $(z,[\mathfrak{a}],[\mathfrak{b}])$ for which $\wtilde{N}_z([\mathfrak{a}], \mathbb W^*,[\mathfrak{b}])$ has dimension at most $d_0$ is locally finite.
Indeed, Proposition~\ref{prop:compactness_d_0_bound} over manifolds with boundaries can be translated to the following Lemma.
\begin{lem}
    For any $[\mathfrak{a}]$ and $d_0$, there are only finitely many pairs $(z,[\mathfrak{b}])$ for which the moduli space $\tilde{N}^+_z([\mathfrak{a}], W^*, [\mathfrak{b}])$ is nonempty and have dimension at most $d_0$.
    \hfill \qedsymbol
\end{lem}

By \cite[Lem.~21.2.1]{KMbook2007}, every open cover of $\tilde{\boldsymbol{\mathcal{B}}}_{k,\text{loc}}(\mathbb W^*)$ has a refinement transverse to all strata in all compactified moduli space $\bar M$ of dimension $d_0$ or less.
Let $\mathcal{U}$ be such an open cover, and $u \in C^d(\mathcal{U};\mathbf{F})$ for $d \le d_0$.
If $\tilde N_z([\mathfrak{a}],\mathbb W^*,[\mathfrak{b}])$ has dimension $d \le d_0$, then there is a well-defined evaluation:
\[\big\langle u,[\wtilde{N}_z([\mathfrak{a}],\mathbb W^*,[\mathfrak{b}])]\big\rangle \in \mathbf{F}.\]
Define the chain level map
\[C^d(\mathcal{U};\mathbf{F}) \otimes \wtilde{C}_*(\mathbb Y_-) \to \wtilde{C}_*(\mathbb Y_+)\]
as
\begin{equation}
    \label{eq:m_tilde_cob_chain_map}
    \wtilde{m}(u\otimes [\mathfrak{a}]) = \sum_{[\mathfrak b]}
    \sum_z \langle u, [\wtilde{N}_z([\mathfrak{a}],\mathbb W^*,[\mathfrak{b}])]\rangle[\mathfrak{b}].
\end{equation}
The proof the following Proposition is significantly simpler than its unframed counterpart \cite[Prop.~25.3.4]{KMbook2007}, since one needs not differentiate the boundary and interior critical points.
\begin{prop}
    \label{prop:id_of_cob_u_map}
    The operator $\wtilde{m}$ satisfies the identity:
    \[\tilde{\del}(\mathbb{Y}_+) \tilde{m}(u \otimes \xi) =
    \tilde{m}(\delta u \otimes \xi) +
    \tilde{m}(u \otimes \tilde{\del}(\mathbb{Y}_-)\xi),\]
    for $u \in C^d(\mathcal{U};\mathbf{F}), \xi \in \tilde{C}(\mathbb{Y}) = \tilde{C}(Y,\tau,\mathbf{p})$, and $d \le d_0 - 1$.
    Hence $\wtilde{m}$ gives rise to a well-defined operator
    \[\wtilde{m} \colon
    \check{H}^d(\mathcal{U};\mathbf{F}) \otimes \thmr(\mathbb{Y}_-) \to
    \thmr(\mathbb{Y}_+),\]
    for any open cover $\mathcal{U}$ of $\tilde{\boldsymbol{\mathcal{B}}}_{k,\text{loc}}(\mathbb W^*)$ transverse to all moduli spaces of dimension less than or equal to $d_0$.
\end{prop}
\begin{proof}
    Consider a moduli space $\wtilde{N}_z([\mathfrak a], \mathbb{W}, [\mathfrak{b}])$ of dimension $d+1$ and its compactification $\overline{N}_z([\mathfrak a], \mathbb{W}, [\mathfrak{b}])$.
    Let $v\colon N_z^+([\mathfrak a], \mathbb{W}, [\mathfrak{b}]) \to \overline{N}_z([\mathfrak a], \mathbb{W}, [\mathfrak{b}])$ be the quotient map.
    Since $d+1 \le d_0$ and by Stokes theorem~\cite[(21.4)]{KMbook2007}
    \[
    \langle \delta u, N_z([\mathfrak a], \mathbb{W}, [\mathfrak{b}])  \rangle  = \sum_{\beta} \delta_{\beta} \langle u, N^d_{\beta} \rangle,
    \]
    where $\beta$ ranges over all $d$-dimensional strata and $\delta$ is the boundary multiplicity.
    To compute $\delta_{\beta}$, let $\mathcal{U}^+$ be the pullback open cover on $N^+$ and $u^+$ be the pullback cochain.
    Apply Stokes theorem again:
    \[
    \sum_{\beta} \delta_{\beta} \langle u^+, (N^+)^d_{\beta} \rangle =
    \langle \delta u^+, N_z([\mathfrak a], \mathbb{W}, [\mathfrak{b}]) =
    \langle \delta u, N_z([\mathfrak a], \mathbb{W}, [\mathfrak{b}])\rangle,
    \]
    where the sum is over all $\beta$ for which the component $(N^+)^d_{\beta}$ of the $d$-dimensional stratum of $N_z^+([\mathfrak a], \mathbb{W}, [\mathfrak{b}])$ whose image in $\overline{N}_z$ is also $d$-dimensional.
    All $\delta_{\beta}^+$ are $\pm 1$ (\cite[Lem.~21.3.1]{KMbook2007}).
    By Proposition~\ref{prop:strat_moduli_on_X_codimension},
    either
    \begin{equation}
        \label{eq:cob_map_N_strat_1}
        (N^+)^d = \breve{N}_{z_{-1}}([\mathfrak{a}],[\mathfrak{a}_0]) \times \wtilde{N}_{z_0}([\mathfrak{a}_0],\mathbb{W}^*,[\mathfrak{b}])
    \end{equation}
    or
    \begin{equation}
        \label{eq:cob_map_N_strat_2}
        (N^+)^d =  \wtilde{N}_{z_0}([\mathfrak{a}],\mathbb{W}^*,[\mathfrak{b}_0])
    \times \breve{N}_{z_{+1}}([\mathfrak{b}_0],[\mathfrak{b}]),
    \end{equation}
    along which the moduli space is a $C^0$ manifold with boundary.
    The sum of contributions of the form \eqref{eq:cob_map_N_strat_1} gives rise to $\tilde{m}(u \otimes \tilde{\del}(\mathbb{Y}_-)\xi)$ and  \eqref{eq:cob_map_N_strat_2} gives rise to $\tilde{\del}(\mathbb{Y}_+) \tilde{m}(u \otimes \xi)$.
\end{proof}
    At the limit over all transverse open covers of $\tilde{\boldsymbol{\mathcal{B}}}_{k,\text{loc}}(\mathbb W^*)$, the \v{C}ech homology $\check{H}^d(\mathcal{U};\mathbf{F})$ can be identified with $H^d(\tilde{\boldsymbol{\mathcal{B}}}_{k,\text{loc}}(\mathbb W^*);\mathbf{F})$.
    The cohomological version $\wtilde{m}^*$ is defined similarly using the formula~\eqref{eq:m_tilde_cob_chain_map}.
\begin{defn}
    \label{defn:cob_eval_u_map}
    Given a marked real cobordism $\mathbb{W}=(W,\tau_W,\mathbf{a})$ from $\mathbb{Y}_- = (Y_-,\tau_-,\mathbf{p}_-)$ to $\mathbb{Y}_+ = (Y_+,\tau_+,\mathbf{p}_+)$, equipped with a Riemannian metric and a perturbation $\mathfrak{p}$.
    Let $u \in H^d(\tilde{\boldsymbol{\mathcal{B}}}_{k,\text{loc}}(\mathbb W))$.
    The operator
    \[\thmr_*(u|\mathbb{W}) \colon
    \thmr(\mathbb{Y}_-) \to
    \thmr(\mathbb{Y}_+),\] is
    defined as the operator $\wtilde{m}_*(u \otimes -)$.
    The dual map $\wtilde{m}^*(u \otimes -)$ defines the operator
    \[\thmr^*(u|\mathbb{W}) \colon
    \thmr^*(\mathbb{Y}_+) \to
    \thmr^*(\mathbb{Y}_-).\]
\end{defn}
The following is analogous to \cite[Prop.~25.3.8]{KMbook2007}, but its proof involves less complicated identities.
\begin{prop}\label{prop:cobordism_map_independence}
    Let $g(0)$ and $g(1)$ be two $\tau_W$-invariant Riemmanian metric, containing the same cylindrical metric at the collar of the boundary.
    Let $\mathfrak{p}(0)$ and $\mathfrak{p}(1)$ be two perturbations on $\mathbb W$ constructed from the same perturbations on $\mathbb Y_{\pm}$.
    Assume the corresponding moduli spaces on $\mathbb{W}$ are both regular.
    Let $\wtilde{m}(0)$ and $\wtilde{m}(1)$ be defined in Definition~\ref{defn:cob_eval_u_map} using $(g(0),\mathfrak{p}(0))$ and $(g(1),\mathfrak{p}(1))$.
    Then there is an operator
    \[\wtilde{K} \colon C^d(\mathcal{U};\mathbf{F}) \otimes \wtilde{C}(\mathbb{Y}_-) \to \wtilde{C}(\mathbb{Y}_+)\]
    for $d \le d_0$, satisfying
    \[\tilde{\del} \wtilde{K} (u \otimes \xi)= \wtilde{K}(\delta u \otimes \xi) + \wtilde{K}(u \otimes \tilde{\del}\xi)+\wtilde{m}(0)(u \otimes \xi) - \wtilde{m}(1)(u \otimes \xi)\]
    for $u \in C^d(\mathcal{U};\mathbf{F})$ and $d < d_0$, and $\xi \in \wtilde{C}(\mathbb{Y}_-)$.
    Thus Definition~\ref{defn:cob_eval_u_map} is independent of the metric and perturbation $(g,\mathfrak{p})$.
    \hfill \qedsymbol
\end{prop}
Similarly, by following the proof of \cite[Prop.~26.1.2]{KMbook2007}, one can prove the composition law:
\begin{prop}
	\label{prop:composition-law}
    Assume $\mathbb{Y}_0,\mathbb{Y}_1,\mathbb{Y}_2$ are three marked real 3-manifolds with metrics and admissible perturbations.
    Let $\mathbb{W}_{01},\mathbb{W}_{12}$ be marked real cobordisms from $\mathbb{Y}_0$ to $\mathbb{Y}_1$, and $\mathbb{Y}_1$ to $\mathbb{Y}_2$, respectively;
    let $\mathbb{W}= \mathbb{W}_{12} \circ \mathbb{W}_{01}$.
    Consider cohomology classes
    \[u_{01} \in H^{d_{01}}(\wtilde{\boldsymbol{\mathcal{B}}}(\mathbb{W}_{01});\mathbf{F}), \quad
    u_{12} \in H^{d_{12}}(\wtilde{\boldsymbol{\mathcal{B}}}(\mathbb{W}_{12});\mathbf{F}),\]
    and their product $u = u_{12}u_{01}$ in $H^{d_{01}+d_{12}}(\wtilde{\boldsymbol{\mathcal{B}}}(\mathbb{W});\mathbf{F})$.
    Then
    \begin{align*}
        \thmr_*(u|\mathbb{W})
        &= \thmr_*(u_{12}|\mathbb{W}_{12}) \circ \thmr_*(u_{01}|\mathbb{W}_{01}),\\
        \thmr^*(u|\mathbb{W})
        &= \thmr^*(u_{01}|\mathbb{W}_{01}) \circ \thmr^*(u_{12}|\mathbb{W}_{12}).
    \end{align*}
\end{prop}
\begin{proof}
    The proof follows the standard scheme over counting boundary degeneration of a parametrized family of moduli spaces that break along the middle $\mathbb{Y}_1$.
    As in the proof of Proposition~\ref{prop:id_of_cob_u_map}, instead of introducing a matrix $\check K$ of operators between different types of critical points, the framed case requires a single operator $\tilde{K}$ and less formidable identities.
\end{proof}
Recall
\[\widetilde{B}(Y,\tau,\mathbf{p};\tfrr) \simeq \widetilde{\mathbf{T}}_{R,\mathbf{p}} =\frac{H^1(Y;\mathbf{R})^{-\tau^*}}{\Upgamma(Y,\tau,\mathbf{p})}.\]
Then $H^1(\widetilde{\mathbf{T}}_{R,\mathbf{p}};\mathbf{Z})$ is isomorphic to the dual lattice $\Upgamma^*(Y,\tau,\mathbf{p}) := \text{Hom}(\Upgamma(Y,\tau,\mathbf{p}),\mathbf{Z})$
and hence the mod-2 cohomology ring is
\[H^*(\widetilde{\mathbf{T}}_{R,\mathbf{p}};\mathbf{F}) \cong \Lambda^*(\Upgamma^*(Y,\tau,\mathbf{p}))\otimes \mathbf{F}.\]
The action of $u \in \Lambda^*(\Upgamma^*)$ by $\thmr_*(u|I \times \mathbb{Y})$ defines the module structure on $\thmr_*(\mathbb{Y})$ (and similarly on cohomology):
\begin{cor}
    The framed real monopole Floer homology $\thmr_*(Y,\tau,\mathbf{p})$ is a module over the exterior algebra
    \[\Lambda^*(\Upgamma^*(Y,\tau,\mathbf{p}))\otimes \mathbf{F}.\]
\end{cor}
\section{Closed marked real 4-manifold invariant}
This paper \emph{does not} address the issue of orientability of real Seiberg--Witten invariants.
Rather, \emph{assuming orientability}, this section introduces the \emph{integer-valued} framed version of Seiberg--Witten invariants on a closed marked real 4-manifolds.

In this section, denote the $(-\tau_X^*)$-invariant part of the Betti numbers by $b^{\ell}_{R}$ of a marked real 4-manifold $\mathbb{X}= (X,\tau_X,\mathbf{c})$ as
\begin{equation*}
    b^{\ell}_{R}(\mathbb{X}) = \dim H^{\ell}(X;\mathbf{R})^{-\tau_X^*}.
\end{equation*}
For instance, $b^+_R$ denotes the dimension of the $\tau_X^*$-skew-invariant self-dual $\mathbf{R}$-cohomology.
A \emph{real $\ell$-form} is a $(-\tau_X^*)$-invariant $\ell$-form.
Denote the space of real $\ell$-forms by $\Omega^{\ell}_R(X)$, and the space of $\tfrr$-invariant real spinors by $\Gamma(S)_R$ when the choice of real structure $\tfrr$ is unambiguous.
\subsection{The framed moduli space}
Let $\mathbb{X}$ be a closed oriented real marked 4-manifold and $(\mathfrak{s}_X,\tfrr_X)$ be a {\rrscs}.
Fix an imaginary-valued real 2-form $\omega$ on $X$.
Define the perturbed Seiberg--Witten equations as
\begin{equation*}
    \mathfrak{F}_{\omega}(A,\Phi)=\left(\frac{1}{2}\rho(F^+_{A^t}-4\omega^+) - (\Phi\Phi^*)_0, D^+_A\Phi\right),
\end{equation*}
over the framed configuration space $\tilde{\mathcal{B}}_k(X,\tau_X,\mathfrak{s}_X,\tfrr_X)$.
Let $\wtilde N(X,\tfrr_X) = \wtilde N(X,\tau_X,\mathfrak{s}_X,\tfrr_X)$ be the set of solutions to $\mathfrak{F}_{\omega}(\upgamma)=0$.
The formal dimension (see e.g.~\cite[Proposition~4.3]{ljk2022}) of the Seiberg--Witten moduli space is
	\begin{equation}
		d = \ind(D^+_A)_R + b^1_R(\mathbb X) - b^+_R(\mathbb X) - b^0_R(\mathbb{X})
		= \frac{c_1(\mathfrak{s})^2 - \sigma(X)}{8} + b^1_R(\mathbb X) - b^+_R(\mathbb X)- b^0_R(\mathbb{X}).
	\end{equation}
\begin{prop}
	Given $\mathbb{X}$ and $(\mathfrak{s}_X,\tfrr_X)$ as above.
	Assume $b^+_R(\mathbb{X}) = 0$ and $b^1_R(\mathbb{X}) = 0$ and the the index of the Dirac operator is nonnegative.
	Then for a residual set of perturbing real 2-form $\omega \in L^2_{R,k-1}(X;i\Lambda^2(X))$, the section $\mathfrak{F}_{\omega}$ is transverse to zero, and
	the framed moduli space $\tilde{N}(X,\tfrr_X)$ of framed solutions is a closed smooth compact manifold of dimension
	\begin{equation*}
		d = \ind(D^+_A)_R  = \frac{c_1(\mathfrak{s})^2 - \sigma(X)}{8}
	\end{equation*}
\end{prop}
\begin{proof}
	See e.g.~\cite[Transversality Theorem~1]{Moore1996lectures} and also the proof of~\cite[Lemmma~27.1.1]{KMbook2007}.
\end{proof}

To define the enumerative invariant for more general settings,
one can introduce framed perturbations at the markings:

Choose one auxiliary point $q_j$ on $\mathbf{c}$ in each connected component of $X$,
and for each such $q_j$, choose a small $\tau$-invariant $4$-ball $B$ at $q$ that meets $\mathbf c$ in a single arc.
The boundary $Y_j$ of $B_j$ is diffeomorphic to $S^3$ with two marked points $\mathbf{p}_j = Y_j\cap\mathbf c$.
Moreover, choose a collar neighourhood $(I\times Y_j)\subset X$ satisfying
\[
 \mathbf c\cap(I\times Y_j)=I\times\mathbf p_j,
\]
Assume all collars are metrically products.
Any real framed gauge transformation in $\mathcal G_{\mathbf c}(X)$ restricts to an element of $\mathcal G_{\mathbf p_j}(Y_j)$.
Let $f_j$ be a framed real cylinder function on $Y_j$.
Choose bump functions $\beta_j$ supported on $I \times Y_j$, and let $\mathfrak{p}_j = \grad f_j$.
View $\beta_j(t)\mathfrak{p}_j$ as a perturbation supported on $I \times Y_j$.
Define 4-dimensional perturbation
\[
	\mathfrak{p} = \sum_j\beta_j(t)\mathfrak{p}_j.
\]
This defines a framed perturbation on the 4-manifold $\mathbb{X}$.
For an imaginary-valued real form $\omega^+ \in \Omega^+$, consider
\begin{equation}\label{eq:closed_framed_equation}
 \mathfrak F_{\omega,\mathfrak p}(A,\Phi)
 =\left(\frac12\rho(F^+_{A^t}-4\omega^+)
                   -(\Phi\Phi^*)_0,\ D_A^+\Phi\right)
        +\widehat{\mathfrak p}(A,\Phi)=0.
\end{equation}
Denote its framed moduli space by $\widetilde N(X,\widetilde{\sr}_X;\omega,\mathfrak{p})$.

\begin{prop}\label{prop:closed_framed_regularity}
Given $\mathbb X$ and $(\mathfrak s_X,\tfrr_X)$ as above, fix $\omega$.
For a residual set of perturbations $\mathfrak{p}$,
the section $\mathfrak F_{\omega,\mathfrak p}$ is transverse to zero,
and the framed moduli space $\widetilde N(X,\tfrr_X;\omega,\mathfrak p)$ is a closed smooth compact manifold of dimension
\begin{equation}\label{eq:closed_framed_dimension}
 d(\mathfrak s_X,\tfrr_X) =\frac{c_1(\mathfrak s_X)^2-\sigma(X)}8+b_R^1(X)-b_R^+(X).
\end{equation}
For a smooth of invariant Riemannian metrics and perturbations parametrized by a compact manifold $P$,
there is a residual set of regular perturbations relative to $\partial P$ on which the given perturbations are already regular.
The resulting family moduli space is a compact smooth manifold of dimension $d+\dim P$,
with boundary over $\partial P$.
\end{prop}

\begin{proof}
Let $X=X_-\cup_Y X_+$, where $Y = \cup_j \, Y_j$, and $X_-$ is union of all the auxiliary 4-balls, containing the supports of all collar perturbations.
The proof of Proposition~\ref{prop:family_transversality_manifold_with_d}
applies with the cylindrical-end piece replaced by $X_+$.
Indeed, for fixed perturbations, the difference of the linearized boundary restriction maps is Fredholm,
and the differential of the universal restriction map has dense range, including at reducibles.
This implies regularity and the family version of regularity.
Compactness over $P$ follows from the uniform energy bound on the closed manifold $X$ and the compactness estimates used in the proof of Theorem~\ref{thm:family_proper}.
\end{proof}
\subsection{The determinant line bundles over framed real configuration spaces}
The material in this subsection draws heavily from~\cite[\S 3.2]{TianWang2009},
except the appearance of perturbations on $I \times Y_j$.
See \cite[\S 20.2]{KMbook2007} for a general discussion on determine line bundles.

Consider the linearized real Seiberg--Witten operator $L_{A,\Phi} = \mathcal{D}_{A,\Phi}\mathfrak{F}_{\omega, \mathfrak{p}}$ at $(A,\Phi)$ coupled with gauge-fixing:
\[
	L_{A,\Phi} \colon i\Omega^1_R(X) \oplus \Gamma(S^+)_R \to i\Omega^0_R(X,\mathbf{c}) \oplus i\Omega^+_R(X)\oplus \Gamma(S^-)_R,
\]
\begin{equation}
    \label{eqn:DFw_in_comp}
    (a, \phi) \mapsto \left(
    -d^*a + i\text{Re}\langle i\Phi, \phi\rangle, \quad
    \rho(d^+a) - (\phi\Phi^* + \Phi\phi^*)_0, \quad
	D_A^+\phi + \rho(a)\Phi\right) \ +  \mathcal{D}\widehat{\mathfrak{p}}.
\end{equation}
The family of $\mathcal{G}_{\mathbf{c}}$-equivariant Fredholm operators $\{L_{A,\Phi} : (A,\Phi) \in \mathcal{C}(\mathbb{X},\tfrr)\}$ gives rise to the determinant line bundle
\[\det L =\Lambda^{\text{max}} \ker(L_{A,\Phi}) \otimes \left(\Lambda^{\text{max}} \text{coker} (L_{A,\Phi})\right)^*\]
over the configuration space $\mathcal{C}(\mathbb{X},\tfrr)$.
This is a $\mathbf{R}$-line bundle that descends to the $\mathcal{G}_{\mathbf{c}}$-quotient, giving rise to the bundle
\[\det L\to \widetilde{B}_{\mathbf{c}}(\mathbb{X},\tfrr_X).\]
Following \cite[\S 3.2]{TianWang2009}, express $L_{A,\Phi}$ as the sum of operators $L^X + D_{A,R}^+ + \eta + \mathcal{D}\widehat{\mathfrak{p}}$, where:
\begin{itemize}[leftmargin=*]
    \item $L ^X \colon i\Omega^1_R \to i\Omega^0_{R,\mathbf{c}} \oplus i\Omega^+_R$ is the linear operator $(-d^*,d^+)$, independent of $A$ and $\Phi$;
    \item $D_{A,R}^+ \colon \Gamma(S^+)_R  \to \Gamma(S^-)_R $ is the real Dirac operator twisted by $A$;
    \item $\eta = \eta_{\Phi}$ is the zero-th order, $\Phi$-dependent operator given by the rest of the terms in \eqref{eqn:DFw_in_comp}:
    \[(a,\phi) \mapsto i\text{Re}\langle i\Phi, \phi\rangle - (\phi\Phi^* + \Phi\phi^*)_0 + \rho(a)\Phi.\]
    The operator $\eta$ satisfies $\eta_{t\Phi} = t\eta_{\Phi}$.
\end{itemize}

Let $\mathcal{A}$ denote the set of real spin\textsuperscript{c} connections.
The framed real gauge group $\mathcal{G}_{\mathbf{c}}$ acts on $\mathcal{A}$ freely, giving rise to the quotient $\widetilde{\mathcal{B}}^{\text{conn}}_{\mathbf{c}} = \mathcal{A}/\mathcal{G}_{\mathbf{c}}$.
%This is a smooth Hilbert manifold.
Consider the natural forgetful map $[A,\Phi] \mapsto [A]$, denoted as
\[p \colon \widetilde{\mathcal{B}}_{\mathbf{c}} \to \widetilde{\mathcal{B}}^{\text{conn}}_{\mathbf{c}}.\]
Since the real Dirac operators $D_{A,R}^+$ depend only on $A \in \mathcal{A}$, let $\det \ind D^+_{R} \to \widetilde{\mathcal{B}}_{\mathbf{c}}^{\text{conn}}$ be the corresponding determinant line bundle over the gauge-equivalent family of connections.
Once again, for the lack of complex structure, this is a $\mathbf{R}$-bundle that is \emph{a priori} not oriented.

The following theorem is the framed version of~\cite[Thm.~3.6]{TianWang2009}, for which the proof applies verbatim.
The key is that the zeroth order term $\eta + \mathcal{D}\widehat{\mathfrak{p}}$ can be canonically homotoped to zero via $t \mapsto t\eta + t\mathcal{D}\widehat{\mathfrak{p}}$,
whereas $\ind L^X$ is parameter-independent and can be oriented using a \emph{real homology orientation}.
That is, an orientation on the real vector space of the $\tau_X$-skew-invariant cohomologies
as follows:
\[H^0_R(X,\mathbf{c};\mathbf{R}) \oplus H^1_R(X;\mathbf{R}) \oplus H^+_R(X;\mathbf{R}).\]
($H^0_R(X,\mathbf{c};\mathbf{R})$ is always zero when $X$ is connected.)
The framed version below uses the full framed  gauge-equivalence classes, including the reducibles, instead of just the irreducible part as~\cite[Thm.~3.6]{TianWang2009}.
\begin{theorem}
    \label{thm:det_delta_p_D}
    Fix an orientation on $H^0_R(X,\mathbf{c};\mathbf{R}) \oplus H^1_R(X;\mathbf{R}) \oplus H^+_R(X;\mathbf{R})$.
    Then the line bundle $\det \ind L$ is isomorphic to $p^*(\det \ind D^+_{R})$ up to an isomorphism that is unique up to a positive continuous function.
    \[\det L \otimes p^*(\det \ind D^+_{R})  \to \widetilde{\mathcal{B}}_{\mathbf{c}}\]
    is orientable with a canonical orientation.
    \hfill \qedsymbol
\end{theorem}
\begin{rem}
    Readers comparing~\cite[Thm.~3.6]{TianWang2009} with Theorem~\ref{thm:det_delta_p_D} should note that the $\widetilde{\mathcal{B}}_{\mathbf{c}}$ typically cannot be canonically identified with $\mathcal{B}_R$ in the notation of \cite{TianWang2009}, as the framed configuration spaces are quotients of smaller gauge groups.
\end{rem}
By Theorem~\ref{thm:det_delta_p_D}, once a real homology orientation is chosen, an orientation of $\det \ind L \to \widetilde{\mathcal{B}}^{\text{conn}}_{\mathbf{c}}$, if exists, is completely determined by the determinant of the family of real Dirac operators.
This is in contrast with the ordinary case, where the Seiberg--Witten moduli spaces can be oriented by the homology orientation; the family of Dirac operators there come with canonical complex orientations.
\begin{notat}
	Suppose $\det L$ is trivializable so that there exists two possible choices of orientations.
    The set of orientations, i.e. the choices of sections, of $\det \ind D^+_{R}$ for {\rrscs} $(\mathfrak{s},\tfrr)$ over $\mathbb{X} = (X,\tau_X,\mathbf{c})$ will be written as $\Lambda_{D^+}(\mathbb{X}, \mathfrak{s}_X,\tfrr_X) $ and let $\Lambda_{D^+}(\mathbb{X})$ be the set of total choices of orientations
    \[\Lambda_{D^+}(\mathbb{X}) = \prod_{(\mathfrak{s},\tfrr) \in \underline{\text{RSpin}}^c(\mathbb{X})} \Lambda_{D^+}(\mathbb{X}, \mathfrak{s}_X,\tfrr_X).\]
\end{notat}
\subsection{Definitions of framed real Seiberg--Witten invariants}
Let $(X,\tau_X,\mathbf{c})$ be a marked 4-manifold.
Once and for all, fix a real homology orientation on $H^0_R(X,\mathbf{c};\mathbf{R}) \oplus H^1_R(X;\mathbf{R}) \oplus H^+_R(X;\mathbf{R})$.
Fix also regular perturbations $\omega \in i\Omega^2_R(\mathbb{X})$ and $\mathfrak{p}$,
with $\mathfrak{p}$ supported around interior marked points.
\begin{defn}\label{defn:closed_framed_invariant}
Given a Dirac orientation $\mathfrak{o} \in \Lambda_{D^+}(\mathbb X,\mathfrak s_X,\tfrr_X)$,
the \emph{framed real Seiberg--Witten invariant} $\trsw(\mathbb X;\mathfrak s_X,\tfrr_X;\mathfrak o)$ is the signed count of $\widetilde N(X,\tfrr_X;\omega,\mathfrak p)$ with respect to $\mathfrak o$ and the real homology orientation:
\[
 \trsw(\mathbb X;\mathfrak s_X,\tfrr_X;\mathfrak o)
   =\#_{\mathfrak o}
       \widetilde N(X,\tfrr_X;\omega,\mathfrak p)\in\mathbf Z,
\]
provided that $d(\mathfrak s_X,\tfrr_X)$ is zero.
Define $\trsw(\mathbb X;\mathfrak s_X,\tfrr_X;\mathfrak o) = 0$ if $d(\mathfrak s_X,\tfrr_X)\ne0$.
Given a choice of Dirac orientations $\mathfrak{o} \in \Lambda_{D^+}(\mathbb{X})$ in all relative real $\spinc$ structures,
the \emph{total framed real Seiberg--Witten invariant} is the sum
\[
 \trsw(\mathbb X)
 =\sum_{(\mathfrak s_X,\tfrr_X) \in \underline{\mathrm{RSpin}}^c(\mathbb{X})}
       \trsw(\mathbb X;\mathfrak s_X,\tfrr_X;\mathfrak o).
\]
\end{defn}
The total framed real Seiberg--Witten invariant is finite if there is a uniform bound $m$ for which $\|\widehat{p}(A,\Phi)\|_{L^2} \le m(\|\Phi\|_{L^2} + 1)$.
Indeed, the argument in the ordinary case shows that only finitely many underlying \text{spin\textsuperscript{c}} structures can have non-empty Seiberg--Witten moduli space of formal dimension zero.
But by the classification of $\rrsc$ structures (Proposition~\ref{prop:rrsc_classify}),
there are only finitely many compatible relative real structures for each underlying \text{spin\textsuperscript{c}} structure.

\begin{theorem}\label{thm:closed_framed_invariance}
The framed real Seiberg--Witten invariant $\trsw(\mathbb X;\mathfrak s_X,\tfrr_X;\mathfrak o) \in \mathbf{Z}$ is independent of the Riemannian metrics, the regular perturbations and the collars of interior marked points.
\end{theorem}
\begin{proof}
Assume $d=0$,
and let $U = \{I \times Y_j\}$ be a collection of collar regions.
Let $(g_i, \omega_i, \mathfrak{p}_i)$, $i = 0,1$ be two given triples of metrics, self-dual 2-forms, and regular perturbations.
Let $(g_s, \omega_s,\mathfrak p_s)$ be a smooth path of metrics and perturbations.
By Proposition~\ref{prop:closed_framed_regularity}, one can assume $(\omega_s, \mathfrak{p}_s)$ is regular, with the same end points $(g_i, \omega_i, \mathfrak{p}_i)$.
Thus the following parametrized moduli space is a compact smooth 1-manifold:
\[
 \mathcal N
 =\bigcup_{s\in[0,1]}
   \{s\}\times
   \widetilde N(\mathbb X, g_s;\omega_s,\mathfrak p_s).
\]
The real homology orientation and Dirac orientation $\mathfrak o$,
together with the standard orientation of $[0,1]$, orient $\mathcal N$ so that
\[
 \partial\mathcal N
 =
 \widetilde N(\mathbb X,g_1;\omega_1,\mathfrak p_1)
 \sqcup
 \bigl(-\widetilde N(\mathbb X, g_0;\omega_0,\mathfrak p_0)\bigr).
\]

Suppose the two collections of collars $U_0,U_1$ are disjoint.
By allowing perturbations that are sums of perturbations supported on
$U_0$ and $U_1$,
interpolat the two collar perturbations $\mathfrak{p}_0$ and $\mathfrak{p}_1$ via
$\mathfrak p_s=(1-s)\mathfrak p_0+s\mathfrak p_1$.
Using Proposition~\ref{prop:closed_framed_regularity} and variations supported on $U_0\cup U_1$,
one can assume the family $\mathfrak{p}_s$ is regular.
Next, suppose the collars $U_0,U_1$ are arbitrary.
Choose small subcollars $U_i' \subset U_i$ for which the complement $X \setminus \overline{U_0' \cup U_1'}$ contains at least one marked point for each comppnent of $\mathbf{c}$.
Let $\mathfrak p_i'$ be regular perturbations supported on $U_i'$.
Then $\#_{\mathfrak o}\widetilde N(\mathbb X;\omega_i,\mathfrak p_i)=\#_{\mathfrak o}\widetilde N(\mathbb X;\omega_i,\mathfrak p_i')$.
Let $U_*$ be a collar around marked points disjoint from both $U_i'$.
Apply the disjoint case to $(U_0',U_*)$ and $(U_*,U_1')$.
\end{proof}
\section{Relative mod two gradings in framed real monopole Floer homology}
\label{sec:grading}
\subsection{Failure of the na\"{i}ve absolute Z/2 grading in the unframed real homology}
Let $(Y,\tau,\mathbf{p})$ be a marked real 3-manifold and $C = \fix(\tau)$.
Recall the framed real monopole Floer homology is graded by the set $\mathbf{J}(\mathfrak{s},\tfrr)$.
As a corollary to Lemma~\ref{lem:loop_and_12}, the counterpart of \cite[Lem.~11.7]{ljk2022} is obtained by replacing $H^1(Y;\mathbf{Z})^{-\tau_*}$ with	$\Upgamma(Y,\tau,\mathbf{p})$:
\begin{lem}
    \label{lem:stab_of_J}
    The action of $\mathbf{Z}$ on $\mathbf{J}(\mathfrak{s},\tfrr)$ is transitive.
    The stabilizer is the image of the map
    \[
    \Upgamma(Y,\tau,\mathbf{p}) \to \mathbf{Z}, \quad
    [\sigma] \mapsto \frac{1}{2}\langle c_1(\mathfrak{s}) \cup [\sigma], Y\rangle.
    \]
    In particular, the action is free if and only if $c_1(\mathfrak{s})$ is torsion.
    \hfill \qedsymbol
\end{lem}
\begin{rem}
    The stabilizer is not always contained in $2\mathbf{Z}$, in contrast with the ordinary case~\cite[Lem.~22.3.2]{KMbook2007}.
    This causes the lack of a canonical absolute mod-two grading in unframed real monopole homology.
\end{rem}

Recall the mod-two grading in the ordinary case requires the choices of a reducible \emph{configuration} $\mathfrak{a}_0$, and for each $[\mathfrak{a}]$, a path $\upgamma$ between $[\mathfrak{a}_0]$ and $[\mathfrak{a}]$, and a 1-parameter family of perturbations $\mathfrak{p}$ joining $\mathfrak{q}$ to $0$.
This leads to
\begin{equation}
    \label{eq:unframed_mod-2_index}
    \gr^{(2)}_{\text{ord}}([\mathfrak{a}],\mathfrak{q},n)
    = \text{ind}(P_{\upgamma,\mathfrak{p}})(\text{mod 2}).
\end{equation}
Dependence on $\mathfrak{a}_0$ and $\upgamma$ is reflected in the parity of $\text{ind}(P_{\upgamma,\mathfrak{p}})$.
In the ordinary case, this parity is constant since the Dirac operator is complex.
In the real unframed case $\text{ind}(P_{\upgamma,\mathfrak{p}})$
can change by an odd number upon change of $\upgamma$ according to Lemma~\ref{lem:stab_of_J}.
Moreover, when $\mathfrak{a}_0 = (B_0,0)$ is replaced by another reference connection $\mathfrak{a}_0'=(B_0',0)$, \eqref{eq:unframed_mod-2_index} changes by the index of an operator of the form (cf.~\eqref{eq:DBt_in_grading_decomp})
\[
\frac{d}{dt} + D_{B(t)}
\]
on the finite cylinder with spectral boundary conditions, and where $B(t)$ is a path of connections from $B_0$ to $B_0'$.
Again, this operator acting on the real spinors is not complex and the parity of its index may change.
\subsection{Relative Z/2 grading in the framed case}
The following two assumptions will apply:
\begin{enumerate}[label=(\roman*),leftmargin=*]
    \item The image of $C$ is nullhomologous in the quotient 3-manifold, i.e. $
    [C/\tau] = 0 \in H_1(Y';\mathbf{Z})$ for $Y' = Y/\tau$.
    Moreover, $Y \to Y'$ is determined by an oriented Seifert surface $F$ of $C/\tau$ in $Y/\tau$.
    \item Every component of $C$ contains at least one basepoint.
\end{enumerate}
Assumption~(i) implies that $Y \setminus N(C) \to Y' \setminus N'(C/\tau)$ is given by class $w \in H^1(Y' \setminus N'(C/\tau); \mathbb{Z}/2)$,
for which $w$ admits an integral lift dual to the Seifert surface $F$.
Assumption~(ii) ensures that any framed real gauge transformation $g$ satisfies $g|_{C} \equiv 1$.

The choices of $[\upgamma]$ form a torsor over the subgroup $\Upgamma(Y,\tau,\mathbf{p})$ of $H^1(Y,\mathbf{Z})^{-\tau^*}$.
The former is a full-rank lattice of the latter.
This restricts the indices that appear in \eqref{eq:unframed_mod-2_index},
making the relative grading depends ``less on loops''.
The following proposition implies Theorem~\ref{thm:mod-2_grading_intro}.
\begin{prop}
    \label{lem:gamma_in_2H}
    If Assumptions~(i) \& (ii) hold,
    then the stabilizer of the $\mathbf{Z}$ action on $\mathbf{J}(\mathfrak{s},\tfrr)$ is contained in $2
    \mathbf{Z}$.
    Moreover, \eqref{eq:unframed_mod-2_index} is independent of $\upgamma$ for fixed endpoint and reference point.
\end{prop}
\begin{proof}
	Up to homotpy,
	assume $g = 1$ on an invariant tubular neighourhood $N(C)$ of $C$.
	Since the action of $\tau$ on $Y \setminus N(C)$ is free, $g$ can be equivariantly perturbed so that
	$(-1)$ is a regular value,
	while taking value $1$ on $\partial N(C)$.
	Now, $Z = g^{-1}(-1)$ is a close oriented surface in $Y \setminus N(C)$ that represents $\mathrm{PD}(g)$.
	Conjugation reverses the orientation of $S^1$ and hence $\tau$ reverses the orientation of $Z$.
	In particular, $\chi(Z') = \chi(Z)/2$ where $Z' := Z/\tau$.

	The spinor bundle $S|_Z$ decomposes as two eigenline bundles $E \oplus E'$ of $i\rho(\nu)$,
	where $\nu$ is a unit normal of $Z$.
	Since $\frr$ is conjugate-linear and $\tau_*\nu = -\nu$,
	both eigenlines are preserved by $\frr$.
	Therefore $E' \cong E \otimes TZ$,
	where $TZ$ is viewed as a complex line bundle determined by the metric on $Z$.
	Thus
	$c_1(\mathfrak{s})|_Z = 2c_1(E) + e(TZ)$, and
	\[ \frac{1}{2}\langle c_1(\mathfrak{s}) \cup [g], [Y]\rangle
	=  \frac{1}{2}\langle c_1(\mathfrak{s}),[Z]\rangle
	= \langle c_1(E),[Z]\rangle + \frac{\chi(Z)}{2} \equiv \chi({Z'}) \bmod 2,\]
	where the last congruence follows from $\langle c_1(E), [Z]\rangle$ being even.
	Indeed, choose a transverse section of the quotient bundle $E/\frr \to Z'$,
	which pulls back to a transverse section $s$ of $E$.
	Then the zeros of $s$ appear in pairs that are interchanged by $\tau$.

	Since $\tau$ acts freely, $w_1(TZ') = w|_{Z'}$, where $w \in H^1(Y' \setminus N'(C/\tau); \mathbb{Z}/2)$ characterizes $Y \to Y'$.
	But then
	\[
	\chi(Z') \equiv
	\langle w_2(Z'), [Z'] \rangle =
	\langle w_1(TZ')w_1(\nu_{Z'}), [Z'] \rangle =
	\langle w_1(TZ')^2,[Z']\rangle
	= \langle w^2|_{Z'}, [Z'] \rangle = 0 \bmod 2.
	\]
	The second equality uses $TZ' \oplus \nu_{Z'} = TY'|_{Z'}$ and the fact that $Y'$ is spin.
	The last term vanishes due to the assumption that $w$ admits an integral lift.
\end{proof}
\begin{rem}
While Lemma~\ref{lem:gamma_in_2H} guarantees a relative $\mathbf{Z}/2$ grading and independence of \eqref{eq:unframed_mod-2_index} on $[\upgamma]$, this index still depends on the reference reducible $\mathfrak{a}_0= (B_0,0)$, and the {\rrscs}.
\end{rem}
\bibliographystyle{alpha}
\bibliography{hmr_pointed}
\end{document}